\documentclass[11pt]{article}
\usepackage{amsmath,amssymb}
\usepackage{amsthm}
\usepackage{mathtools}
\usepackage[noend]{algorithmic}
\usepackage[ruled,vlined]{algorithm2e}
\usepackage{url}
\usepackage{fullpage}
\usepackage{makeidx}
\usepackage{enumerate}
\usepackage[top=1in, bottom=1.25in, left=1in, right=1in]{geometry}
\usepackage{graphicx,float,psfrag,epsfig,caption}
\usepackage[usenames,dvipsnames,svgnames,table]{xcolor}
\definecolor{darkgreen}{rgb}{0.0,0,0.9}
\usepackage[pagebackref,letterpaper=true,colorlinks=true,pdfpagemode=none,citecolor=OliveGreen,linkcolor=BrickRed,urlcolor=BrickRed,pdfstartview=FitH]{hyperref}
\usepackage{color}
\usepackage{xr}
\usepackage{subfig}
\usepackage{caption}
\usepackage{graphicx}
\usepackage[utf8]{inputenc}

\usepackage{scalerel,stackengine}

\stackMath
\newcommand\reallywidehat[1]{%
\savestack{\tmpbox}{\stretchto{%
  \scaleto{%
    \scalerel*[\widthof{\ensuremath{#1}}]{\kern.1pt\mathchar"0362\kern.1pt}%
    {\rule{0ex}{\textheight}}
  }{\textheight}%
}{2.4ex}}%
\stackon[-6.9pt]{#1}{\tmpbox}%
}
\usepackage[mathscr]{euscript}

\DeclareSymbolFont{rsfs}{U}{rsfs}{m}{n}
\DeclareSymbolFontAlphabet{\mathscrsfs}{rsfs}

\numberwithin{equation}{section}

\newtheoremstyle{myexample} 
    {\topsep}                    
    {\topsep}                    
    {\rm }                   
    {}                           
    {\bf }                   
    {.}                          
    {.5em}                       
    {}  

\newtheoremstyle{myremark} 
    {\topsep}                    
    {\topsep}                    
    {\rm}                        
    {}                           
    {\bf}                        
    {.}                          
    {.5em}                       
    {}  

\newtheorem{claim}{Claim}[section]
\newtheorem{lemma}[claim]{Lemma}

\newtheorem{theorem}[claim]{Theorem}

\newtheorem{corollary}[claim]{Corollary}
\newtheorem{definition}[claim]{Definition}
\newtheorem*{theorem*}{Theorem}
\newtheorem*{proposition*}{Proposition}

\theoremstyle{myremark}
\newtheorem{remark}{Remark}[section]

\theoremstyle{myremark}

\theoremstyle{myexample}

\definecolor{darkgreen}{rgb}{0.0, 0.5, 0.0}

\newcommand{\comm}[1]{\textcolor{black}{#1}}

\newcommand{\bea}{\begin{eqnarray}}
\newcommand{\eea}{\end{eqnarray}}
\newcommand{\<}{\langle}
\renewcommand{\>}{\rangle}
\newcommand{\E}{{\mathbb E}}

\def\sTV{\mbox{\tiny \rm TV}}
\def\sinit{\mbox{\tiny \rm init}}

\def\Proj{{\sf P}}
\def\Unif{{\sf Unif}}

\def\eps{{\varepsilon}}

\def\id{{\boldsymbol{I}}}

\def\supp{{\rm supp}}
\def\Cont{\mathscrsfs{C}}
\def\Lp{\mathscrsfs{L}}

\def\cuD{\mathscrsfs{D}}
\def\cuP{\mathscrsfs{P}}
\def\cuF{\mathscrsfs{F}}
\def\cuU{\mathscrsfs{U}}
\def\oDelta{\overline{\Delta}}

\def\Law{{\rm Law}}

\def\hf{\hat{f}}
\def\hrho{\hat{\rho}}

\def\bZ{{\boldsymbol{Z}}}

\def\bN{{\boldsymbol{N}}}

\def\bPhi{{\boldsymbol{\Phi}}}
\def\tbPhi{\tilde{\boldsymbol{\Phi}}}

\def\bq{{\boldsymbol{q}}}
\def\bM{{\boldsymbol{M}}}
\def\obM{\overline{\boldsymbol{M}}}
\def\bV{{\boldsymbol{V}}}
\def\obV{\overline{\boldsymbol{V}}}

\def\bg{{\boldsymbol{g}}}

\def\bzero{{\mathbf 0}}

\def\cF{{\mathcal F}}

\def\cX{{\mathcal X}}

\def\blambda{{\boldsymbol \lambda}}

\def\naturals{{\mathbb N}}
\def\reals{{\mathbb R}}

\def\normal{{\sf N}}

\def\sT{{\sf T}}

\def\bz{{\boldsymbol{z}}}
\def\bx{{\boldsymbol{x}}}
\def\ba{{\boldsymbol{a}}}
\def\bb{{\boldsymbol{b}}}
\def\tbb{\tilde{\boldsymbol{b}}}

\def\bB{\boldsymbol{B}}

\def\bK{\boldsymbol{K}}

\def\de{{\rm d}}
\def\tbX{\tilde{\boldsymbol{X}}}
\def\bX{\boldsymbol{X}}
\def\bY{\boldsymbol{Y}}
\def\bW{\boldsymbol{W}}
\def\prob{{\mathbb P}}
\def\E{{\mathbb E}}

\def\<{\langle}
\def\>{\rangle}

\def\Ball{{\sf B}}

\def\cv{{*}}
\def\by{{\boldsymbol{y}}}
\def\bw{{\boldsymbol{w}}}

\def\cP{{\mathcal P}}

\def\blambda{{\boldsymbol{\lambda}}}
\def\bD{{\boldsymbol{D}}}

\def\cD{{\cal D}}
\def\hcD{\widehat{\cal D}}

\def\bphi{{\boldsymbol{\varphi}}}
\def\obphi{\overline{\boldsymbol{\varphi}}}

\def\b0{{\boldsymbol{0}}}

\def\bfone{{\boldsymbol 1}}
\def\bF{{\boldsymbol F}}
\def\bG{{\boldsymbol G}}

\def\obw{\overline{\boldsymbol w}}
\def\hbw{\hat{\boldsymbol w}}

\def\tbx{\tilde{\boldsymbol x}}

\def\bdelta{{\boldsymbol \delta}}

\def\cD{{{\mathcal D}}}

\def\trho{\tilde{\rho}}

\def\bn{{\boldsymbol n}}
\def\cP{{\mathcal P}}
\def\cB{{\mathcal B}}

\def\ess{{\rm ess}}

\def\balpha{{\boldsymbol \alpha}}

\def\err{{\sf err}}

\def\bfzero{\boldsymbol{0}}

\def\cX{{\cal X}}
\def\bfzero{{\boldsymbol 0}}
\def\hR{\widehat{R}}
\def\hE{\widehat{\mathbb E}}

\title{Analysis of a Two-Layer Neural Network via\\
 Displacement Convexity}

\author{Adel Javanmard\thanks{Data Science and Operations Department, Marshall School of Business, University of Southern California},  \;\;\;Marco Mondelli\thanks{Department of Electrical Engineering, Stanford University}  \;\;\;and\;\;\; Andrea Montanari\thanks{Department of Electrical Engineering and
  Department of Statistics, Stanford University}}

\begin{document}
\maketitle

\begin{abstract}
Fitting a function by using linear combinations of a large number $N$ of `simple' components is 
one of the most fruitful ideas in statistical learning. This idea lies at the core of a variety of methods, from two-layer neural networks
to kernel regression, to boosting. 
In general, the resulting risk minimization problem is non-convex and  is solved by gradient descent or
its variants. Unfortunately, little is known about global convergence properties of
these approaches.

Here we consider the problem of learning a concave function $f$ on a compact convex domain  $\Omega\subset \reals^d$, using linear combinations of
`bump-like' components (neurons). The parameters to be fitted are the centers of $N$ bumps, and the resulting empirical risk minimization 
problem is highly non-convex.
We prove that, in the limit in which the number of neurons diverges, the evolution of 
gradient descent converges to a Wasserstein gradient flow in the space of probability distributions over $\Omega$.
Further, when the bump width $\delta$ tends to $0$, this gradient flow has a
limit which is a viscous porous medium equation. 
 Remarkably, the cost function optimized by this gradient flow exhibits a special property known as \emph{displacement convexity},
which implies exponential convergence  rates for $N\to\infty$, $\delta\to 0$.

Surprisingly, this asymptotic theory appears to capture well the behavior for moderate values of $\delta, N$. Explaining  this
phenomenon, and understanding the dependence on $\delta,N$ in a quantitative manner remains an outstanding challenge.

\end{abstract}



\section{Introduction}\label{sec:intro}

In supervised learning, we are given data $\{(y_j,\bx_j)\}_{j\le n}$ which are often assumed to be independent and identically distributed from a common
law $\prob$ on $\reals\times\reals^d$ (here $\bx_j\in \reals^d$ is a feature vector, and $y_j\in\reals$ is a label or response variable).
We would like to find a function $\hf:\reals^d\to\reals$ to predict the labels at new points $\bx\in\reals^d$. Throughout this paper,
we will quantify the quality of our prediction by square loss, hence we are interested in minimizing $R(\hf) = \E\{(y-\hf(\bx))^2\}$.
 
One of the most fruitful ideas in this context is to use functions that are linear combinations of simple components:
\begin{align}
\hf(\bx;\bw) = \frac{1}{N}\sum_{i=1}^Na_i\sigma(\bx;\bw_i)\, .  \label{eq:GeneralF_Composition}
\end{align}
Here $\sigma:\reals^d\times\reals^D\to\reals$ is a component function (a `neuron'  or `unit' in the neural network parlance),
and $\bw=(\bw_1,\dots,\bw_N)\in\reals^{D\times N}$, $\ba =
(a_1,\dots,a_N)\in\reals^N$ are parameters to be learnt from data. 
\comm{Standard choices for the activation function are $\sigma(\bx;\bw) = (1+\exp(-\<\bw,\bx\>))^{-1}$ (sigmoid) or $\sigma(\bx;\bw) = \max(\<\bw,\bx\>;0)$
(ReLU). In this paper we will instead study a class of activation that depends on the difference $\bx-\bw$.}
The objective is to minimize the population (prediction) risk
\begin{align}\label{eq:poprisk}
R_N(\ba,\bw) = \E\Big\{\Big[y-\frac{1}{N}\sum_{i=1}^Na_i\sigma(\bx;\bw_i)\Big]^2\Big\}\, .
\end{align}
Special instantiations of this idea include (we provide only pointers to the immense literature on each topic):
\begin{itemize}
\item Two-layer neural networks \cite{rosenblatt1962principles,anthony2009neural};
\item Sparse deconvolution \cite{donoho1992superresolution,candes2014towards};
\item Kernel ridge regression  and related random feature methods \cite{cristianini2000introduction,rahimi2008random};
\item Boosting \cite{schapire2003boosting,friedman2001greedy,buhlmann2003boosting}.
\end{itemize}
Despite the impressive practical success of these methods, the risk function $R_N(\bw)$ is highly non-convex and
little is known about global convergence of algorithms that try to minimize it (we refer to Section \ref{sec:Related} for further discussion of the 
related literature). 

Notable exceptions to the last statement are provided by random
features and by boosting algorithms. In random feature methods, the  parameters $\bw_i$ are not optimized over (they are drawn i.i.d.
from some common distribution), and the resulting risk function becomes convex in the weights $(a_1,\dots,a_N)$ to be learnt.
While this is a fruitful idea, it gives up the degrees of freedom afforded by the $\bw_i$'s.

Boosting overcomes non-convexity by fitting the components $\bw_1$, \dots, $\bw_N$ one at the time, sequentially. The underlying
assumption is that the problem of minimizing $R_N(\bw)$ with respect to one of the hidden units $\bw_i$ is tractable. However,
this is generally not the case when the parameters $\bw_i$ belong to a high-dimensional space. 

\comm{The risk function \eqref{eq:poprisk} crystalizes a central conundrum in statistical learning. In a number of applications
(especially at low noise),  it is rarely the case that low prediction error can be achieved through a function that is linear in the raw covariates, e.g.
$\hf(x) = \<\bw,\bx\>$. In a  classical setting, the statistician would craft nonlinear features out of the covariates on the basis of expert knowledge. 
For the model of Eq.~\eqref{eq:GeneralF_Composition}, this amounts to constructing vectors $\bw_1,\dots,\bw_N$. Statistical methods would then be  confined to 
the convex task of fitting the coefficients $a_1,\dots,a_N$.  This step is well understood from a statistical and computational perspective.}

\comm{Modern machine learning approaches (boosting, neural networks, etc.) hold the promise of automatizing feature extraction,
hence producing superior performances in a wide variety of applications. Unfortunately, we are still far from understanding in which cases optimizing
over the  $\bw_i$'s yields a significant improvement over --say-- choosing them randomly. 
This central challenge intertwines statistical and computational aspects. It is not hard to see that varying
the weights $\bw_i$'s produces a significantly larger function class \cite{bach2017breaking}. The relevant question is what part of this class can be accessed using gradient
descent or other practical algorithms.}

\comm{The main objective of this paper is to introduce a nonparametric regression model in which these questions can be
addressed rigorously.  
The model is interesting for at least two reasons: $(i)$ From a theoretical point of view, 
global convergence can be proved in the limit of a large neurons. The proof relies on a mathematical mechanism that has not been 
explored in the statistics or machine learning literature before. $(ii)$ From a practical point of view, the model is nontrivial enough
to illustrate the potential advantage of fitting the features $\bw_i$ (we demonstrate this numerically in Section \ref{sec:Numerical}.)}

Let $\Omega\subset\reals^d$ be a compact convex set with $\Cont^2$ boundary.
We assume $\{(y_j,\bx_j)\}_{j\ge 1}$ to be i.i.d.  where $\bx_j\sim\Unif(\Omega)$ and
\begin{align}
\E(y_j|\bx_j) = f(\bx_j) \, ,
\end{align}
with $f:\Omega\to \reals$ a smooth function. We try to fit these data using a combination of bumps, namely
\begin{align}
\hf(\bx;\bw)= \frac{1}{N}\sum_{i=1}^NK^\delta(\bx-\bw_i)\, , \label{eq:GeneralFhat}
\end{align}
where $K^\delta(\bx) = \delta^{-d}K(\bx/\delta)$, $K:\reals^d\to\reals_{\ge 0}$ is a first order kernel with compact support,
and $\bw_i\in \Omega^{\delta}$ for $i\le N$. Here $\Omega^{\delta}$ is a slightly smaller compact set, with $\Omega^{\delta}\to\Omega$
as $\delta\to 0$.  (Note that in our setting the hidden units $\bw_i$ and input data $\bx_j$ have same dimensions, i.e., $d = D$.) We refer to Section \ref{sec:Main} for a formal statement of our assumptions.
\comm{From Eq.~\eqref{eq:poprisk}, we have
\begin{align*}
R_N(\bw) = R_{\#} +\E\big\{\big[f(\bx)-\frac{1}{N}\sum_{i=1}^NK^\delta(\bx-\bw_i)\big]^2\big\}\,,
\end{align*}
where $R_{\#} = \E[(y- f(\bx))^2]$ and we use the fact that $\E[y - f(\bx)|\bx] = 0$.
Since the constant $R_{\#}$ does not depend on parameters $\bw$, it does not matter in optimizing $R_N(\bw)$ over $\bw$ and henceforth we write, with a slight abuse of notation, 
\[
R_N(\bw) = \E\big\{\big[f(\bx)-\frac{1}{N}\sum_{i=1}^NK^\delta(\bx-\bw_i)\big]^2\big\}\,.
\]}

The model (\ref{eq:GeneralFhat}) is general enough to include a broad class of radial-basis function (RBF) networks
which are known to be universal function approximators
\cite{park1991universal}. 
To the best of our knowledge, there is no result on the global convergence of stochastic gradient descent for learning RBF networks, 
and this paper establishes the first result of this type.

It is important to emphasize a few differences
with respect to standard RBF networks. First of all, we do not require the kernel $K(\bx)$ to be radial, i.e. to depend uniquely on the norm $|\bx|$. 
Second, we require $K$ to have compact support. This is mainly a technical requirement that simplifies some arguments: we expect 
our results to be generalizable to kernels that decay rapidly enough. 
Finally, and most crucially, the form (\ref{eq:GeneralFhat}) does not include non-uniform weights for the $N$ components. 
A more standard formulation would posit
$\hf(\bx;\bw)= \sum_{i=1}^Na_iK^\delta(\bx-\bw_i)$ and learn the weights $a_i$ from data,  see Eq.~\eqref{eq:GeneralF_Composition}.
We deliberately set the weights to a fixed value because the risk function is convex in  $\ba = (a_i)_{i\le N}$, 
and hence fitting $\ba$'s to global optimality is `easy.'   Indeed, universal approximation could be achieved by keeping the centers $\bw_i$
fixed (and sufficiently dense in $\Omega$) and only adjusting $\ba$. 
As discussed above, our focus is on the role of the $\bw_i$'s. 

Our main result is a proof that,  for sufficiently large $N$ and small $\delta$,  gradient descent algorithms  converge 
to weights $\bw$ with nearly optimum prediction error, provided $f$ is strongly concave.
Let us emphasize that the  resulting population risk $R_N(\bw)$ is non-convex regardless of the concavity properties of $f$.
Our proof unveils a novel mechanism by which global convergence 
takes place. Convergence results for non-convex empirical risk minimization are generally proved by carefully ruling out local minima 
in the cost function (see Section \ref{sec:Related} for pointers to this literature). Instead we prove that, as $N\to\infty$, $\delta\to 0$,
the gradient descent dynamics converges to a gradient flow in Wasserstein space, and that the corresponding cost function is `displacement convex.'
Breakthrough results in optimal transport theory guarantee dimension-free convergence rates for this limiting dynamics \cite{carrillo2001entropy,carrillo2003kinetic,carrillo2006contractions}. \comm{In particular, we expect the cost function $R_N(\bw)$ to have many local minima, which are however completely neglected by the gradient descent dynamics.}

More specifically, our first step is to show that -- for large $N$ -- the 
evolution of the weights $\bw_1,\dots,\bw_N$ under gradient descent can be replaced by the evolution of a probability 
distribution\footnote{Throughout,$\cuP_2(\cX)$ denotes the space of probability distributions on $\cX$, endowed with Wasserstein metric $W_2$.} $\rho^{\delta}\in \cuP_2(\Omega)$,
which approximates their empirical distribution.  Namely, if $(\bw^k_1,\dots,\bw^k_N)$ denote the weights after $k$ iterations with
step size $\eps$, and $\hrho^{(N)}_k = \sum_{i=1}^N\delta_{\bw_i^k}/N$ is their empirical distribution, then we have
\begin{align}
\lim_{N\to\infty,\eps\to 0}\hrho^{(N)}_{t/\eps} = \rho^{\delta}_t \, ,\label{eq:PDE-Approx}
\end{align}
where the limit holds in the sense of weak convergence \comm{or in $W_1$ distance (the two are equivalent since $\Omega$ is compact).} The limit evolution $(\rho^{\delta}_t)_{t\ge 0}$ satisfies a partial differential equation (PDE)
that can also be described as the Wasserstein $W_2$ gradient flow (i.e. gradient flow in $\cuP_2(\Omega)$), 
for the following effective  risk
\begin{align}
R^{\delta}(\rho) &= \nu_0 \, \int_{\Omega} \big[f(\bx) - K^\delta\ast \rho(\bx)\big]^2\de \bx\, , \label{eq:Rdelta-first}
\end{align}
where $\nu_0 = 1/|\Omega|$ and $|\Omega|$ denotes the volume of the set $\Omega$.
Here $\ast$ denotes the usual convolution. 
\comm{Let us emphasize that the convergence to Wasserstein gradient flow holds regardless of the concavity of $f$.}

The use of $W_2$ gradient flows to analyze two-layer neural networks was recently developed in several 
papers \cite{mei2018mean,rotskoff2018neural,chizat2018global,sirignano2018mean}. 
However, we cannot rely on earlier results because of the specific boundary conditions in our problem. \comm{We constrain the $\bw_i\in\Omega^{\delta}$ by
running projected stochastic gradient descent (SGD): at each step $\bw_i$ moves in the direction of a stochastic gradient of $R_N(\bw)$ and then projected back
to $\Omega^{\delta}$. This results in a PDE with Neumann boundary condition on $\Omega^{\delta}$, which is not covered by previous theory.}
We establish a quantitative version of the limit (\ref{eq:PDE-Approx}) via propagation-of-chaos techniques.

Even if the cost  (\ref{eq:Rdelta-first}) is quadratic and convex in $\rho$, its $W_2$ gradient flow can have multiple fixed points,
and hence global convergence cannot be guaranteed. Global convergence results were proven in \cite{mei2018mean} and in \cite{chizat2018global}
by showing that, for all $t\ge 0$ $\rho^{\delta}_t$ has a density that is either smooth, or strictly positive everywhere. However,
these convergence results are non-quantitative, and do not provide convergence rates\footnote{An argument indicating convergence 
in a time polynomial in $d$ was put forward in \cite{wei2018margin},
but for a different type of continuous flow.}.

Indeed, the mathematical property that controls global convergence of $W_2$ gradient flow is not ordinary convexity but \emph{displacement convexity}.
Roughly speaking, displacement convexity is convexity along geodesics of the $W_2$ metric, see Section \ref{subsec:PDEdelta0}.
\comm{The risk function (\ref{eq:Rdelta-first}) is not displacement convex. Indeed,  its quadratic term reads 
$\nu_0\int K_{\delta}\ast K_{\delta}(\bx-\bx')\rho(\bx)\rho(\bx') \de\bx\de\bx'$ which is not displacement convex unless $K_{\delta}\ast K_{\delta}$ is 
convex (see Lemma \ref{lem:DisplCvx}),
which cannot be in our setting}. However, for small $\delta$,
we can formally approximate $K^\delta\ast \rho\approx \rho$, and hence hope to replace  the risk function
(\ref{eq:Rdelta-first}) with a simpler one
\begin{align}
R(\rho) &= \nu_0\int_{\Omega} \big[f(\bx) - \rho(\bx)\big]^2\de \bx\, .
\label{eq:R-delta-0}
\end{align}
Most of our technical work is devoted to making rigorous this $\delta\to 0$ approximation. Namely, we prove that,
as $\delta\to 0$,  $\rho^{\delta}_t\Rightarrow \rho_t$ where $\rho_t$ follows the $W_2$ gradient flow for the risk $R(\rho)$. 

Remarkably, the risk function $R(\rho)$ is strongly displacement convex (provided $f$ is strongly concave). 
A long line of work in PDE and optimal transport theory establishes dimension-free convergence rates for its $W_2$ gradient flow 
\cite{carrillo2001entropy,carrillo2003kinetic,carrillo2006contractions}.
Namely, if $f$ is $\alpha$-strongly concave, then $R(\rho_t) \le R(\rho_0)\, e^{-2\alpha t}$. By using the approximation results outlined above,
we obtain global convergence for SGD. With high probability, 
\begin{align}
R_N(\bw^{k}) \le R_N(\bw^0)\, e^{-2\alpha k\eps} +\, \err(N,d,\eps,\delta)\, ,\label{eq:RoughApproximation}
\end{align}
where the error term $\err$ vanishes as $N\to\infty$, $\eps,\delta\to 0$ in a suitable order. 

This result implies that SGD converges exponentially fast to a near-global optimum with
a rate that is controlled by the convexity parameter $\alpha$. 

Our bounds are not sharp enough to provide quantitative control on the error term  $\err(N,d,\eps,\delta)$,
especially in high dimension. Nevertheless, the convergence rate predicted by our asymptotic theory
is in excellent agreement with numerical simulations, cf. Section \ref{sec:Numerical}.
Explaining this surprising quantitative agreement is an outstanding challenge.

\section{Related literature}
\label{sec:Related}

The present work ties in several lines of research, some of which were already mentioned in the introduction. 
A substantial amount of work has been devoted to analyzing two-layer neural networks and developing algorithms with convergence guarantees,
see e.g. \cite{zhong2017recovery,tian2017symmetry,bakshi2018learning}. However these approaches are typically based on tensor factorization
or similar initialization steps that are not used in practice, and do not scale well (although polynomially) in high dimension.

The landscape of empirical risk minimization was also studied in a number of papers, see e.g. \cite{li2017convergence,soltanolkotabi2018theoretical}.
However, global convergence was only proved in the extremely overparametrized regime in which the neural network essentially behaves as kernel
ridge regression \cite{du2018gradient}.

Classical theory of neural networks was largely devoted to the two-layer case \cite{anthony2009neural}, although the focus was on representation
and approximation  questions \cite{cybenko1989approximation,barron1993universal}, as well as on generalization error.
It was already clear in that context that a two-layer network is conveniently characterized by the empirical distribution of the hidden neurons, and
that it is useful to relax this from a distribution with $N$ atoms, to a general probability measure. This representation plays an important
role, for instance, in \cite{bartlett1998sample}, and was exploited
again under the label of `convex neural networks' in \cite{bengio2006convex}.

Over the last year, several groups independently revisited this connection, with the objective of understanding the landscape structure of two-layer networks, and
the dynamics of gradient descent methods \cite{nitanda2017stochastic,mei2018mean,rotskoff2018neural,sirignano2018mean,chizat2018global,mei2019mean}.  
In particular, it was proven in
\cite{mei2018mean} that, under certain smoothness condition on the underlying data distribution, the gradient descent evolution is
well approximated by a Wasserstein gradient flow, provided that the number of neurons exceeds the data dimensions.
\comm{As mentioned above, the algorithm treated here differs from the ones analyzed in earlier work, because the weights $\bw_i$ are constrained to
lie in the convex set $\Omega^{\delta}$. We enforce this constraint by using projected SGD, i.e. projecting at each step the weights onto the set $\Omega^{\delta}$.
We generalize the analysis of \cite{mei2018mean}, obtaining convergence to a PDE with Neumann (reflecting) boundary conditions. As in
\cite{mei2018mean}, we build on ideas that were first developed in the context of interacting particle systems \cite{dobrushin1979vlasov,sznitman1991topics}.}

The Wasserstein gradient flow  approach was used in \cite{mei2018mean,chizat2018global} to establish global convergence results. However, these results fall short of our 
objectives for several  reasons:
\begin{itemize}
\item The global convergence result of \cite{chizat2018global} rely on certain homogeneity properties of the neurons that are lacking here.
We could obtain homogeneity by adding coefficients to Eq.~(\ref{eq:GeneralFhat}), i.e. considering $\hf(\bx;\bw)= \sum_{i=1}^Na_iK^\delta(\bx-\bw_i)$ and minimizing
the risk with respect to the coefficients $a_i$. 
As mentioned above, we refrain from introducing coefficients not to oversimplify the problem:
when $N\to\infty$, it is sufficient to fit the coefficients $a_i$ to achieve vanishing risk. Fitting the $a_i$'s is a least squares problem.
\item Most importantly, the techniques \cite{mei2018mean,chizat2018global} do not establish any convergence rates.
This is not surprising, as those results hold under weak assumptions on the data distribution and the activation function. 
\comm{In particular, \cite{mei2018mean,chizat2018global,mei2019mean}  cover general risk functions of the form 
\eqref{eq:poprisk} under certain smoothness and boundedness conditions on $\sigma$ and on the functions 
$V(\bw) = -\E\{f(\bx)\sigma(\bx;\bw)\}$, $U(\bw_1,\bw_2) = \E\{\sigma(\bx;\bw_1)
 \sigma(\bx;\bw_2)\}$. In such a general setting \cite{mei2018mean} provides examples in which the Wasserstein gradient flow has multiple fixed points,
which are singular with respect to the Lebesgue measure. Global convergence is established in \cite{mei2018mean,chizat2018global} by proving
that PDE solution $\rho_t$ has a strictly positive density. However, it is difficult to imagine this condition to hold in a quantitative dimension-independent 
manner.}

\comm{In contrast, our results are a first step towards dimension-independent convergence rate, in a more restricted setting than 
\cite{mei2018mean,chizat2018global,mei2019mean}.}
\end{itemize}
\comm{In summary, our results do not subsume earlier work, that assumes a more general setting, but rather establish stronger results in 
narrower context. 
Indeed, we believe that specific structural conditions must be imposed on the data distribution and activation function
for the Wasserstein gradient flow approach to yield quantitative convergence rates. This paper presents one
specific set of assumptions. Although our results are not strong enough to establish non-asymptotic convergence rates, they point clearly in that direction.}

\section{Model and assumptions}
\label{sec:Model}
\subsection{Notations}
We will use lowercase boldface for vectors, e.g. $\bx,\by,\dots $,  uppercase for random variables, e.g. $X,Y,\dots$, and uppercase
boldface for random vectors, e.g. $\bX,\bY,\dots$. The scalar product of two vectors is denoted by $\<\bx,\by\> =\sum_{i=1}^d x_i y_i$,
and the $\ell_2$ norm of a vector is denoted by  $|\bx|$. 
The Euclidean ball in $\reals^d$ with center $\bx$ and radius $r$ is denoted by $\Ball(\bx;r)$.
Given a set $\Omega\subseteq\reals^d$, we denote by $|\Omega|$ its
volume.

We will refer to several function spaces in what follows. The most
common is the space of $p$-th integrable functions  $\Lp^p(\cX)$ on a measure space   $(\cX,\cF,\mu)$.
Given a function $f:{\cal X}\to\reals$, we denote by $\|f\|_{\Lp^p(\cX)}$ its $\Lp^p$ norm, namely $\|f\|_{\Lp^p(\cX)}^p = \int_{\cX} |f(x)|^p\, \mu(\de x)$. 
For $S\subseteq\reals^m$, $\Cont^k(S)$ denotes the space of continuous functions $f:S\to\reals$ with continuous derivatives up to order $k$. In particular, $\Cont(S)$ denotes the space of continuous real-valued functions defined on $S$. In addition, for $T\in \reals_{+}$ and a metric space $\mathcal{M}$ (with distance $d_{\mathcal{M}}$), $\Cont([0,T], \mathcal{M})$ denotes the set of continuous functions $f:[0,T]\to \mathcal{M}$, endowed with the distance between two functions $f,g\in \Cont([0,T],\mathcal{M})$  
defined as $d_{\Cont([0,T],\mathcal{M})}(f,g) \equiv \sup_{t\in[0,T]} d_{\mathcal{M}}(f(t),g(t))$. For a function $f:S\to \reals$, we let $\|f\|_{\rm Lip}\equiv \sup_{\bx\neq\by \in S} |f(\bx)-f(\by)|/|\bx-\by|$ be the Lipschitz constant of the function $f$.
Finally, as mentioned above, $\cuP_2(\cX)$ denotes the space of probability distributions on $\cX$, endowed with the Wasserstein metric $W_2$

Throughout the paper, we use $C$ to denote finite constants, which can vary from point to point. When these constants can depend on
some  of the problem parameters, e.g.  $a,b,c$, we will write $C(a,b,c)$. When they are absolute numerical constants, we will emphasize this by 
writing $C_*$. 

\subsection{Data}

As mentioned above, we are given data $(y_j,\bx_j)\sim_{\rm i.i.d.}\prob$ where $\bx_j\sim\Unif(\Omega)$, \comm{with $\Omega\subset \reals^d$ a compact convex set,}
and 
$y_j = f(\bx_j) +\eps_j$, with $f:\Omega\to\reals_{\ge 0}$.
We assume the $\eps_j$ to be i.i.d. $\sigma^2$-subgaussian random variables with $\E(\eps_j|\bx_j) = 0$. 
We assume the function $f$ to be  concave and smooth. 

Our formal assumptions on the set $\Omega$ and the function $f$ are as follows:
\begin{enumerate}
\item[\sf (A1)] $\Omega\supseteq\Ball(\b0;r)$, with $r>0$, is a compact convex set with $\Cont^2$ boundary. 

\item[\sf (A2)] $f : \Omega \to \reals_{\ge 0}$ uniformly concave, i.e., there exists $\alpha>0$ such that
\begin{equation}\label{eq:concavityf}
\<\by , \nabla^2 f(\bx)\by \> \le -\alpha |\by|^2, \qquad \,\forall \bx\in \Omega, \; \by\in\reals^d\, ,
\end{equation}
where $\nabla^2 f$ denotes the Hessian of $f$.

\item[\sf (A3)] $f\in \Cont^{\infty}(\Omega)$, with $\|f\|_{\Lp^\infty(\Omega)}, \|\nabla f\|_{\Lp^\infty(\Omega)}\le C_*$ for an absolute constant $C_*$.
\end{enumerate}
Without loss of generality, we can also assume that $\int_{\Omega}f(\bx)\, \de\bx = 1$. As a running example, we will use $\Omega=\Ball(\b0;r)$, \comm{where we remind $r$ is defined in Assumption {{\sf (A1)}}}.
\comm{\begin{remark}
The assumption  $\bx_j\sim\Unif(\Omega)$ is quite strong but simplifies our analysis. We believe our approach can be generalized to a broader family of probability
distribution for the covariates $\bx_j$, but defer these generalizations to future work. 
\end{remark}}

\subsection{Neural network and SGD}

Let $K\in \Cont^2(\reals^d)$ be a non-negative symmetric first order kernel with compact support. Formally, we assume that
\begin{align}
{\sf (A4)} \quad \qquad
&\int K(\bx)\, \de\bx =1 \, ,\;\;\;\;   K(\bx)\ge 0 ,\;\;\;\; \int K(\bx)\, \bx \, \de\bx =0 ,\label{eq:defK}\\
&K(-\bx) = K(\bx)\, ,\;\;\;\;\;\; \supp(K) \subseteq \Ball(\bzero,c_0)\, .
\end{align}
The assumptions of symmetry and compact support are not crucial, but simplify some of the technical details later.
We will further assume $\|\nabla K\|_{\Lp^\infty(\mathbb R^d)}$, $\|\nabla^2 K\|_{\Lp^\infty(\mathbb R^d)}$ and $c_0$ to be independent of the ambient dimension $d$. Notice that this requirement 
follows from the differentiability and compact support assumptions if 
$K(\bx) =\kappa(\|\bx\|_2)$ is a radial function.

For $\delta>0$, let $K^\delta(\bx) = \delta^{-d} K(\bx/\delta)$.
We try to fit the function (\ref{eq:GeneralFhat})  with parameters $\bw = (\bw_1,\dots,\bw_N)$.
These parameters are constrained to $\bw_i\in\Omega^\delta$ which is a suitable scaling of $\Omega$, as defined in the following. Given $\delta < r/c_0$, with $r$ defined in {\sf (A1)}, define
$$\Omega^\delta = \lambda_\delta \,\Omega,$$
where
\begin{equation}\label{eq:lambda}
\lambda_{\delta}  =\sup \big\{\lambda\ge 0 :\; \lambda \Omega \oplus \Ball(\b0, c_0 \, \delta)\subseteq \Omega\big\}\, .
\end{equation} 
For two sets $A,B\subseteq\reals^d$, their Minkowski sum is defined as $A\oplus B=\{\bx+\by : \bx\in A, \by \in B \}$. Note that $\lambda_{\delta}\in [0,1]$ for all $\delta$. Furthermore, $\Omega\supseteq\Ball(\b0;r)$ implies $\lambda_\delta >0$ for all $\delta<r/c_0$. 
Finally, $\lambda_{\delta=0}=1$, whence $\Omega^{\delta=0}=\Omega$. In our running example, $\Omega^\delta = \Ball(\bzero;r-c_0\delta)$ is a ball of slightly smaller radius. 
Clearly, since $\Omega$ is convex, $\Omega^\delta$ is convex as well.

We use stochastic gradient descent to minimize the population risk (\ref{eq:poprisk}). \comm{At each step, we use a new data point $(y_k,\bx_k)$, thus the sample size is equal to the number of iterations of the algorithm.} Assuming for simplicity constant step size $\eps>0$, we update the parameters by
\begin{align}\label{eq:SGDup}
\bw_i^{k+1} = \Proj\left\{\bw_i^k-\eps\nabla K^\delta(\bx_{k+1}-\bw^k_i) \, \left(y_{k+1}-\hf(\bx_{k+1};\bw^k)\right) 
+\sqrt{2\eps\tau}\, \bg_i^{k+1}\right\}\, .
\end{align}
Here $\bg_i^{k+1}\sim\normal(0,\id_d)$ is Gaussian noise which we take to be i.i.d. across time and neuron indices, $k$ and $i$,
and $\Proj$ is the orthogonal  projector onto $\Omega^\delta$:
\begin{align}
\Proj(\bz) = \arg\min\big\{ |\bz-\bx|:\; \;\;\bx\in\Omega^\delta\big\} \, .
\end{align}
The noise term $\sqrt{2\eps\tau}\, \bg_i^{k+1}$ is added mainly for technical reasons. Namely, it allows us to control the smoothness
of the solutions of the resulting PDE. In simulations we do not find it useful, and we believe that a more careful analysis would be able to establish smoothness without
the noise term.

Again, in our running example, we have 
\begin{align}
\Proj(\bz) = \begin{cases}
\bz & \mbox{ if $|\bz|\le r-c_0\delta$,}\\
(r-c_0\delta)\bz/|\bz| & \mbox{ if $|\bz|> r-c_0\delta$.}
\end{cases}
\end{align}
We initialize SGD with $(\bw_i^0)_{i\le N}\sim_{\rm i.i.d.}\rho^{\delta}_{\sinit}\in \cuP_2(\Omega^\delta)$, where 
$\rho^{\delta}_{\sinit}$ is a scaling of a fixed distribution $\rho_{\sinit}\in \cuP_2(\Omega)$, i.e. $\rho^{\delta}_{\sinit}(S) = \rho_{\sinit}(S/\lambda_{\delta})$.
We assume that the initialization is smooth:
\begin{enumerate}
\item[\sf (A5)] $\rho_{\sinit}\in \Cont^\infty(\Omega^\delta)$.
\end{enumerate}

\subsection{PDE Model, $\delta>0$}\label{subsec:PDEdeltag0}

In the $N\to\infty$ limit the population risk is approximated  by
the effective risk $R^{\delta}:\cuP_2(\Omega^\delta)\to \reals$ defined in Eq.~\eqref{eq:Rdelta-first}.
We emphasize that $\rho$ is a probability distribution supported on $\Omega^\delta$.
Note that
\begin{align}
\inf_{\rho}R^{\delta}(\rho) \le R^{\delta}(f) = \nu_0\int_{\Omega}\big[f(\bx)-K^{\delta}*f(\bx)\big]^2 \de\bx\, .
\end{align}
In particular $\lim_{\delta\to 0}\inf_{\rho\in\cuP_2(\Omega)} R^{\delta}(\rho) = 0$.

Our first main result is that the dynamics of SGD is well approximated by the following PDE
(see Section \ref{subsec:convPDESGD} for a formal statement):
\begin{equation}\label{EQ:PDEDELTA}
\begin{split}
\partial_t \rho_t(\bw) &= \nabla\cdot\big(\rho_t(\bw)  \nabla\Psi(\bw;\rho_t)\big) +\tau\Delta\rho_t(\bw)\, ,\\
\Psi(\bw;\rho) & \equiv -\nu_0 \, K^{\delta}\ast f(\bw) + \nu_0\, K^{\delta}\ast K^{\delta}\ast\rho(\bw)\, ,
\end{split}
\end{equation}
with initial and boundary conditions
\begin{equation}\label{eq:PDEdelta-BD}
\begin{split}
&\rho_0 = \rho^{\delta}_{\sinit},\\
 \<\bn(\bw), \rho_t(\bw)\nabla\Psi(\bw;\rho_t)&+\tau\nabla\rho_t(\bw)\> = 0\, \;\;\;\;\; \forall \bw\in\partial\Omega^\delta\, ,
\end{split}
\end{equation}
where $\bn(\bx)$ denotes the inward normal vector to $\partial\Omega^\delta$ at $\bx$.

A rigorous definition of solutions of this PDE, along with some of their properties, is given in Appendix \ref{app:genPDE}. 
In Appendix \ref{app:nonlinear}, we discuss the connection between the PDE  \eqref{EQ:PDEDELTA} and the 
so-called ``nonlinear dynamics'', i.e. a stochastic differential equation that captures the trajectories of the weights $\bw_i^k$. 
Using this connection, we prove existence and uniqueness of weak solutions of Eq.~(\ref{EQ:PDEDELTA}).   In the proofs, we will often assume 
$\nu_0=1$, which amounts to a rescaling of time $t$.

For $\tau=0$, the evolution defined by Eq.~(\ref{EQ:PDEDELTA}) corresponds to the gradient flow in Wasserstein metric 
for the risk function $R^{\delta}(\rho)$. For $\tau>0$, it is the gradient flow for the free energy functional 
$F^{\delta}(\rho)$ defined below
\begin{align}
F^{\delta}(\rho) = \frac{1}{2}R^{\delta}(\rho)-\tau \,S(\rho)\, ,\;\;\; S(\rho) = -\int \rho(\bw)\log\rho(\bw) \, \de\bw\,  .
\end{align}

\subsection{Limit PDE, $\delta=0$}\label{subsec:PDEdelta0}

As mentioned above, in the limit $\delta\to 0$ the risk function
$R^{\delta}(\rho)$ is well approximated by $R: \Lp^2(\Omega)\to\reals$, where $R(\rho) = \nu_0 \|f-\rho\|_{\Lp^2(\Omega)}^2$, cf. Eq.~(\ref{eq:R-delta-0}).

The corresponding Wasserstein gradient flow is also known as \emph{viscous porous medium equation} \cite{vazquez2007porous}
and it is given by
\begin{align}
\partial_t \rho_t(\bw) &= -\nu_0\nabla\cdot\big(\rho_t(\bw)  \nabla f(\bw)\big) +
\frac{\nu_0}{2}\Delta(\rho_t^2(\bw)) +\tau\Delta \rho_t(\bw)\, ,\label{eq:PME}
\end{align}
with initial and boundary conditions
\begin{equation}\label{eq:PME-BC}
\begin{split}
\rho_0&=\rho_{\sinit},\\
\<\bn(\bw), \nu_0\rho_t(\bw)\, \nabla (f(\bw)-\rho_t(\bw))& - \tau\nabla\rho_t(\bw)\> = 0\, \;\;\;\;\; \forall \bw\in\partial\Omega\, .\\
\end{split}
\end{equation}
In Appendix \ref{app:PDEdelta0}, we give the definition of a weak solution for the PDE \eqref{eq:PME} with initial and boundary conditions \eqref{eq:PME-BC}. We also prove that the weak solution of the PDE \eqref{eq:PME} is unique, under a mild integrability condition. Again, in proofs we will assume without loss
of generality $\nu_0=1$.

As in the $\delta>0$ case, the evolution defined by Eq.~(\ref{eq:PME}) is the gradient flow for the free energy $F(\rho) = (1/2)R(\rho)-\tau S(\rho)$.
Our analysis uses a key property of the risk function $R(\rho) = \nu_0 \|f-\rho\|_{\Lp^2(\Omega)}^2$ (and the free energy): displacement convexity \cite{mccann1997convexity}.
For the reader's convenience, we recall its definition here, 
referring to \cite{ambrosio2008gradient,villani2008optimal,santambrogio2015optimal} for further background.
Given two probability measures  $\rho_0,\rho_1\in\cuP_2(\Omega)$, their $W_2$ distance is defined by
\begin{align}
W_2(\rho_0,\rho_1)^2 = \inf_{\gamma\in \Gamma(\rho_0,\rho_1)}\int \|\bx-\by\|_2^2\, \gamma(\de\bx,\de\by)\, ,
\end{align}
where the infimum is taken over the set $\Gamma(\rho_0,\rho_1)$ of couplings of $\rho_0$, $\rho_1$ (i.e. probability  measures on 
$\Omega\times\Omega$ whose first marginal coincides with $\rho_0$, and second with $\rho_1$). 
The infimum is achieved by weak compactness of $\cuP_2(\Omega)$.

The metric space $(\cuP_2(\Omega),W_2)$
is a `length space,' and  in particular it is possible to construct geodesics, i.e. paths of minimum length connecting any two probability measures 
$\rho_0,\rho_1$.  Geodesics have a simple description. Let $\gamma_*$ be the coupling achieving the infimum in the definition of $W_2(\rho_0,\rho_1)$.
Letting $(\bX_0,\bX_1)\sim \gamma_*$,  we define $\rho_t$ to be the distribution of $\bX_t = (1-t)\bX_0+t\bX_1$.
The curve $t\mapsto \rho_t$, indexed by $t\in [0,1]$ turns out to be the geodesic between $\rho_0$ and $\rho_1$ in $(\cuP_2(\Omega),W_2)$.

Displacement convexity is convexity along geodesics. Namely, a function $\cF:\cuP_2(\Omega) \to \reals$ is $\lambda$-strongly displacement convex
if
\begin{align}\label{eq:strong-dis}
(1-t)\cF(\rho_0)+t\, \cF(\rho_1) - \cF(\rho_t) \ge \frac{1}{2}\lambda\, t(1-t) W_2(\rho_0,\rho_1)^2\, .
\end{align}
%


A useful observation is that displacement convexity implies that all local minima of $\cF$ are global minimizer. Indeed, by~~\eqref{eq:strong-dis} it is straightforward to see that $\cF$ has at most one global minimizer $\rho^*$. Also, for every other point $\rho$, the geodesic between $\rho$ and $\rho_*$ is a strictly decreasing path for the function $\cF$. Now, suppose that $\bar{\rho}\neq \rho_*$ is a local minimum. Then, there exists a neighborhood $U$ around $\bar{\rho}$ such that, for any $\rho \in U$, $\cF(\rho) \ge \cF(\bar{\rho})$. However, the strictly decreasing path between $\bar{\rho}$ and $\rho_*$ passes through the neighborhood $U$, which leads to a contradiction and so $\rho = \rho_*$ 

It follows from \cite{mccann1997convexity} that the risk function $R(\rho)$ and the free energy $F(\rho)$ are strongly displacement convex. 

\begin{remark}
\comm{The concavity assumption on the regression function $f$ (Assumption {\sf(A2)}) defines a nonparametric class under which 
global convergence can be established, with convergence rates uniquely determined by the curvature  $\alpha$
(in the limit $N\to\infty$, $\delta\to 0$). Nonparametric estimation of concave functions has attracted considerable attention over 
recent years, see e.g. \cite{hannah2013multivariate,chen2016generalized}, and is --by itself-- an interesting domain of applicability.}

\comm{However, our projected SGD algorithm is potentially applicable to any data set, and will return a meaningful estimate $\hf$ regardless whether
$f$ is concave or not. Indeed, in the next section we present numerical simulations indicating convergence to a near-global optimum even for non-concave functions $f$.}

\comm{From mathematical point of view,  Assumption ({\sf A2}) is only used to show the convergence of the solution of the viscous porous medium equation 
(limit PDE, $\delta = 0$) to the unique global minimizer of the free energy 
$F(\rho) = (1/2)R(\rho) - \tau S(\rho)$, as formally stated in Theorem \ref{th:convminfree}. Concavity is not needed for the other results in the paper, 
namely approximating the SGD trajectory with the solution of the PDE ($\delta>0$), see Theorem \ref{TH:CONVPDE}, 
and the convergence of the solution of the PDE ($\delta>0$) to the solution
 of the viscous porous medium equation, see Theorem \ref{TH:CONVPME}. It is therefore foreseeable a more general analysis that relaxes the concavity assumption.}
\end{remark}

\section{Numerical illustrations}
\label{sec:Numerical}

In this section we provide some simple numerical illustrations of our setting, and compare numerical results with the predictions of the
Wasserstein gradient flow theory.

\comm{It is  easy to construct examples of strongly concave functions, satisfying our assumptions. One can start from any strongly concave continuous function
$f_0$ on a compact convex set $\Omega$,  add a constant to make it non-negative, and multiply it by a constant to normalize its integral. The resulting 
function $f(\bx) = (c_1+f_0(\bx))/c_2$ satisfies our conditions. Notable examples of concave functions are given by log-moment generating functions
$f_0(\bx) = -\log \E_{\bZ}\exp\{\<\bx,\bZ\>\}$, where the random variable $\bZ$ satisfies mild assumptions (e.g., it is bounded and its distribution is not supported 
on a proper subspace of $\reals^d$). In general, given any twice differentiable function
$g_0$, the function $f_0(\bx) = g_0(\bx)-c_*\|\bx\|^2_2$ is strongly concave for $c_*$ large enough.}

\subsection{A  one-dimensional concave function} 
\label{sec:OnedConcave}

We set $\Omega = [-1, 1]$ and $f(x) = (1-e^{x-1})/(1-e^{-2})$ (we choose the normalization so that $\int_{-1}^1f(x) \de x=1$). Note that $f$ is uniformly concave in $[-1, 1]$. We set the kernel $K$ as follows:
\begin{align}\label{eq:Kex}
K(x) = C_d\kappa(|x|)\, ,\;\;\;  \kappa(t) =\begin{cases}
1-t^2-2t^3+2t^4&\mbox{ for $t\le c_0=1$,}\\
0  &\mbox{ otherwise,}
\end{cases}
\end{align}
where $C_d$ is a normalization constant ensuring that $\int_{-1}^1 K(x)\,\de x = 1$. The initialization $\rho_{\sinit}$ is a 
truncated  Gaussian: $\rho_{\sinit}(x) = c\cdot\exp(-x^2/(2\sigma^2))\, \bfone_{[-1,1]}(x)$, with $\sigma=1/3$.

We find empirically that standard stochastic gradient descent (SGD) without the projection $\Proj$ onto $\Omega^{\delta}$ works well in this example,
and consider this algorithm for simplicity in our first illustrations.
We pick $N=200$, $\tau=0$ (noiseless SGD), and constant step size $\varepsilon = 10^{-6}$.
 In Figure \ref{fig:SGD}, left column, we plot the true function $f(\,\cdot\,)$  together with the neural network estimate $\hf(\,\cdot\, ;\bw^k)$ at several points in time
 $t$ (time is related to the number of iterations $k$ via $t=k\eps$). Different plots correspond to different values of $\delta$ with $\delta\in\{1/5, 1/10, 1/20\}$. 
We observe that the network estimates $\hf(\,\cdot\, ;\bw^k)$ seem to converge to a limit curve which
is an approximation of the true function $f$. As expected, the quality of the approximation improves as $\delta$ gets smaller.

\begin{figure}[p]
    \centering
    \subfloat[Function $f$ and SGD estimates, $\delta=1/5$.]{\includegraphics[width=.5\columnwidth]{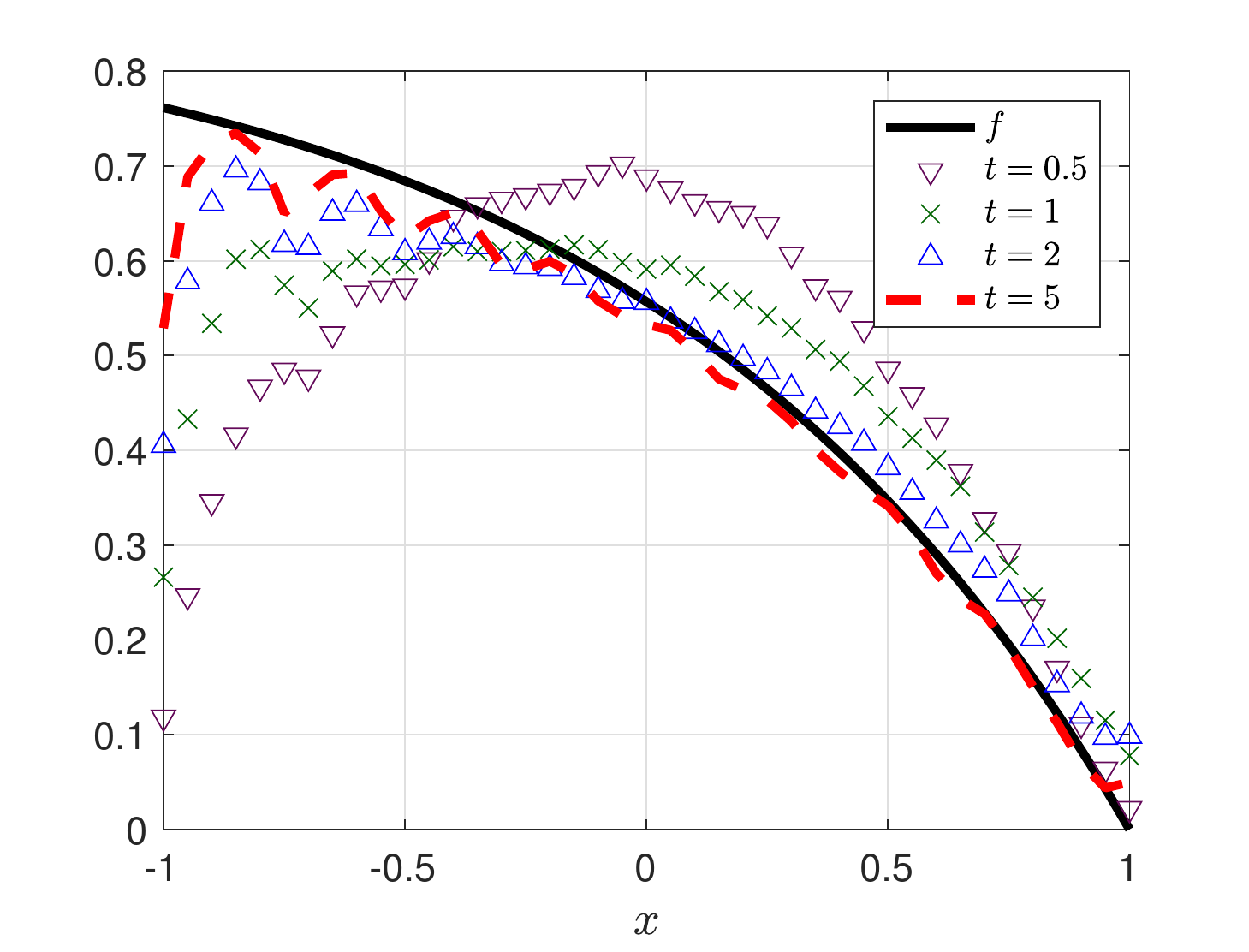}}
    \subfloat[Normalized risk, $\delta=1/5$.]{\includegraphics[width=.5\columnwidth]{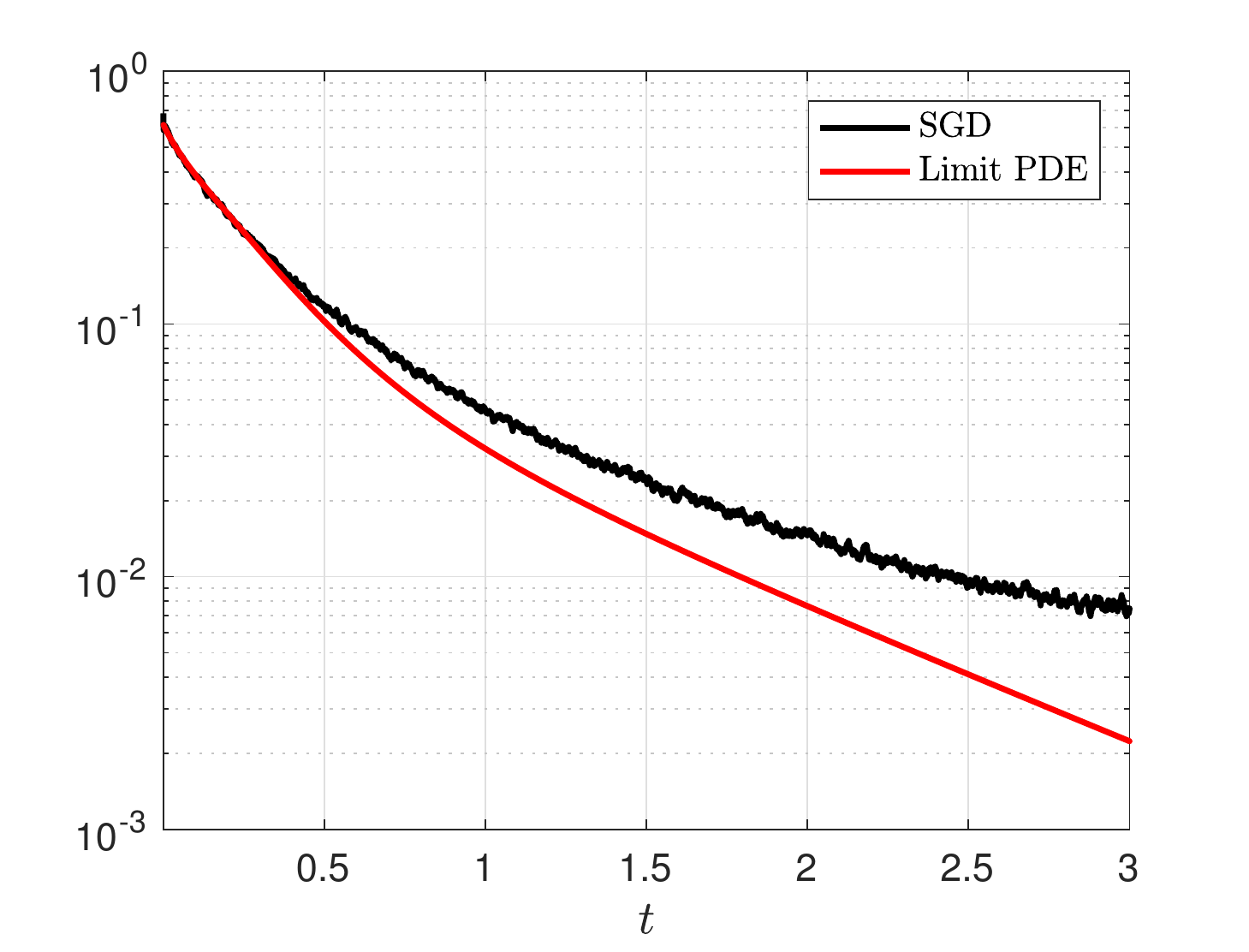}}\\
    \subfloat[Function $f$ and SGD estimates, $\delta=1/10$.]{\includegraphics[width=.5\columnwidth]{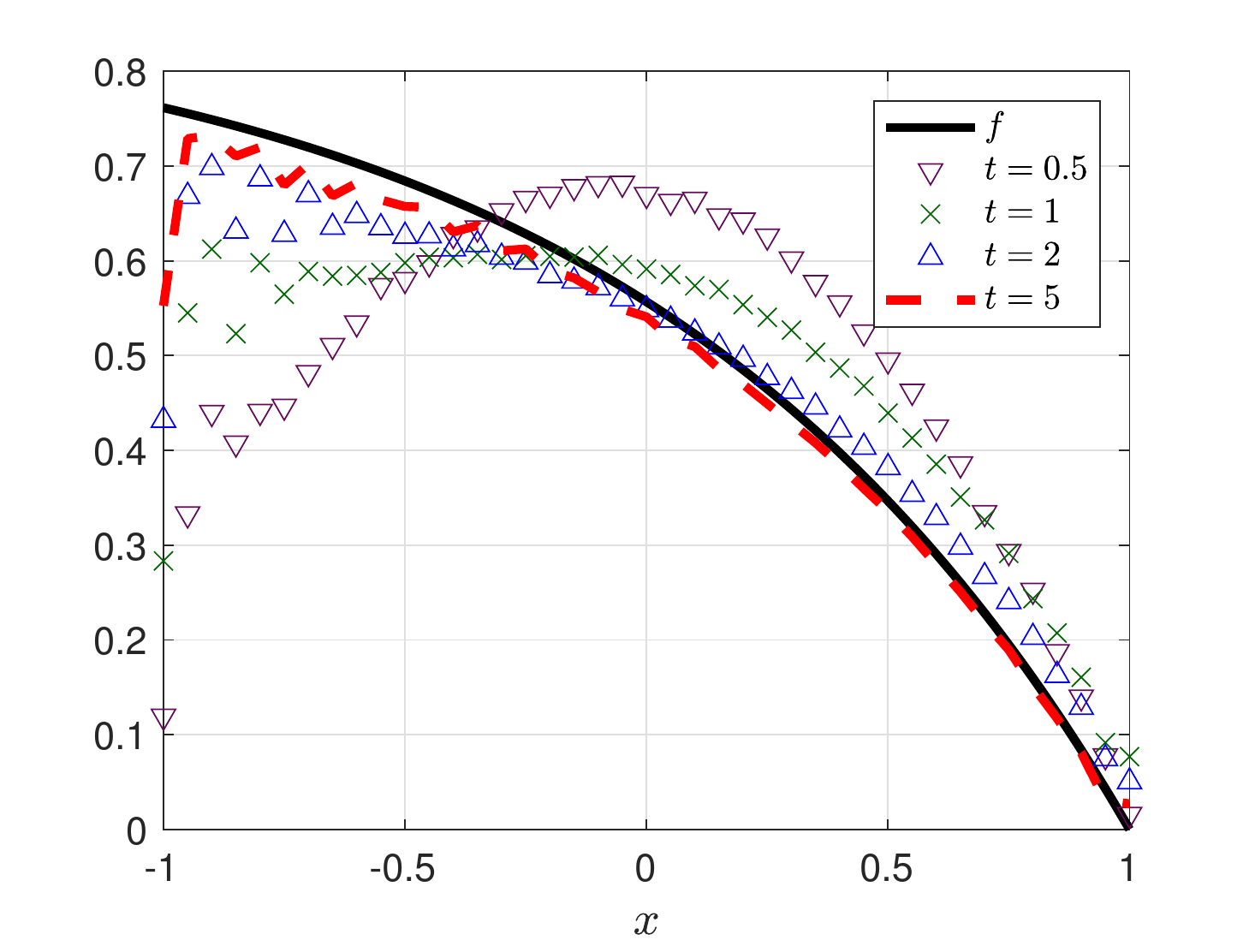}} 
    \subfloat[Normalized risk, $\delta=1/10$.]{\includegraphics[width=.5\columnwidth]{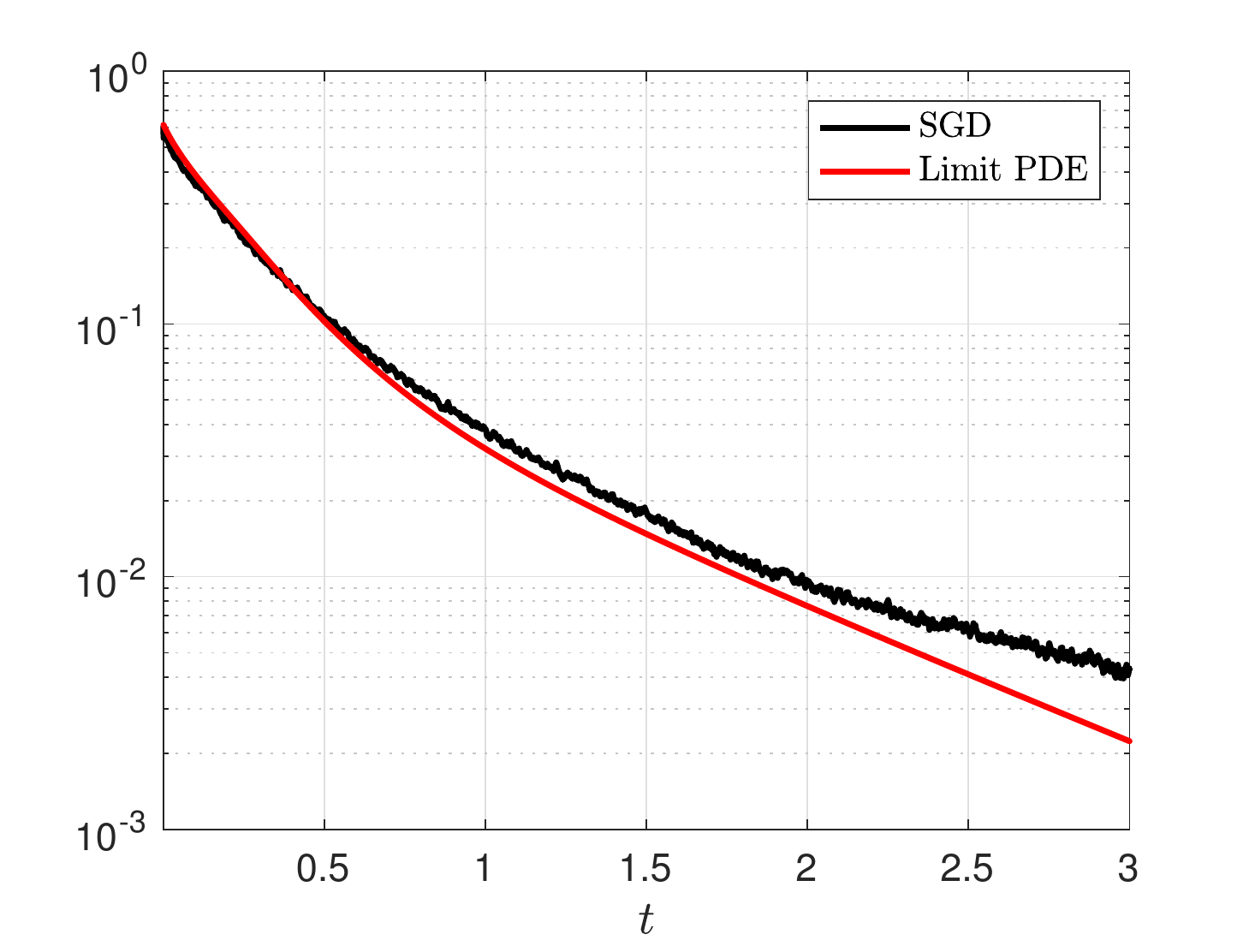}}\\
    \subfloat[Function $f$ and SGD estimates, $\delta=1/20$.]{\includegraphics[width=.5\columnwidth]{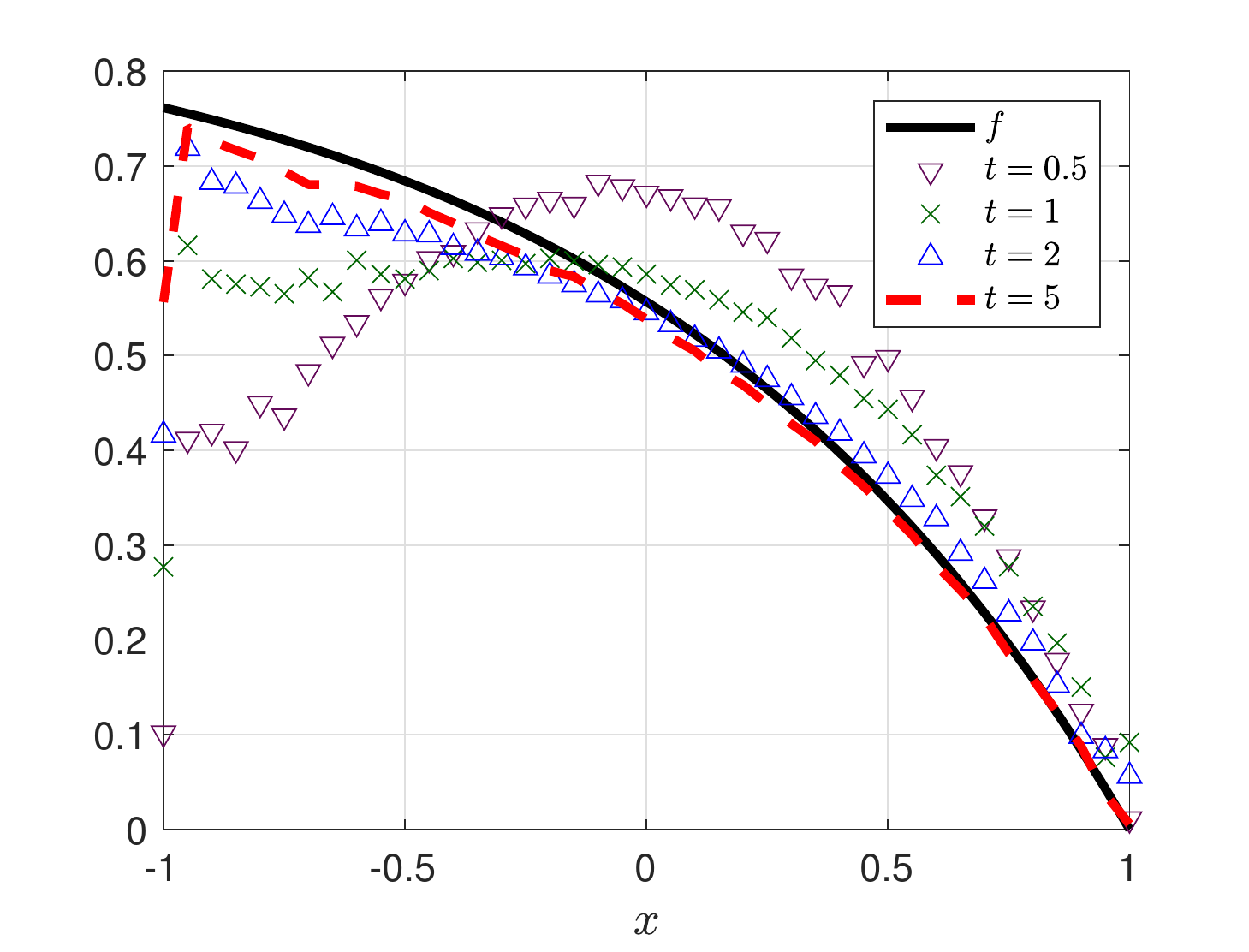}}
    \subfloat[Normalized risk, $\delta=1/20$.]{\includegraphics[width=.5\columnwidth]{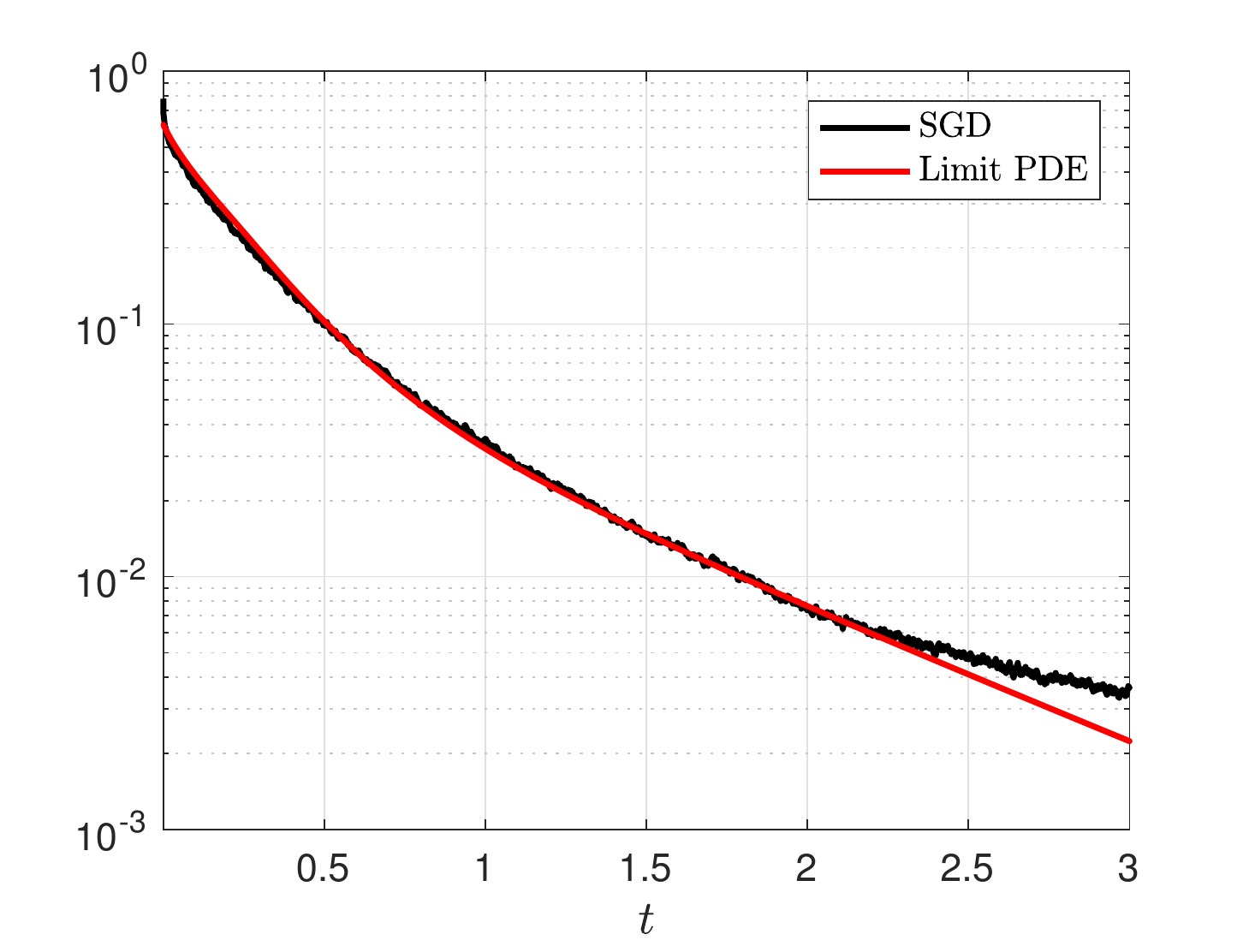}}\\
\caption{Dynamics of SGD update \eqref{eq:SGDup} at different times $t$ and for different values of $\delta$.}
\label{fig:SGD}
\end{figure}

\begin{figure}[t]
    \centering
    \subfloat[Function $f$ and PDE solution.]{\includegraphics[width=.5\columnwidth]{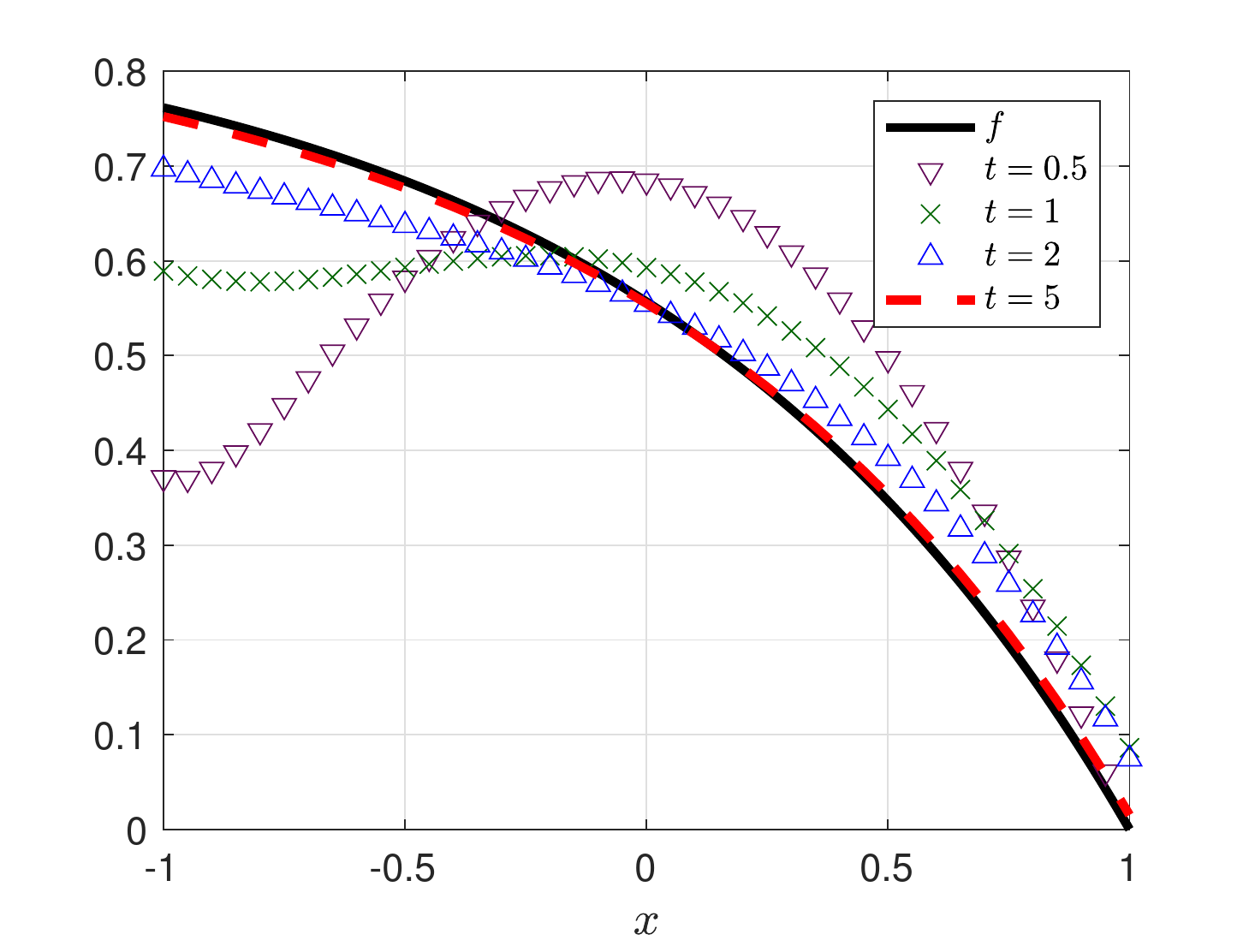}}
    \subfloat[Normalized risk.]{\includegraphics[width=.5\columnwidth]{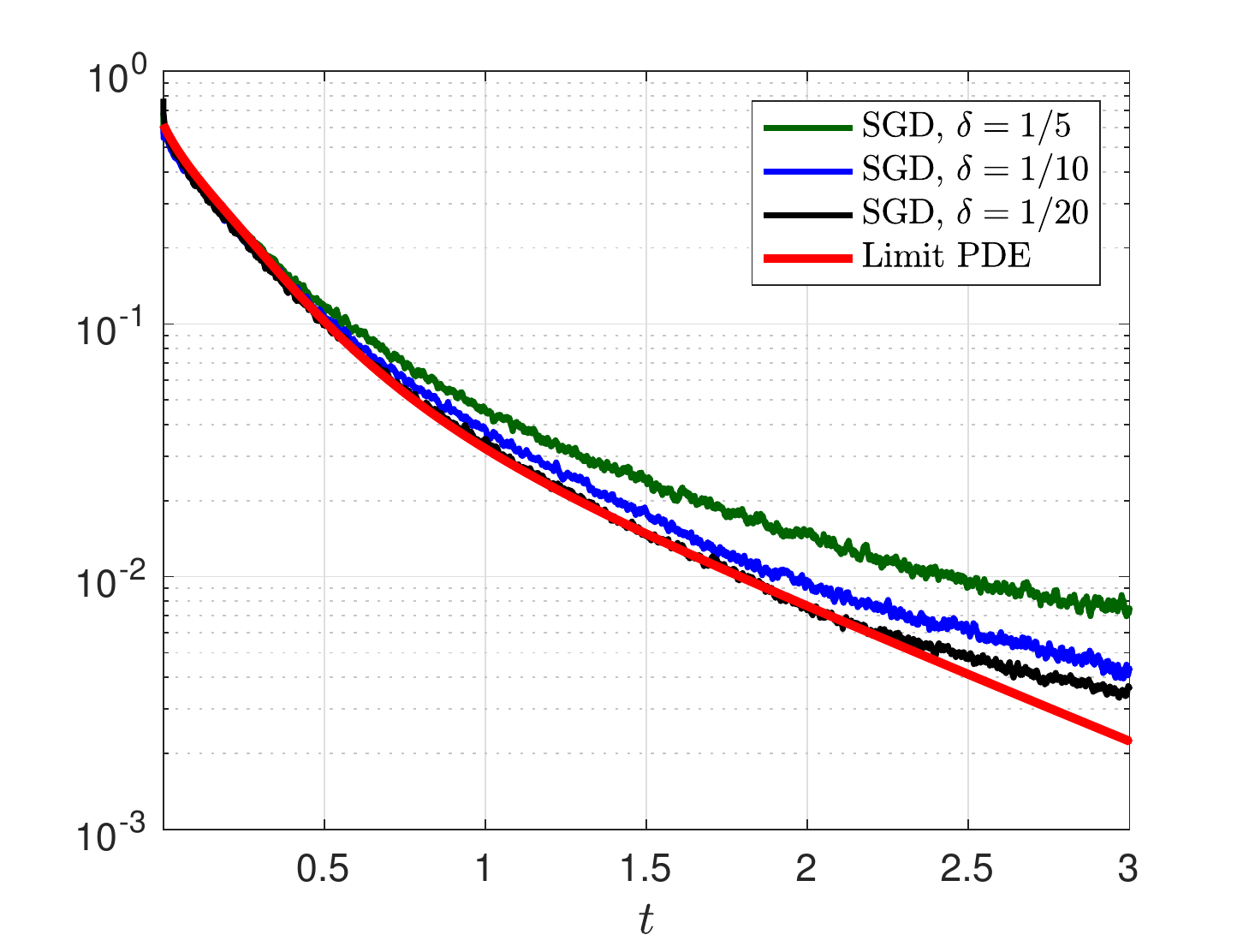}}\\
\caption{Dynamics of limit PDE \eqref{eq:PME} at different times $t$.}
\label{fig:PDE}
\end{figure}

\begin{figure}[t]
    \centering
\includegraphics[width=.5\columnwidth]{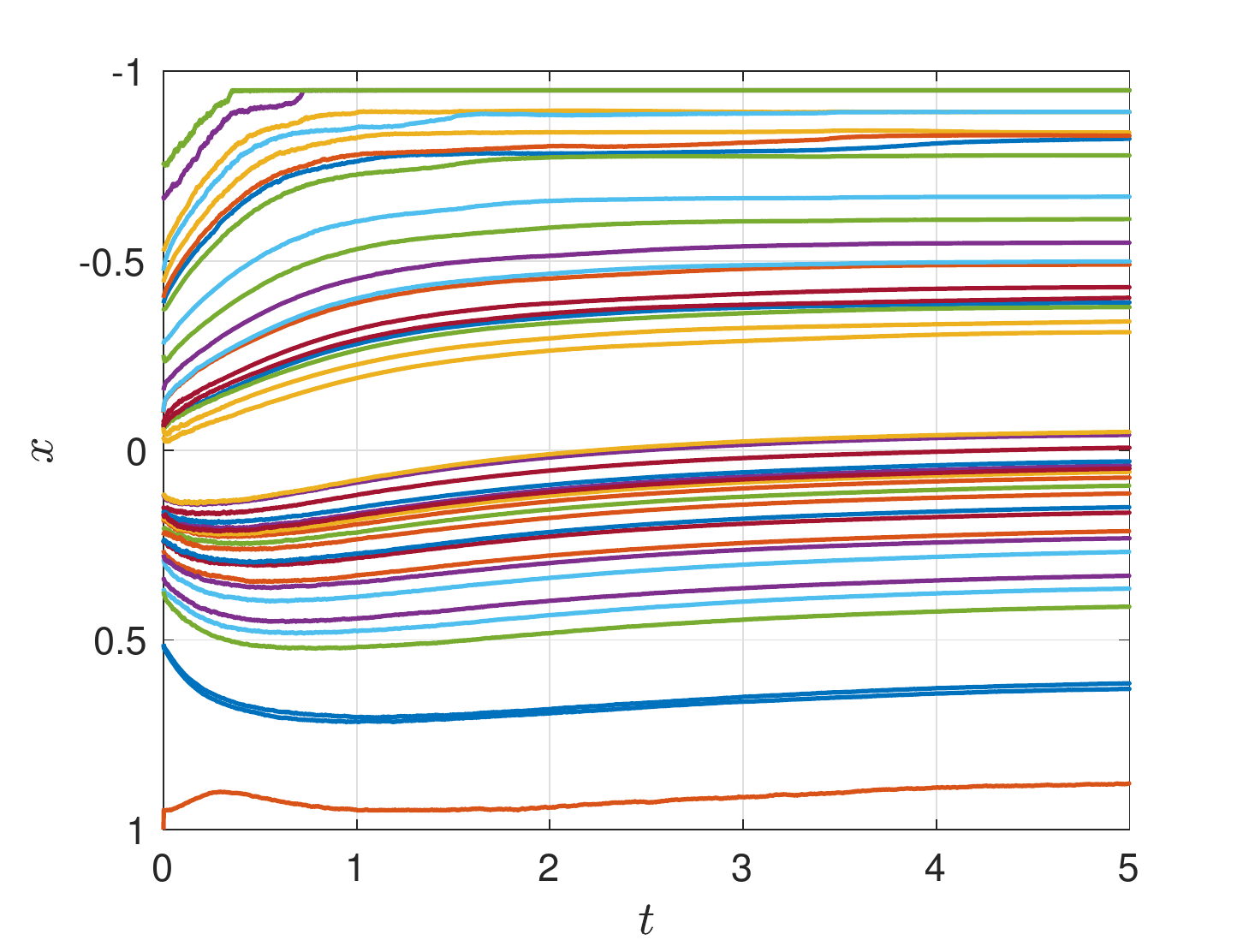}\hspace{-1em}
\includegraphics[width=.5\columnwidth]{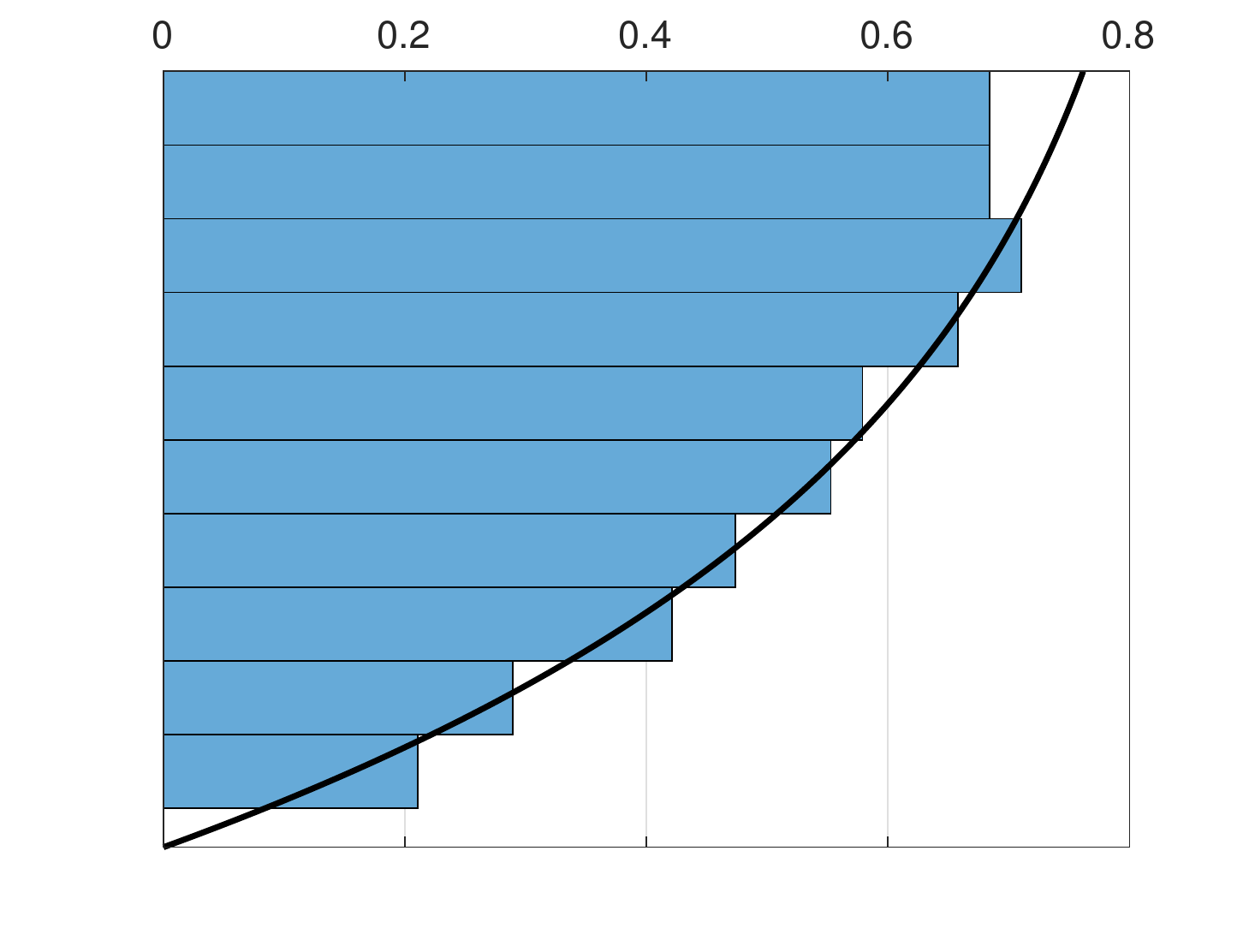}\\
\caption{Evolution of the value of $40$ weights chosen at random and histogram of their empirical distribution at time $t=5$.}
\label{fig:evolution}
\end{figure}

In the right column, we report the evolution of the population risk \eqref{eq:poprisk} normalized by  $\|f\|^2_{\Lp^2(\Omega)}$.
For comparison, we plot the evolution of the risk \eqref{eq:R-delta-0} as predicted by the limit PDE \eqref{eq:PME} with $\tau=0$. 
We solve the PDE \eqref{eq:PME} numerically using a finite difference scheme that enforces the conservation law $\int \rho(x,t)\de x=1$, see, e.g., 
\cite{thomas2013numerical}. In the finite difference scheme, we choose time step and spatial step  $\Delta t=10^{-5}$ and $\Delta x= 10^{-2}$, respectively. 
The curve obtained by this numerical solution appears to capture well the evolution of SGD towards optimality. The main difference is that, while 
the PDE \eqref{eq:PME} corresponds to $\delta=0$, and hence evolves towards a global optimum at zero risk, SGD converges to a non-zero risk value,
which can be interpreted as the approximation error, decreasing with $\delta$.

In Figure \ref{fig:PDE}, we illustrate the numerical solution of the PDE \eqref{eq:PME} by plotting \emph{(i)} the regression function $f$ together with the 
PDE solution $\rho_t$ (which coincides with the prediction $\hat{f}$ at $\delta=0$) at several times $t$, and \emph{(ii)} the PDE prediction for the 
risk $R(\rho_t)$ \eqref{eq:R-delta-0}  normalized with respect to $\|f\|^2_{\Lp^2(\Omega)}$ \comm{(this plot aggregates data from Figs.~\ref{fig:SGD}.(b), (d), (f))}. 
We also compare the risk \eqref{eq:R-delta-0} to the population risk 
$R_N(\bw^k)$ achieved by SGD for different values of $\delta$. Note that, as $\delta$ becomes smaller, the risk converges to 
the predicted curve. The risk of the limit PDE \eqref{eq:PME} converges to $0$ exponentially fast in $t$, as predicted by the 
strong displacement convexity of $R(\rho)$.

In Figure \ref{fig:evolution}, we consider the SGD algorithm with projection $\Proj$, see \eqref{eq:SGDup}. We pick $N=200$, $\tau=0$, $\eps=10^{-6}$ and $\delta=1/20$. On the left, we illustrate the evolution of the value of $40$ weights chosen at random; and on the right, we plot the histogram of their empirical distribution at $t=5$. Note that this histogram matches well the regression function $f$ plotted in black.

\subsection{A two-dimensional concave example} 

Next, we consider a two-dimensional example. We set $\Omega = [-1, 1]^2$ and 
\begin{equation}\label{eq:deffexamples}
f(\bx) = \frac{c_1 - \log(e^{\langle \bq_1, \bx\rangle} + e^{\langle \bq_2, \bx\rangle})}{c_2},
\end{equation}
with $\bq_1 = (2.5127,   -2.4490)$, $\bq_2 = (0.0596,    1.9908)$ and where  $c_1$ and $c_2$ are chosen so that $f$ is non-negative and $\int_\Omega f(\bx)\, \de\bx=1$. The kernel $K$ is given by $K(\bx) = C_d \kappa(|\bx|)$, where $\kappa$ is defined in \eqref{eq:Kex} and  $C_d$ is a normalization constant ensuring that $\int_{\Ball(\b0; 1)} K(\bx)\,\de \bx = 1$. Again, the initialization $\rho_{\sinit}$ is a 
truncated  Gaussian: $\rho_{\sinit}(\bx) = c\cdot\exp(-|\bx|^2/(2\sigma^2))\, \bfone_{[-1,1]^2}(\bx)$, with $\sigma=1/3$. We compare the normalized risk of SGD with no 
projection $\Proj$ ($N=2000$, $\tau=0$ and $\eps = 10^{-6}$) for $\delta \in \{1/3, 1/5, 1/10\}$ with that of the limit PDE \eqref{eq:PME}. Figure \ref{fig:2d} shows that, already at $\delta=1/10$, the risk of SGD converges to the predicted curve and the risk of the limit PDE \eqref{eq:PME} tends to $0$ exponentially fast in $t$.

\begin{figure}[t]
    \centering
    \includegraphics[width=.6\columnwidth]{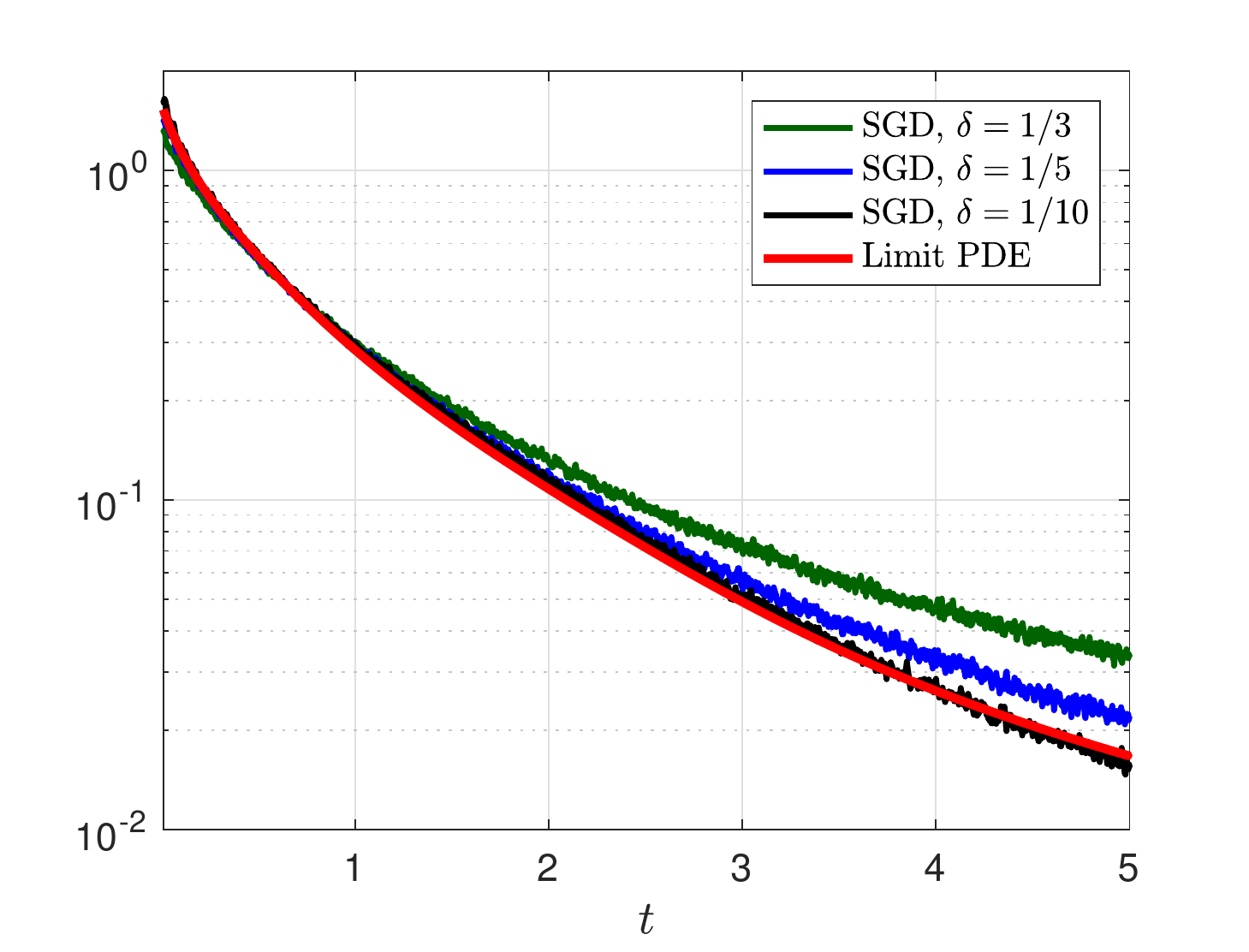}
\caption{Normalized risk of SGD for different values of $\delta$ compared with that of the limit PDE for a two-dimensional example.}
\label{fig:2d}
\end{figure}

\begin{figure}[p]
    \centering
    \subfloat[$N=200$, $n$ varies on the $x$-axis.]{\includegraphics[width=.6\columnwidth]{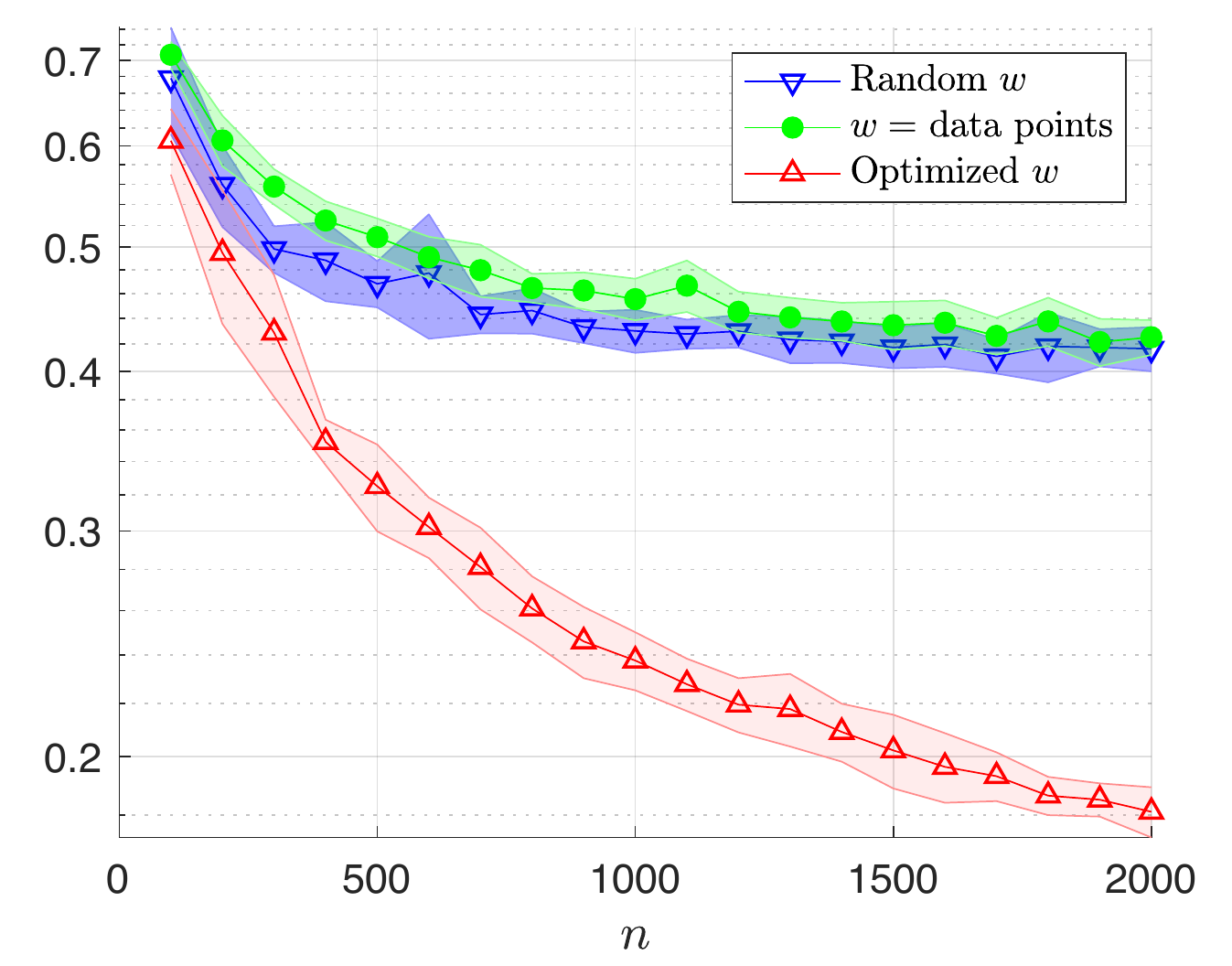}}\\
        \subfloat[$n=2000$, $N$ varies on the $x$-axis.]{\includegraphics[width=.6\columnwidth]{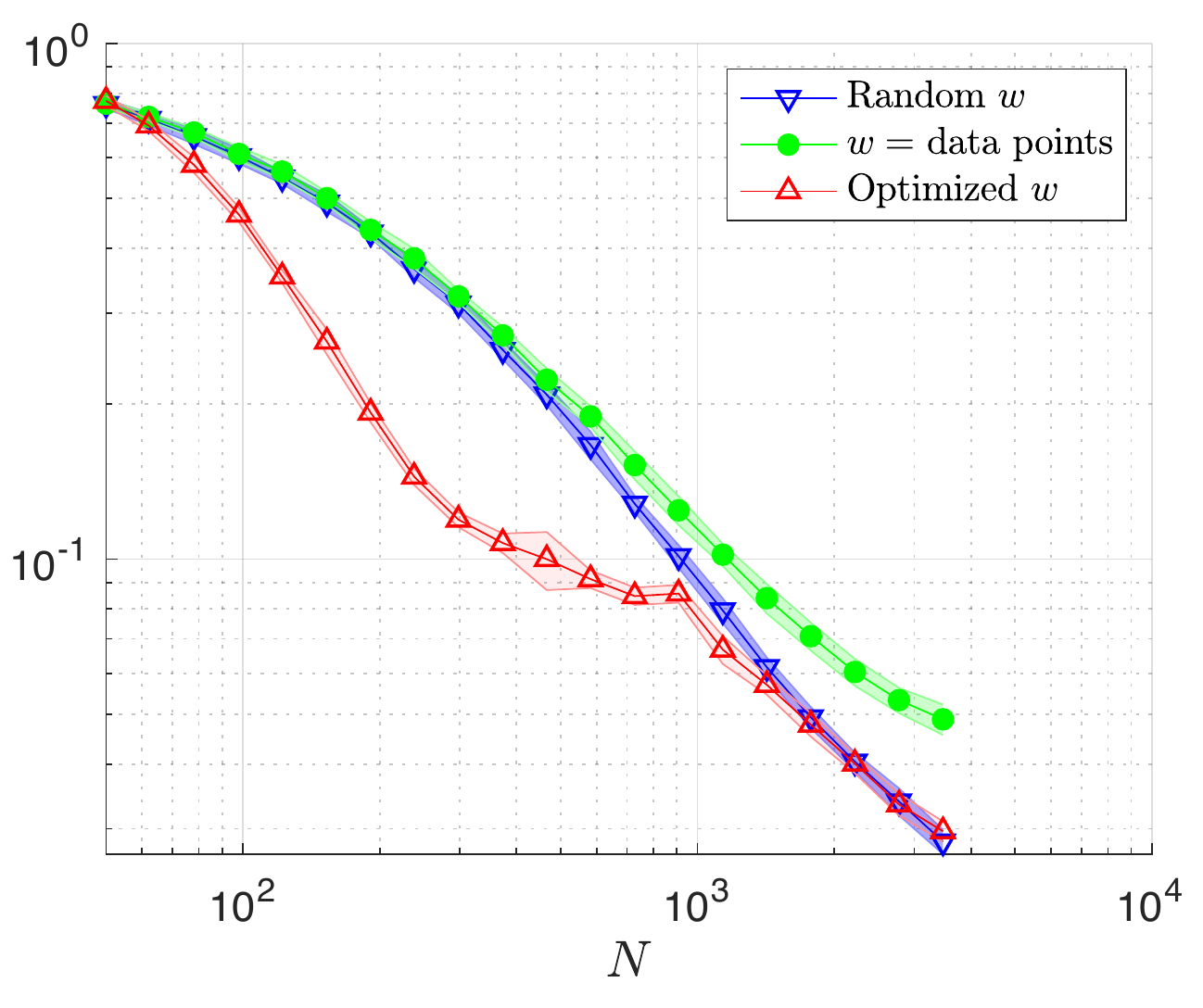}}
\caption{Generalization error achieved by fitting $\ba$ from the data for three different choices of the weights $\bw$: in red, the $\bw_i$'s are optimized before-hand via SGD, as suggested in this paper; in blue, the $\bw_i$'s are uniform in $\Omega$; and in green, the $\bw_i$'s are equal to random data points.}
\label{fig:SGDnew}
\end{figure}

\subsection{\comm{Comparing feature learning to random features}}

\comm{As discussed in the introduction, it is useful to consider the more general model
\begin{align}
\hf(\bx;\bw, \ba)= \sum_{i=1}^N a_i K^\delta(\bx-\bw_i)\, , \label{eq:GeneralFhat2}
\end{align}
with parameters  $\ba=(a_1, \ldots, a_N)$ as well as $\bw= (\bw_1,\dots,\bw_N)$. This setting allows to compare two different approaches:
\begin{itemize}
\item[$(i)$] \emph{Random feature regression}: the weights $\bw$ are chosen independently of the labels $y_i$ (we allow for dependence on the covariates $\bx_i$).
\item[$(ii)$] \emph{Feature learning}: the weights $\bw$ depend on the data $(y_i,\bx_i)$.
\end{itemize}
In order to compare these two approaches, we assume to be given i.i.d. data $\{(y_i,\bx_i)\}_{i\le n}$, with $\bx_i\sim\Unif(\Omega)$,  $y_i=f(\bx_i)$
and determine the parameters $\ba$ by the same method, ridge regression. More explicitly, define the matrix $\bZ\in\mathbb R^{n\times N}$ as $(\bZ)_{i, j} =  K^\delta(\bx_i-\bw_j)$. 
Then, we estimate $\ba$ via
\begin{equation}
\hat{\ba} = (\bZ^\sT\bZ + \lambda \boldsymbol I )^{-1} \bZ^\sT \by,
\end{equation}
where $\lambda$ is chosen via cross-validation on a hold-out set, comprising $10\%$ of the samples.}

\comm{In Figure \ref{fig:SGDnew}, we compare the performance of three different ways to construct the weights $\bw$: 
\emph{`random $\bw$,'}  we choose the weights $\bw_i$ independently and uniformly at random in $\Omega$ (blue triangles pointing down);
\emph{`$\bw=$ data points,'} we choose the weights $\bw_i$ uniformly at random among the data points (green circles);
\emph{`optimized $\bw$,'}  we use the output of the projected SGD algorithm of the previous sections (red triangles pointing up).
The first two can be regarded as `random features' approaches, while the latter is a `feature learning' method.}

\comm{For the optimized $\bw$, we use exactly the same algorithm in as in \eqref{eq:SGDup}  (without coefficients $\ba$ in the SGD update), 
with the only difference that each SGD step is carried out with respect to an independent sample from the empirical data, with replacement.
SGD is stopped after $k_{\max}$ iteration, and the coefficient  $\hat{\ba}$ are computed according to \eqref{eq:GeneralFhat2}.
Notice that this procedure is probably suboptimal, and it would be better to optimize $\ba$ and $\bw$ jointly:
we choose this simpler two-stage procedure to have a more direct application of the algorithm analyzed in the paper, and a comparison with the random feature methods. 
We set $\tau=0$ (noiseless SGD), and constant step size $\varepsilon=5\cdot 10^{-4}$. The number of iterations $k_{\max}\in\{5\cdot 10^3, 15 \cdot 10^3, 5\cdot 10^4, 15 \cdot 10^4, 5\cdot 10^5, 15 \cdot 10^5\}$ is chosen via cross-validation, by using the same hold-out set employed to optimize $\lambda$.}

  \comm{We set $\Omega = [-1, 1]^4$ and define $y_j = f(\bx_j)$, where $f(\bx)$ takes the form \eqref{eq:deffexamples} with $\bq_1 = (-0.3832, 
0.3074,  -0.3198, 0.4792)$ and $\bq_2 = (0.3502, -0.1471,$ $0.1685, 0.0546)$. Again, $c_1$ and $c_2$ are chosen so that $f$ is non-negative and $\int_\Omega f(\bx)\, \de\bx=1$; the kernel $K$ is given by $K(\bx) = C_d \kappa(|\bx|)$, where $\kappa$ is defined in Eq.~\eqref{eq:Kex} and  $C_d$ ensures that $\int_{\Ball(\b0; 1)} K(\bx)\,\de \bx = 1$.}

\comm{After estimating $\bw_i$ and $a_i$ by either methods, we generate a test set of $10,000$ samples and use it to estimate the generalization error. 
We perform $20$ independent trials of the experiment, and we plot the average risk normalized by $\|f\|^2_{\Lp^2(\Omega)}$ together with the error bar at 1 standard deviation.
 In Figure \ref{fig:SGDnew}-(a), we fix the number of neurons $N=200$ and we plot the normalized risk as a function of the number of data points $n$. In Figure \ref{fig:SGDnew}-(b), we fix the number of samples $n$ to $2000$ and we plot the normalized risk as a function of the number of neurons $N$. The data set used for cross-validation has size $\max(n/10, 40)$. Note that feature learning  leads to improved performance in both settings. The improvement becomes more pronounced with the sample size $n$,
presumably because a better set of weights $\bw_i$ can be learnt. On the other hand, when the number of neurons $N$ becomes very large, random $\bw_i$'s are already covering
$\Omega$ densely enough, and there is no significant advantage in feature learning. }

\begin{figure}[p]
    \centering
    \subfloat[Function $f$ and SGD estimates, $\delta=1/5$.]{\includegraphics[width=.5\columnwidth]{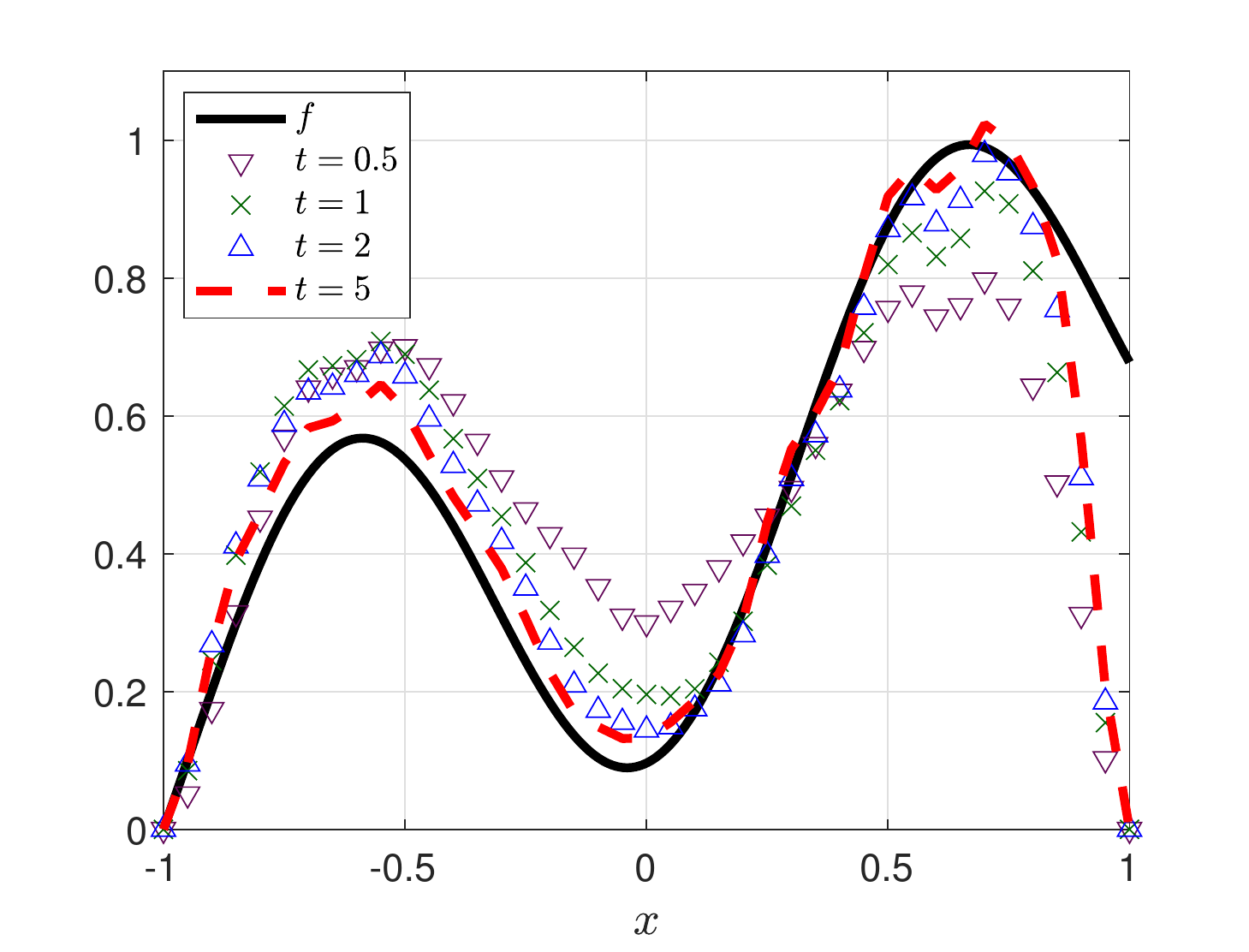}}
    \subfloat[Normalized risk, $\delta=1/5$.]{\includegraphics[width=.5\columnwidth]{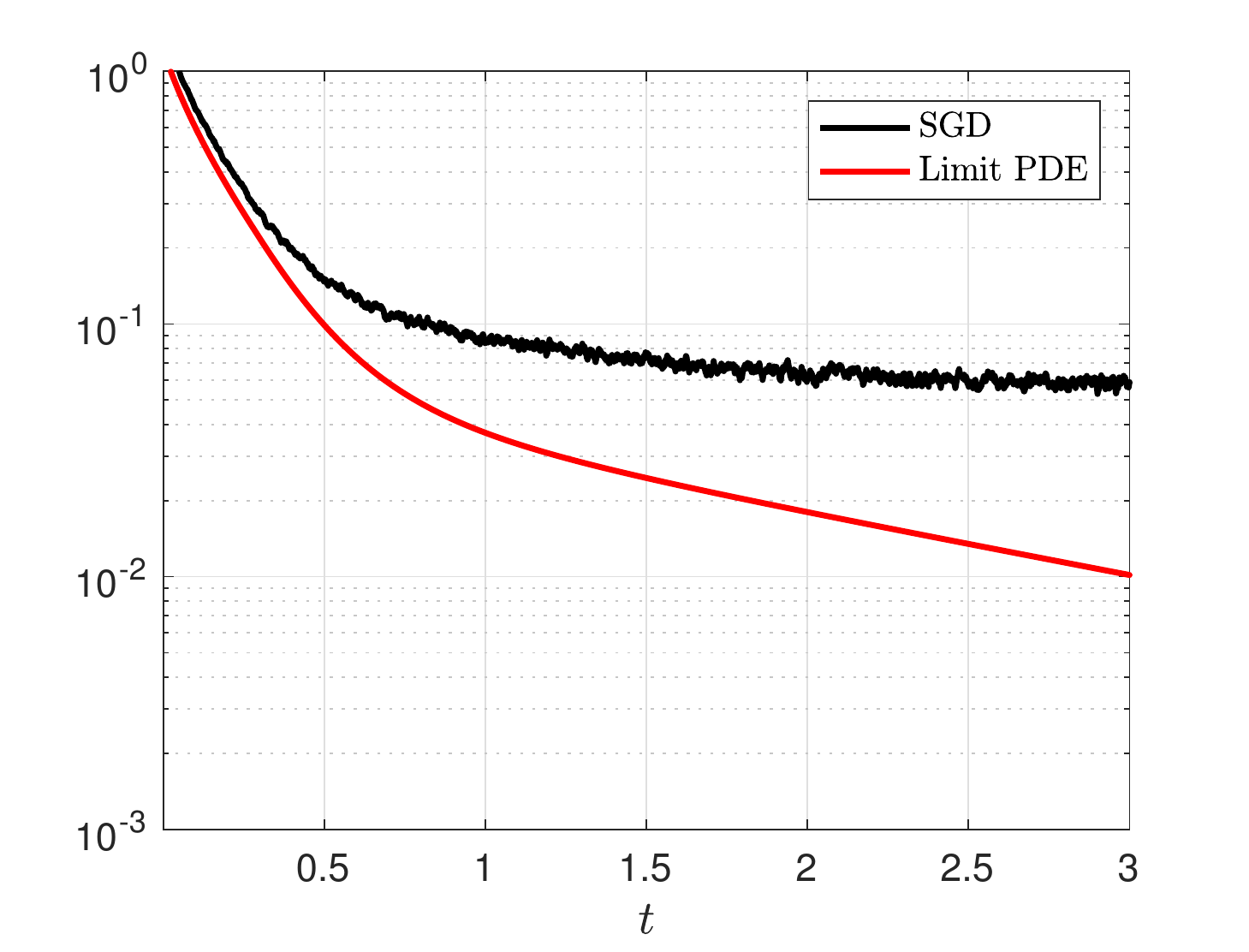}}\\
    \subfloat[Function $f$ and SGD estimates, $\delta=1/10$.]{\includegraphics[width=.5\columnwidth]{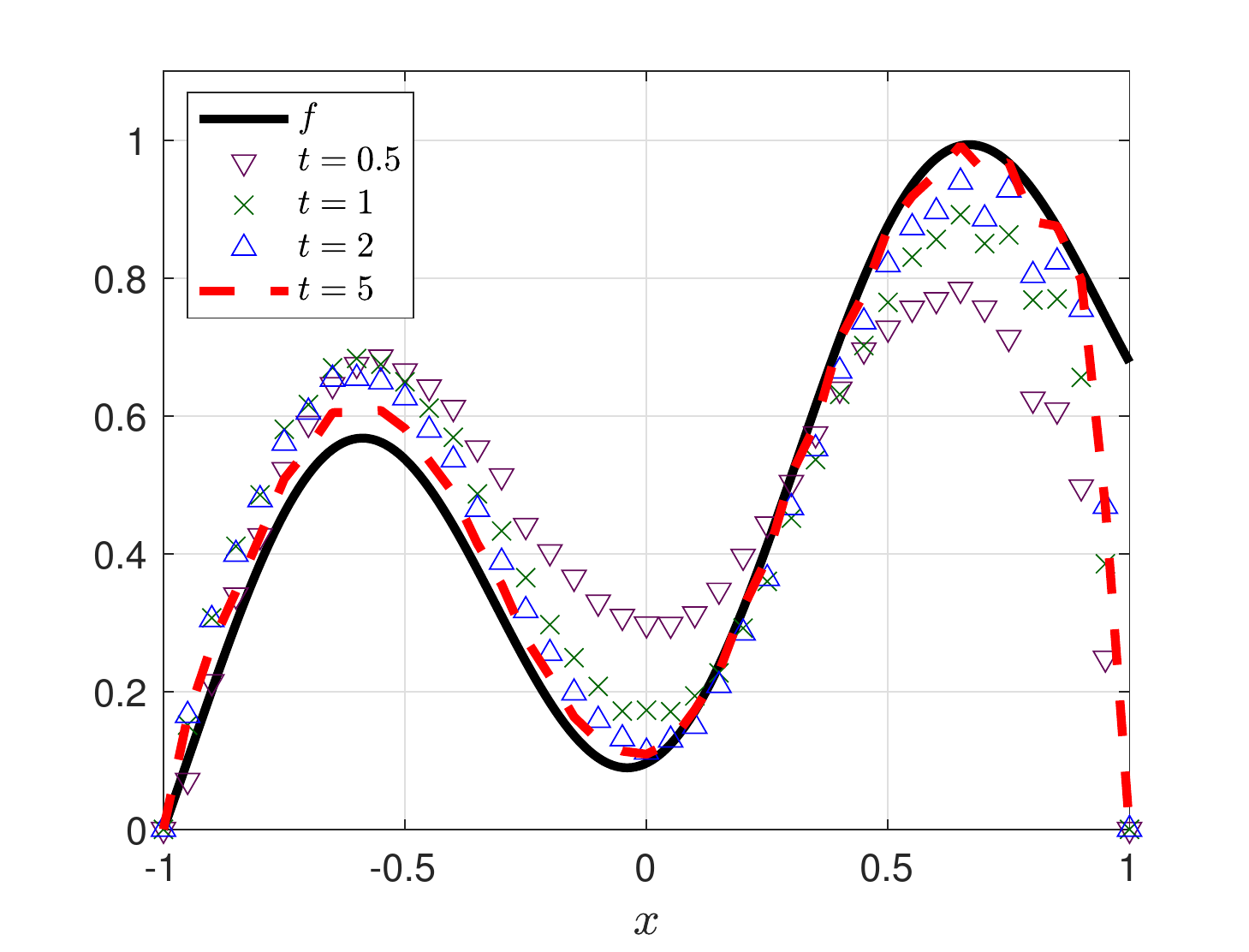}} 
    \subfloat[Normalized risk, $\delta=1/10$.]{\includegraphics[width=.5\columnwidth]{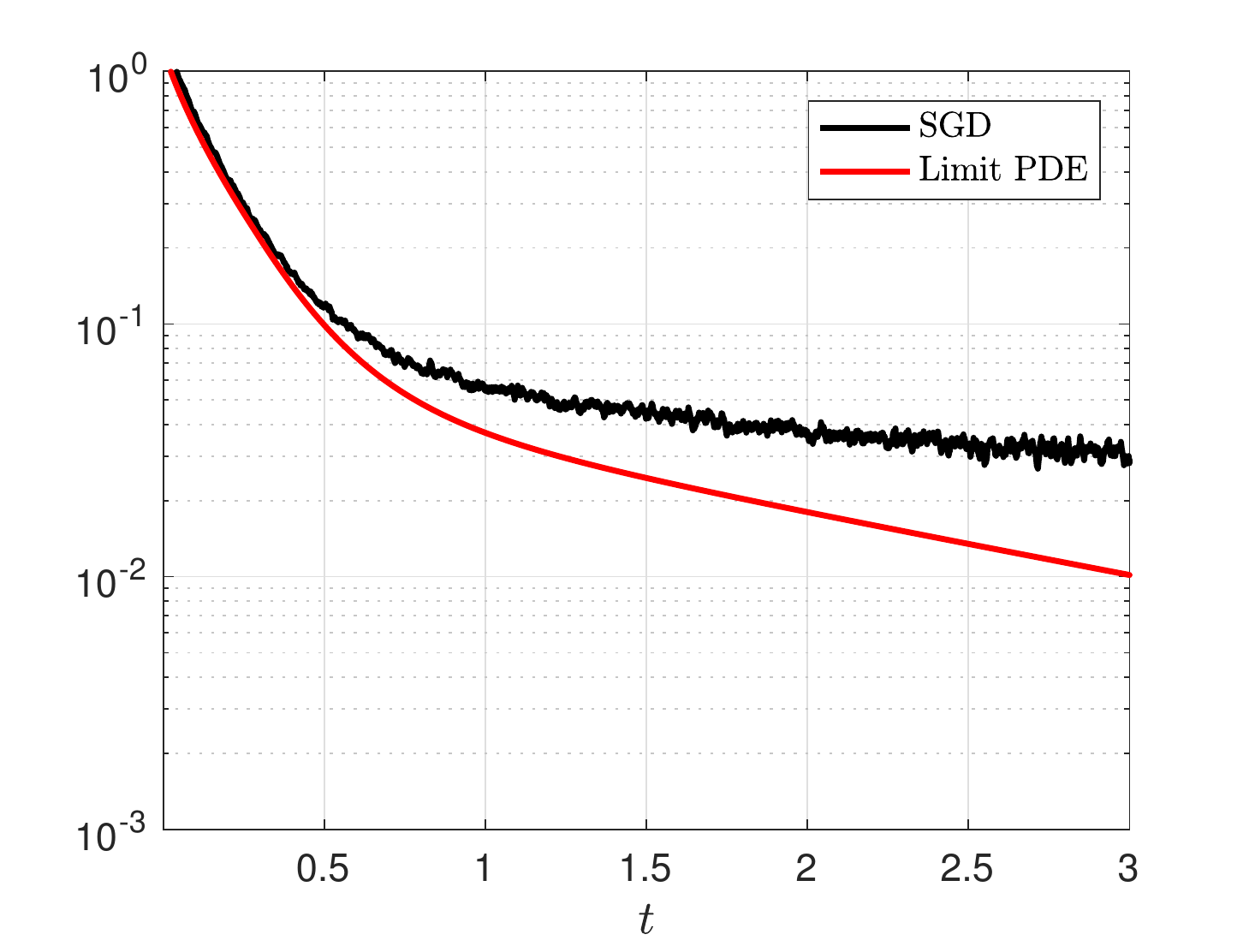}}\\
    \subfloat[Function $f$ and SGD estimates, $\delta=1/20$.]{\includegraphics[width=.5\columnwidth]{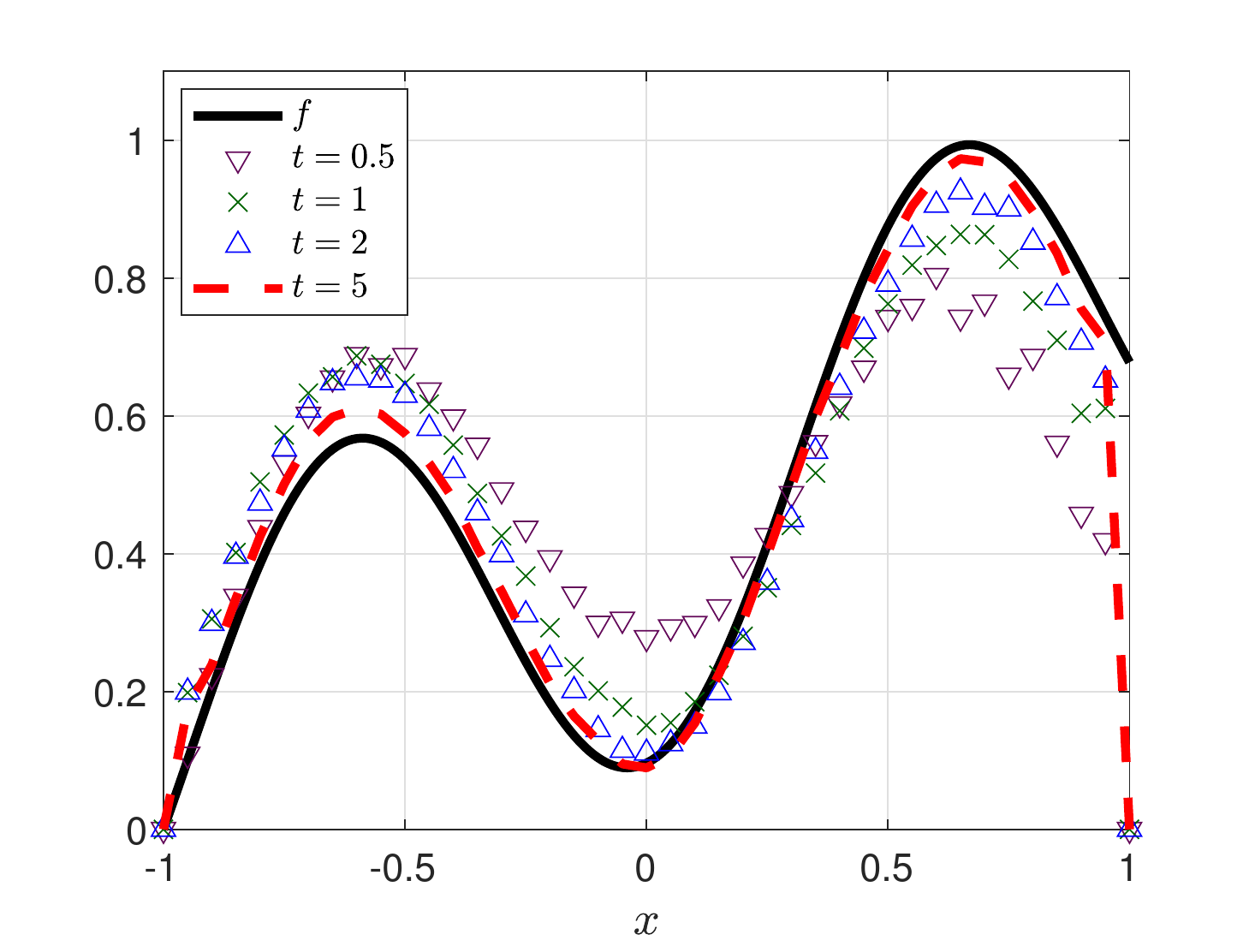}}
    \subfloat[Normalized risk, $\delta=1/20$.]{\includegraphics[width=.5\columnwidth]{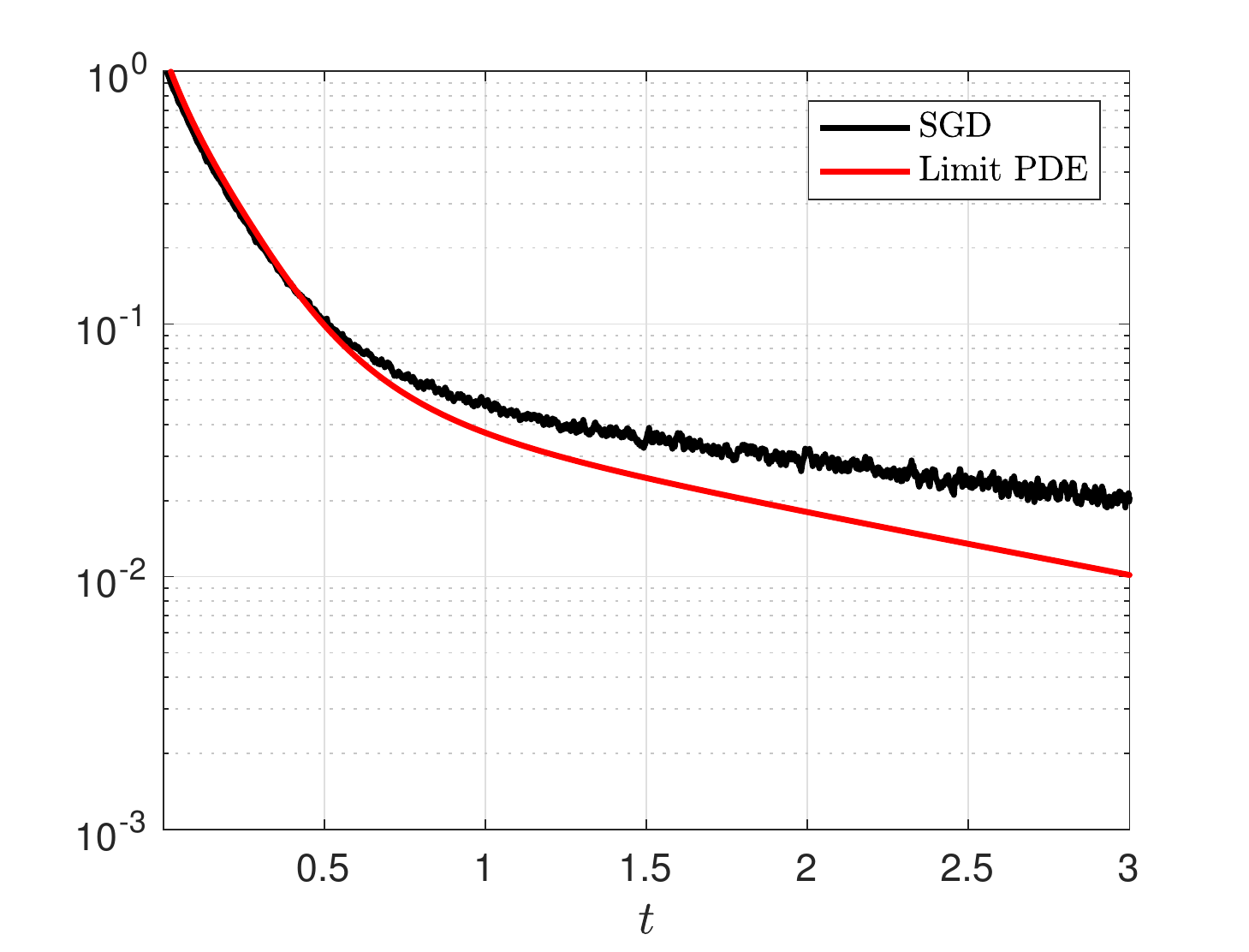}}\\
\caption{Dynamics of SGD update \eqref{eq:SGDup} at different times $t$ and for different values of $\delta$ for a non-concave target function $f$.}
\label{fig:SGDnc}
\end{figure}

\begin{figure}[t]
    \centering
    \subfloat[Function $f$ and PDE solution.]{\includegraphics[width=.5\columnwidth]{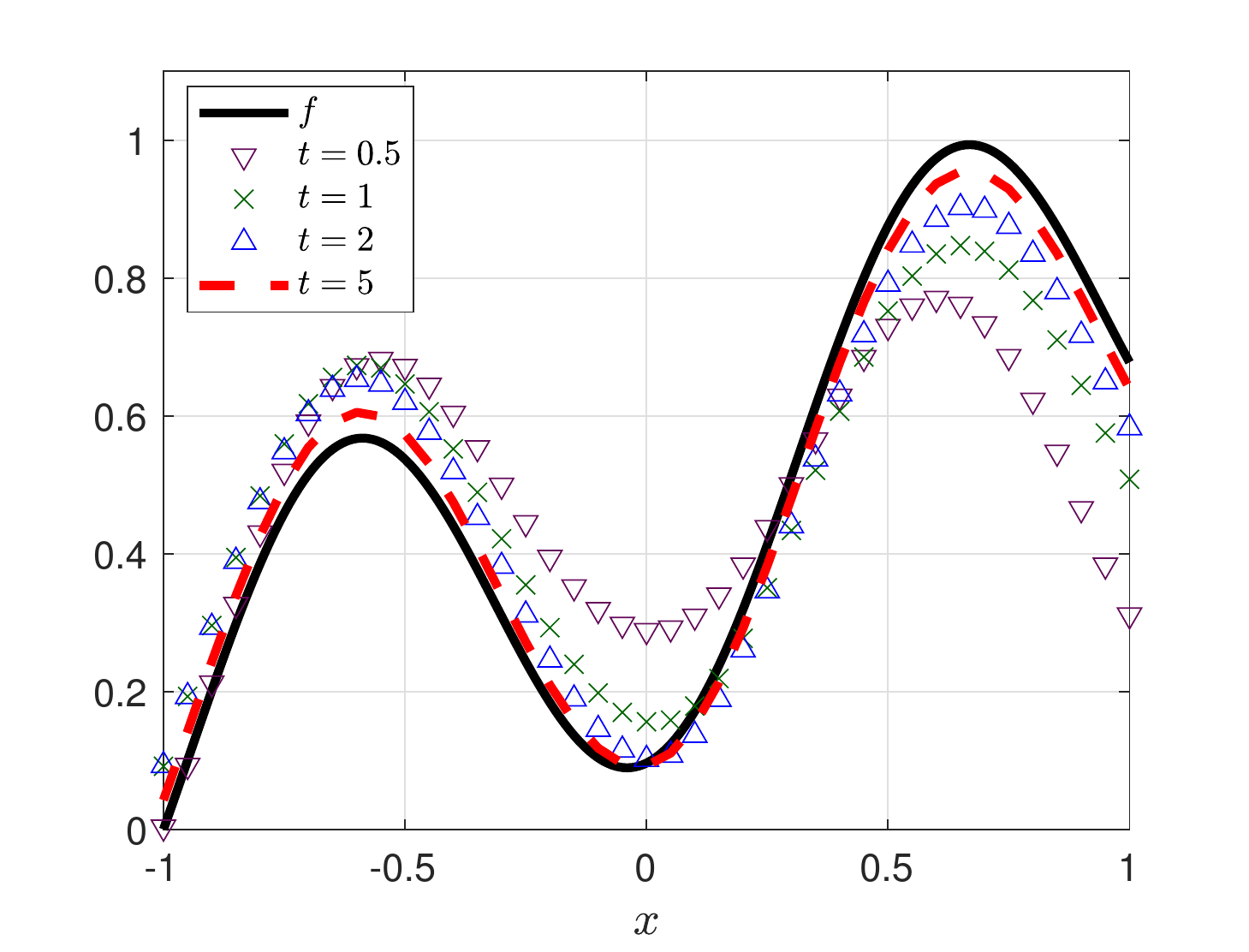}}
    \subfloat[Normalized risk.]{\includegraphics[width=.5\columnwidth]{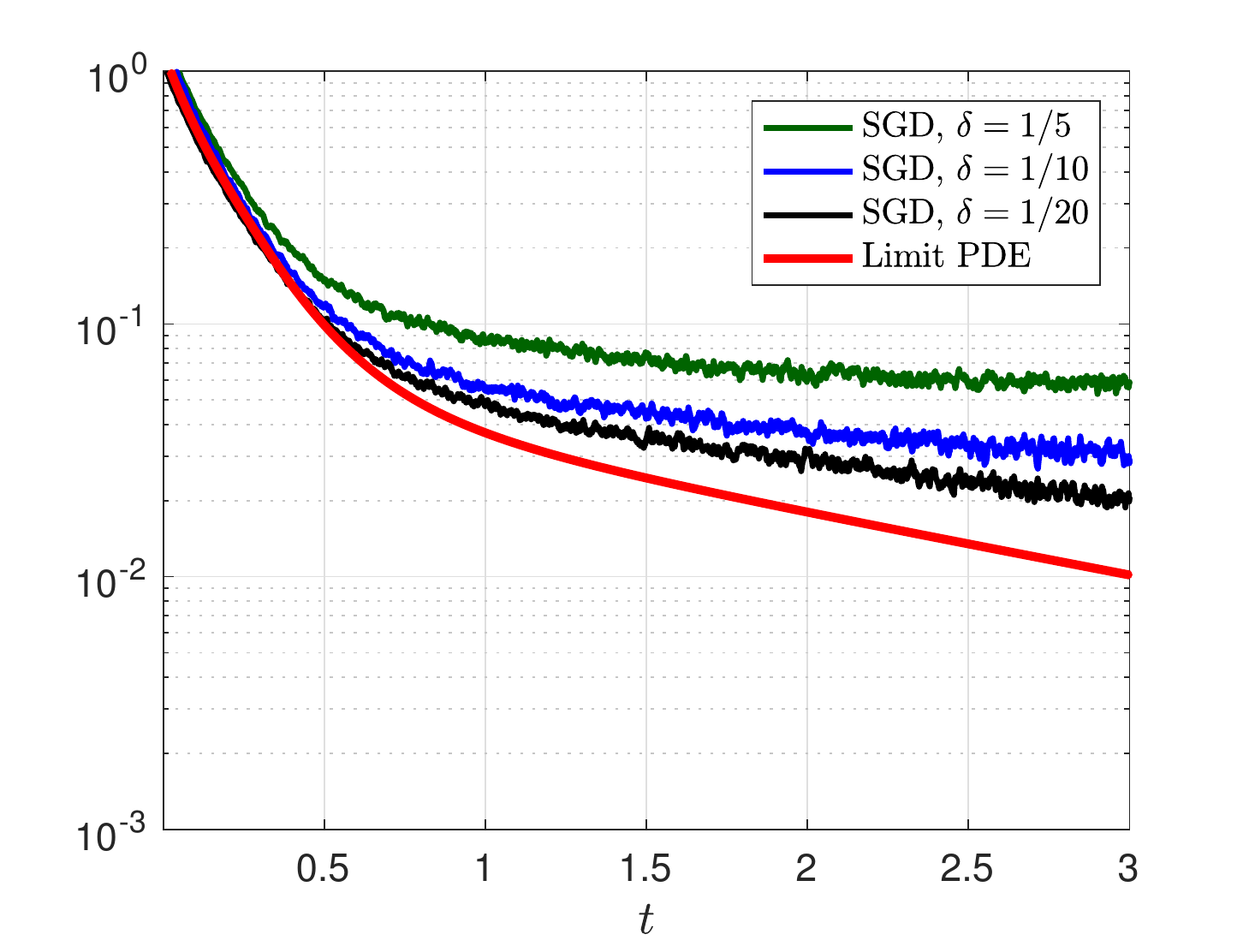}}\\
\caption{Dynamics of limit PDE \eqref{eq:PME} at different times $t$ for a non-concave target function $f$.}
\label{fig:PDEnc}
\end{figure}

\subsection{\comm{A non-concave one-dimensional example}} 
\comm{We set $\Omega = [-1, 1]$ and $f(x) = (x + \sin(5x-\pi/2) -c_1)/c_2$, where $c_1$ and $c_2$ are chosen so that $f$ is non-negative and $\int_\Omega f(x)\,\de x=1$. Note that the target function $f$ is bimodal, thus it is not concave. We perform the same numerical experiment described in Section \ref{sec:OnedConcave}. In Figure \ref{fig:SGDnc}, left column, we plot the true function $f(\,\cdot\,)$  together with the neural network estimate $\hf(\,\cdot\, ;\bw^k)$ at several points in time
 $t$, where different plots correspond to different values of $\delta\in\{1/5, 1/10, 1/20\}$. In the right column, we report the evolution of the population risk \eqref{eq:poprisk} normalized by  $\|f\|^2_{\Lp^2(\Omega)}$. In Figure \ref{fig:PDEnc}, we plot \emph{(i)} the regression function $f$ together with the 
PDE solution $\rho_t$ at several times $t$, and \emph{(ii)} the PDE prediction for the 
risk $R(\rho_t)$ \eqref{eq:R-delta-0} (normalized with respect to $\|f\|^2_{\Lp^2(\Omega)}$) compared with the population risk 
$R_N(\bw^k)$ achieved by SGD for different values of $\delta$. Even if the target function is not concave, the results are similar to those presented in 
the concave case: \emph{(i)} the network estimates $\hf(\,\cdot\, ;\bw^k)$ seem to converge to a limit curve which
is an approximation of the true function $f$, \emph{(ii)} the quality of the approximation improves as $\delta$ gets smaller, and \emph{(iii)} the risk of the limit PDE \eqref{eq:PME} converges to $0$ exponentially fast in $t$.} 
%

\subsection{\comm{Failure for small $N$}}

\comm{ We repeat the same experiment described in Section \ref{sec:OnedConcave} for a smaller number of neurons $N=20$. As can be seen in Figures \ref{fig:SGDN20} and \ref{fig:PDEN20}, the quality of the approximation becomes worse as $\delta$ gets smaller. This is expected because with small number of activations, reducing their bandwidth $\delta$ leads to a worse performance as they are all zero on a large part of the space. Put differently, the number of neurons is too small to guarantee convergence of SGD to the predictions of the Wasserstein gradient flow theory.
}

\begin{figure}[p]
    \centering
    \subfloat[Function $f$ and SGD estimates, $\delta=1/5$.]{\includegraphics[width=.5\columnwidth]{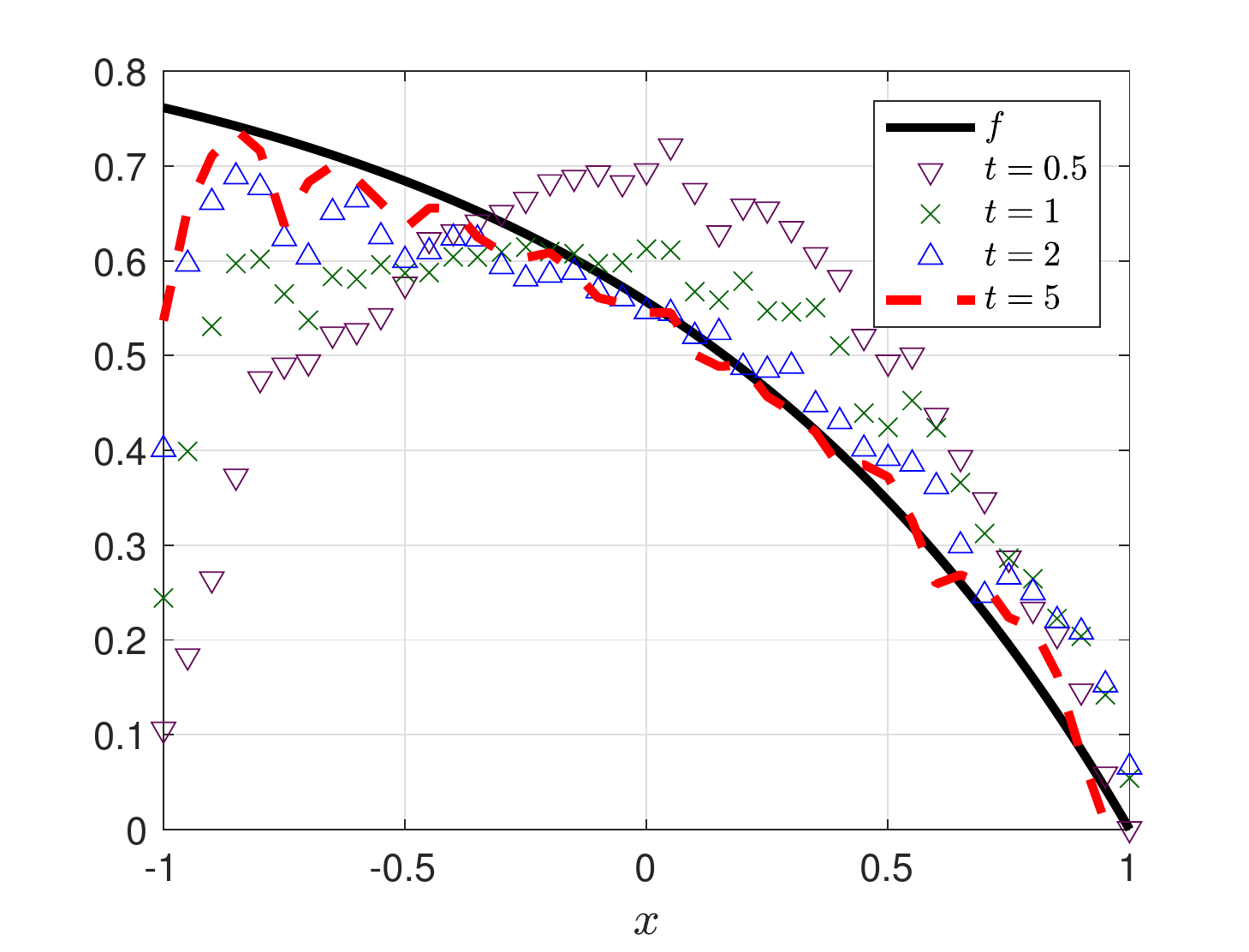}}
    \subfloat[Normalized risk, $\delta=1/5$.]{\includegraphics[width=.5\columnwidth]{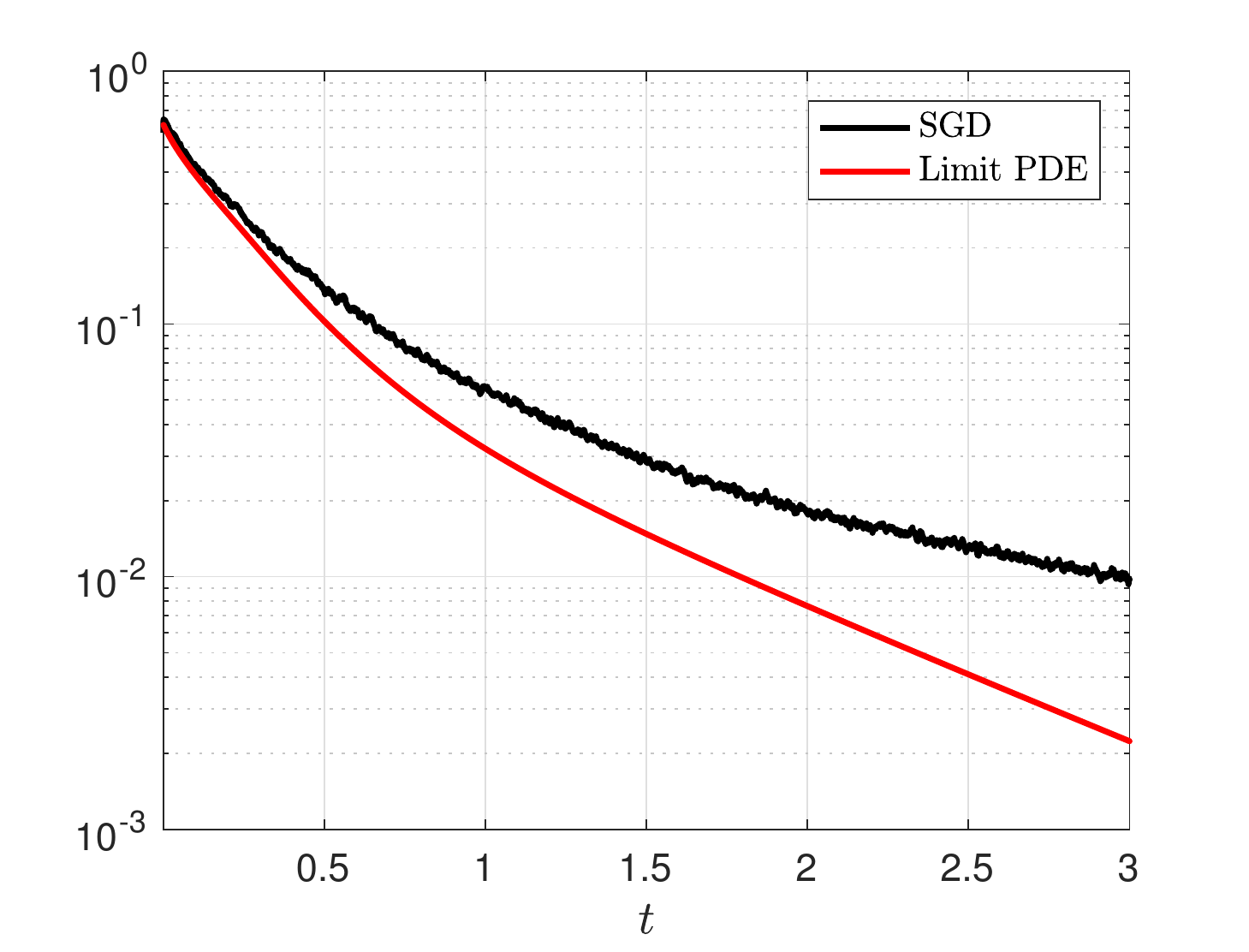}}\\
    \subfloat[Function $f$ and SGD estimates, $\delta=1/10$.]{\includegraphics[width=.5\columnwidth]{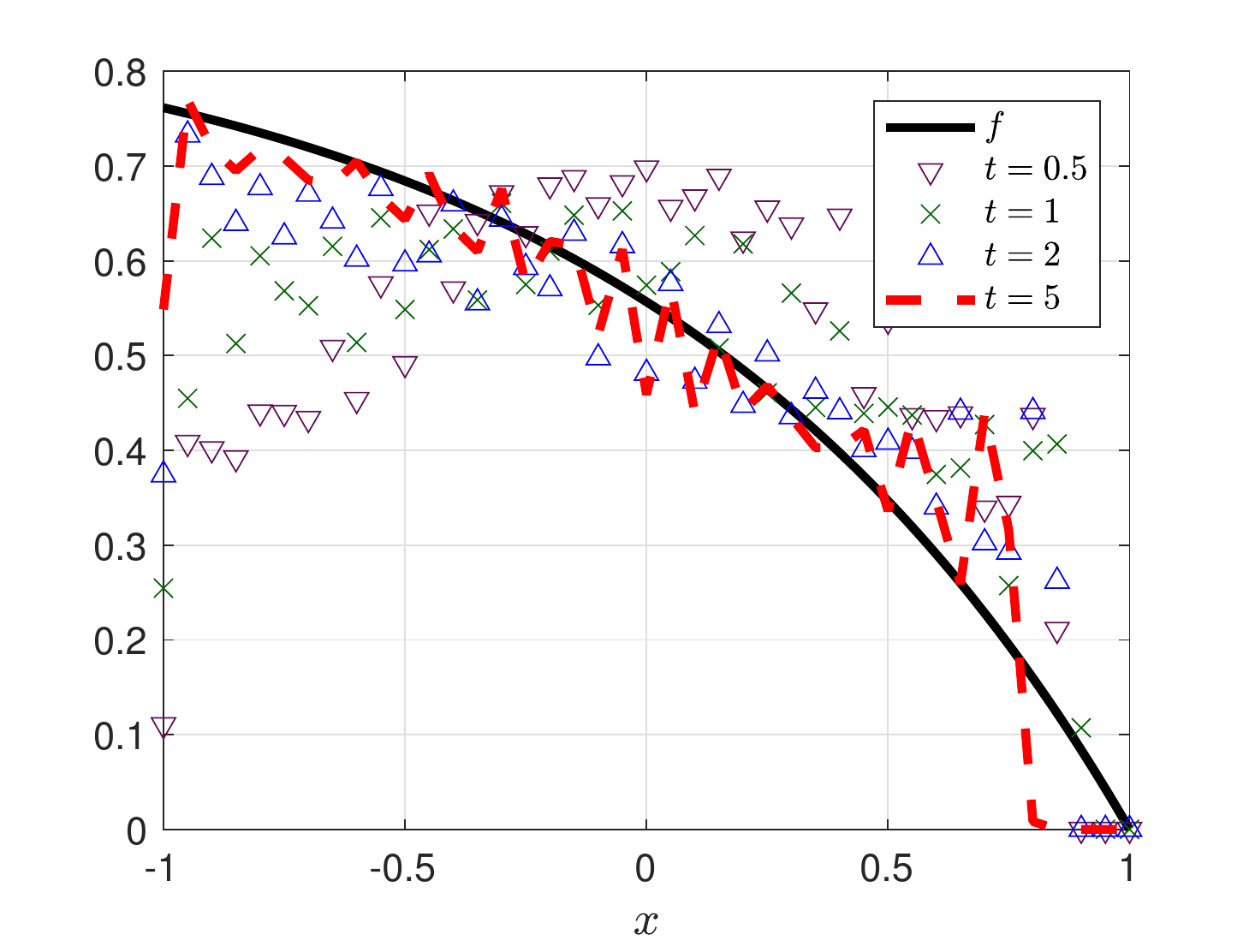}} 
    \subfloat[Normalized risk, $\delta=1/10$.]{\includegraphics[width=.5\columnwidth]{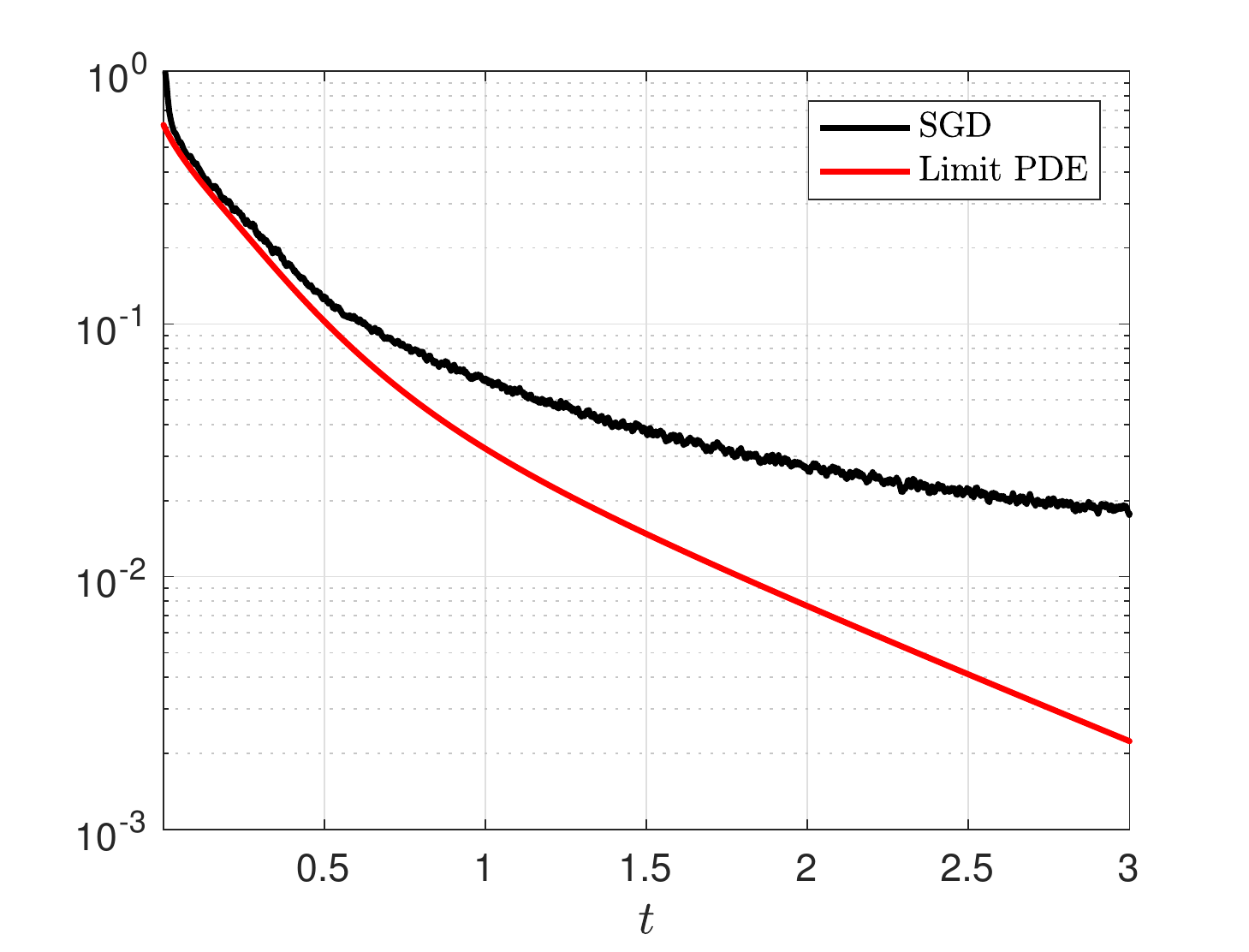}}\\
    \subfloat[Function $f$ and SGD estimates, $\delta=1/20$.]{\includegraphics[width=.5\columnwidth]{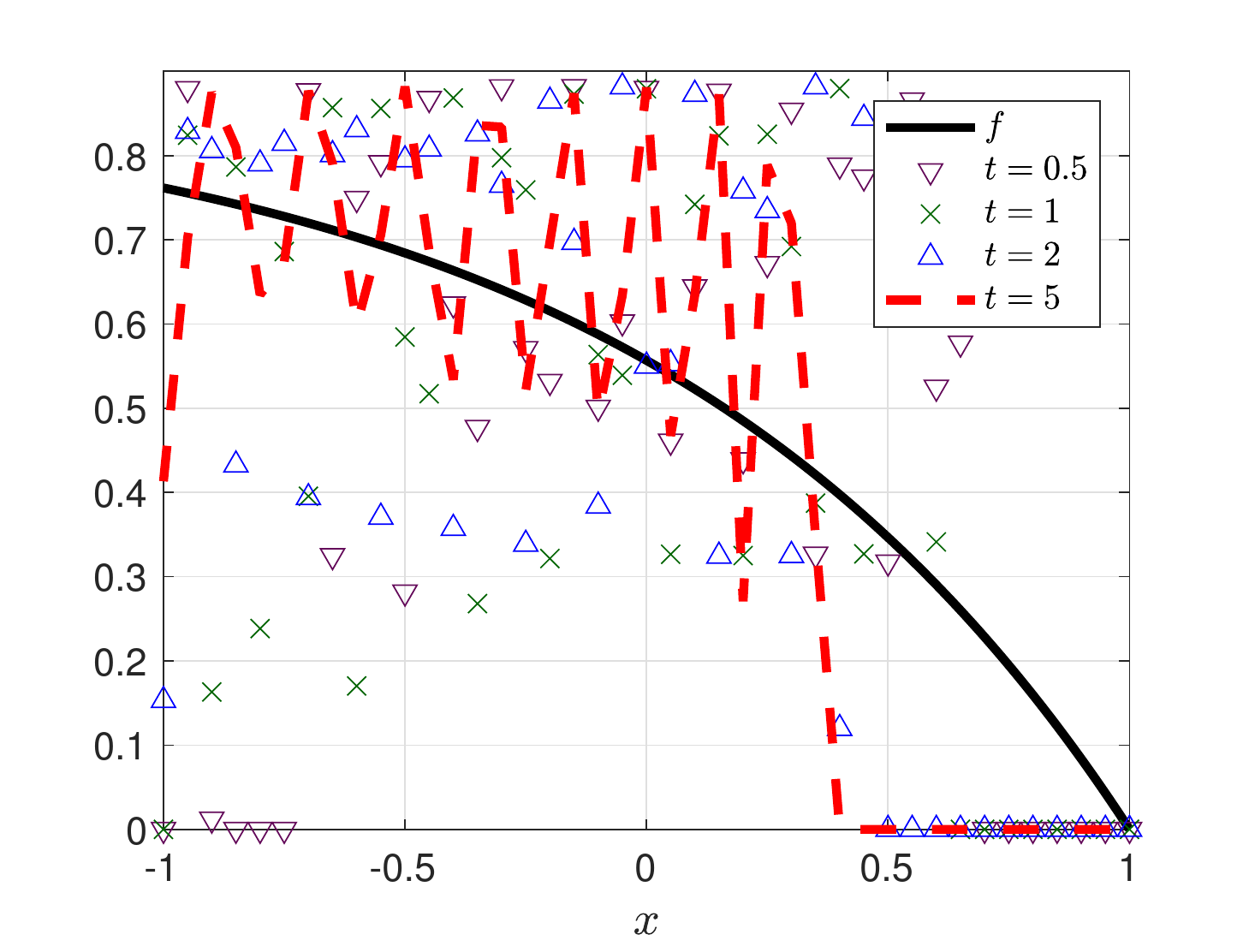}}
    \subfloat[Normalized risk, $\delta=1/20$.]{\includegraphics[width=.5\columnwidth]{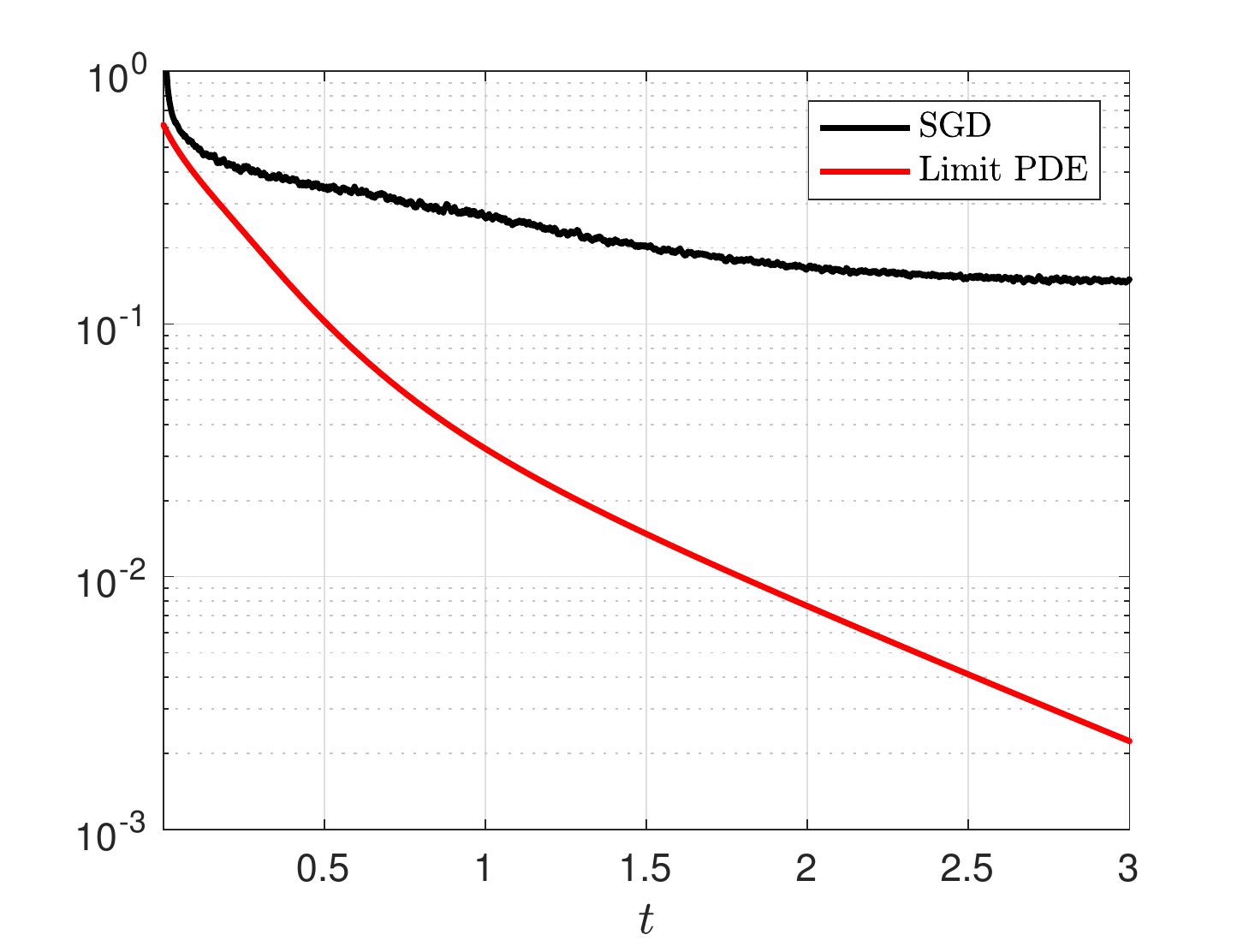}}\\
\caption{Dynamics of SGD update \eqref{eq:SGDup} at different times $t$ and for different values of $\delta$ when the number of neurons is too small ($N=20$).}
\label{fig:SGDN20}
\end{figure}

\begin{figure}[t]
    \centering
\includegraphics[width=.6\columnwidth]{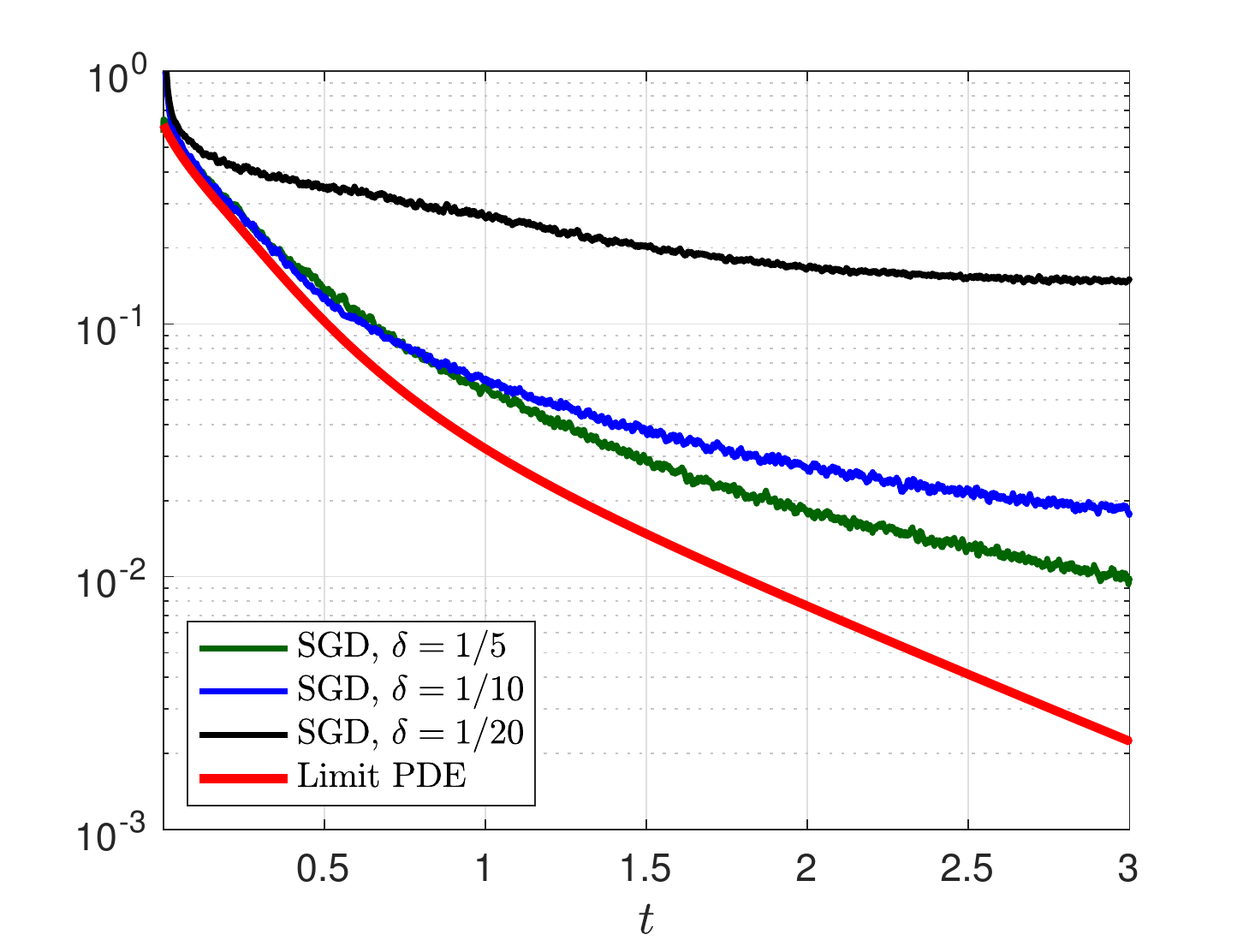}
\caption{Normalized risk of the limit PDE \eqref{eq:PME} and of the SGD update \eqref{eq:SGDup} when the number of neurons is too small ($N=20$).}
\label{fig:PDEN20}
\end{figure}

\section{Main results}
\label{sec:Main}

\subsection{Convergence of SGD to the PDE \eqref{EQ:PDEDELTA} at $\delta>0$ fixed}\label{subsec:convPDESGD}

We now state our result concerning the convergence of the SGD dynamics \eqref{eq:SGDup} to the PDE \eqref{EQ:PDEDELTA}. 
Note that this result does not require  concavity of $f$. Its proof is presented in Appendix \ref{app:conv}.
\begin{theorem}\label{TH:CONVPDE}
Assume that conditions {\sf (A1)}, {\sf (A3)}-{\sf (A5)} hold. Consider the SGD update \eqref{eq:SGDup} with initialization $(\bw_i^0)_{i\le N}\sim_{\rm i.i.d.}\rho_{\sinit}^\delta$ and constant step size $\varepsilon$. For $t\ge 0$, let $\rho_t$ be the unique solution of the PDE \eqref{EQ:PDEDELTA} with initial and boundary conditions 
\eqref{eq:PDEdelta-BD}, and assume
$\supp(\rho^{\delta}_{\sinit})\subseteq \Ball(\bfzero,r)$ 
Then, for any fixed $t\ge 0$, $\rho^{(N)}_{\lfloor t/\varepsilon \rfloor} \Rightarrow \rho_t$ almost surely along any sequence ($N, \varepsilon=\varepsilon_N$) such that $N\to \infty$, $\varepsilon_N\to 0$.

Furthermore, for any $\delta\le 1$, $T\ge 1$, $\eps\le 1$, $p\in\naturals$, and for 
any $g:\mathbb R^d\to \mathbb R$ with $\|g\|_{\rm Lip}\le 1$, the following happens with probability at least $1-z^{-2p}$,
 \begin{equation}\label{eq:ErrorBoundNInftyThm}
\begin{split}
&\sup_{k\in [0, T/\varepsilon]\cap \mathbb N} \bigg| \sum_{i=1}^N g(\bw_i^k)-\int g(\bw)\rho_{k\varepsilon}(\de \bw)\bigg|\le  z\,\err(N,d,\eps,\delta) \,\,e^{C_*p\delta^{-(d+2)}\, T},\\
&\sup_{k\in [0, T/\varepsilon]\cap \mathbb N} | R_N(\bw^k)-R^\delta(\rho_{k\epsilon})|\le z\,\err(N,d,\eps,\delta) \,\,e^{C_*p\delta^{-(d+2)}\, T},
\end{split}
\end{equation}
 where 
\begin{equation}\label{eq:errNdz}
\err(N,d,\eps,\delta) = \sqrt{\frac{d}{N}}\vee \big(\delta^{-2d-1}r(d^2\eps\log(1/\eps))^{1/4}\big)\, .
\end{equation}
\end{theorem}

Our proof is based on the same approach developed in \cite{mei2018mean}. We prove that solutions of the PDE \eqref{EQ:PDEDELTA} are in correspondence with
distributions over trajectories $(\bX_t)_{t\ge 0}$ in $\Omega$ satisfying the following stochastic differential equation 
\begin{align}
\de\bX_t & =-\nabla \Psi(\bX_t,\rho_t) \, \de t+ \sqrt{2\tau}\,\de\bB_t + \de\bPhi_t\, ,  \label{eq:FirstSDE-Main}
\end{align}
where $(\bB_t)_{t\ge 0}$ is a standard Brownian motion and $\de\bPhi_t$ is the boundary reflection (in the sense of a Skorokhod problem).
The density $\rho_t$ is determined, self consistently, via $\rho_t = {\rm Law}(\bX_t)$. We prove existence and uniqueness of solutions to this problem, 
and refer to the corresponding stochastic process  $(\bX_t)_{t\ge 0}$ as \emph{nonlinear dynamics}.
This in turn implies   existence and uniqueness of the solutions of the PDE \eqref{EQ:PDEDELTA}.

We next construct a coupling between the network weights $(\bw^k_1,\dots,\bw^k_N)\in(\Omega^{\delta})^N$, and $N$ i.i.d. 
trajectories of the nonlinear dynamics $(\bX^t_1,\dots,\bX^t_N)\in(\Omega^{\delta})^N$. Controlling the expected distance in this coupling
yields Theorem~\ref{TH:CONVPDE}.
\begin{remark}
The error term in Eq.~\eqref{eq:ErrorBoundNInftyThm} is completely analogous to the error in a similar theorem proved  in \cite{mei2018mean}. The constant 
$\delta^{-d}$ appearing here is obtained by bounding  the Lipschitz constant of $\nabla\Psi(\bw;\rho)$. 
As already mentioned, 
the main technical difficulty with respect to \cite{mei2018mean} is posed by the Neumann (reflecting) boundary conditions. Indeed,
even if we are given a solution of the PDE  \eqref{EQ:PDEDELTA}, existence and uniqueness of solutions of the Skorokhod problem
\eqref{eq:FirstSDE-Main} is a highly non-trivial fact first established in \cite{tanaka1979stochastic,lions1984stochastic}.
As a consequence, while the main proof idea is similar to the one in \cite{mei2018mean}, its implementation is significantly different.
\end{remark}

\begin{remark}
As discussed in Appendix~\ref{app:conv}, our proof applies to a more general version of the PDE (\ref{EQ:PDEDELTA}) and correspondingly of the SGD dynamics (\ref{eq:SGDup}),
where $\Psi$ takes the form $\Psi(\bw,\rho) = V(\bw)+\int U(\bw,\bw') \, \rho(\de\bw')$, for $V:\Omega\to \reals$, $U:\Omega\times 
\Omega\to \reals$ two smooth functions. The SGD update  (\ref{eq:SGDup}) is generalized as in  \cite{mei2018mean}, and Theorem \ref{TH:CONVPDE}
holds with the terms containing $\delta$ (i.e., $\delta^{-2d-1}$ and $\delta^{-(d+2)}$) replaced by a constant that depends uniquely on $\|\nabla V\|_{\Lp^\infty(\Omega)}$, $\|\nabla U\|_{\Lp^\infty(\Omega\times \Omega)}$,
$\|\nabla^2 V\|_{\Lp^\infty(\Omega)}$, $\|\nabla^2 U\|_{\Lp^\infty(\Omega\times \Omega)}$.
\end{remark}

\subsection{\comm{Convergence to the solutions of porous medium equation}}\label{subs:convPME}

\comm{We next prove that the solution of the PDE \eqref{EQ:PDEDELTA} converges, as $\delta\to 0$, to the unique solution of the porous medium equation
\eqref{eq:PME}. As for Theorem \ref{TH:CONVPDE}, this result does not rely on the concavity assumption for $f$.}
\comm{\begin{theorem}\label{TH:CONVPME}
Assume that conditions {\sf (A1)} and {\sf (A3)}-{\sf (A5)} hold. Denote by $\rho^{\delta}$ the unique solution of the PDE \eqref{EQ:PDEDELTA}  with initial condition $\rho^{\delta}_{0}=\rho_{\sinit}$.
Then 
\begin{enumerate}
\item[$(a)$] The porous medium equation \eqref{eq:PME} admits a weak solution $\rho:(t,\bx)\mapsto \rho_t(\bx)$ with initial and boundary conditions 
\eqref{eq:PME-BC}. Further, this solution is unique  under the additional condition $\rho\in \Lp^4([0,T]\times\Omega)$. 
\item[$(b)$] For almost all $t\in [0, T]$, we have $\rho_t^{\delta}\to\rho_t$ in $\Lp^2(\Omega)$ as $\delta\to 0$.
\end{enumerate}
\end{theorem}}

\comm{While this statement is very natural at a heuristic level, its proof is actually the bulk of our technical work.} 
Similar approximation results have been proved in the past by  Oelschl{\"a}ger, Philipowski, Figalli
\cite{oelschlager2002simulation,philipowski2007interacting,figalli2008convergence}, but they do not apply directly to the present case unless
$f=0$ (also, we have to deal with  different boundary conditions).

Our proof follows a classical compactness argument, generalizing the approach of \cite{figalli2008convergence}. Namely we consider the
sequence of trajectories $(\rho^{\delta}_t)_{t\in [0,T]}$ indexed by the width $\delta$. We prove that that this family is bounded and
equicontinuous in $\Cont([0,T],\cuP_2(\Omega))$, and hence admits converging subsequences $(\rho^{\delta_n}_t)_{t\in [0,T]}\to (\rho_t)_{t\in [0,T]}$.
We next prove that any such converging subsequence converges in $\Lp^2(\Omega\times [0,T])$ and that the limit is  a weak solution
of the porous medium equation \eqref{eq:PME}. Unfortunately, uniqueness of weak solutions of the PME  \eqref{eq:PME} is 
--to the best of our knowledge--  an open problem. 
However, we generalize methods from \cite{oelschlager2002simulation} to show that any subsequential limit is actually in $\Lp^4(\Omega\times [0,T])$,
and prove that the weak solution is unique under this condition. This allows us to conclude that $(\rho^{\delta}_t)_{t\in [0,T]}$ converges to this
unique weak solution $(\rho_t)_{t\in [0,T]}$.

\subsection{Global convergence of SGD}\label{subsec:convSGDmin}

Let us now state the main result of this paper: SGD converges to a model with nearly optimal risk.
\begin{theorem}\label{TH:ENDTOEND}
Assume that conditions {\sf (A1)}-{\sf (A5)} hold, and recall that $\alpha>0$ is the concavity parameter of the function $f$, i.e., 
$\<\by , \nabla^2 f(\bx)\by \> \le -\alpha |\by|^2$ for all $\bx\in\Omega$, $\by\in\reals^d$.

Consider the SGD update \eqref{eq:SGDup} with initialization 
$(\bw_i^0)_{i\le N}\sim_{\rm i.i.d.}\rho_{\sinit}$ and constant step size $\varepsilon$. Assume $\supp(\rho_{\sinit})\subseteq\Ball(\bfzero;r)$.  Then, 
for any $k\le T/\eps$, the following holds with probability at least $1-1/z$, 
\comm{\begin{align}
R_N(\bw^k)& \le R_N(\bw^0)e^{-2\alpha k \varepsilon} + 2\tau\, \Delta'(k, \eps, d) +\Delta(N,\eps,T,d,\delta,z)\, ,\label{eq:MainResult}
\end{align}
where
\begin{align}
&\Delta'(k, \eps, d)=\log|\Omega| -(1-e^{-2\alpha k\eps})S(f)-S(\rho_{\sinit})\,e^{-2\alpha k\eps}\,,\\
&\lim_{\delta\to 0}\lim_{N\to\infty,\eps\to 0}\Delta(N,\eps,T,d,\delta,z) = 0\, .
\end{align}}
\end{theorem}

\comm{\begin{remark}
The error term $2\tau\, \Delta'(k, \eps, d)$ in Eq.~(\ref{eq:MainResult}) is always non-negative. In fact, $\Delta'(k, \eps, d)\ge 0$ as $S(\rho)\le \log|\Omega|$ for any $\rho\in \cP_2(\Omega)$. Furthermore, by applying Jensen's inequality, we have that, for any $\rho\in \cP_2(\Omega)$, 
$$S(\rho)=-\int \rho(\bx) \log\rho(\bx)\, \de\bx\ge -\log\mathbb \int \rho(\bx)^2\de\bx = -2\,\log\|\rho\|_{\Lp^2(\Omega)},$$ which gives the following upper bound 
\begin{equation*}
\Delta'(k, \eps, d)\le \log |\Omega| + 2\,\left|\log\|f\|_{\Lp^2(\Omega)}\right|+2\,\left|\log\|\rho_{\sinit}\|_{\Lp^2(\Omega)}\right|.
\end{equation*}
Recall that $\tau$ controls the variance of the noise, which is added at each step of the SGD algorithm for technical purposes. Thus, we can take $\tau$ sufficiently small so that the term $2\tau \Delta'(k, \eps, d)$ is arbitrarily small.  
\end{remark}}

\begin{remark}
The proof of Theorem \ref{TH:ENDTOEND} provides a somewhat more explicit expression for the error term $\Delta(N,\eps,T,d,\delta,z)$ in Eq.~(\ref{eq:MainResult}).
Namely, for an arbitrary but fixed $p\in \naturals$, 
\begin{align}
\Delta(N,\eps,T,d,\delta,z) &= \Delta_1(N,\eps,T,d,z)+\Delta_2(\delta,T,d)\, ,\label{eq:boundrisk}\\
\Delta_1(N,\eps,T,d,z)& =   \left( \sqrt{\frac{d}{N}}\vee \Big(r\delta^{-2d-1} (d^2\eps\log(1/\eps))^{1/4}\Big)\right)\label{eq:delta111}\\ 
&\phantom{AAAAAAA} \cdot \exp\Big\{\sqrt{2C_*\delta^{-(d+2)}\, T\log (z)}\Big\}\nonumber
\, \\
\lim_{\delta\to 0}& \;\Delta_2(\delta,T,d) = 0\, .
\end{align}
The term $\Delta_1$ bounds the error due to describing the SGD dynamics using the PDE \eqref{EQ:PDEDELTA}.
It vanishes when $N\to \infty$, $\eps\to 0$, under the stated conditions. The term $\Delta_2$ captures the error due to approximating the PDE  \eqref{EQ:PDEDELTA} with the porous medium equation 
\eqref{eq:PME}.
Finally,  the term $e^{-2\alpha k \varepsilon}$ describes the convergence to equilibrium of the solution of the porous medium equation.
\end{remark}
The proof of Theorem \ref{TH:ENDTOEND} is presented in Appendix \ref{app:convglob} and relies crucially on regularity results for the 
PDE \eqref{EQ:PDEDELTA} which are established in Appendix \ref{app:reg}.  

More specifically, the proof is based on three steps, which we spell out once more:
\begin{enumerate}
\item[$(i)$] We  approximate the dynamics of SGD by the PDE \eqref{EQ:PDEDELTA} at $\delta>0$ fixed. In doing so, we incur an error $\Delta_1$
which is controlled using Theorem \ref{TH:CONVPDE}. 
\item[$(ii)$] We approximate the solution $\rho^{\delta}_t$ of the PDE \eqref{EQ:PDEDELTA} at $\delta>0$ using the solution $\rho_{t}$ of the porous
medium equation \eqref{eq:PME}, as stated in Theorem \ref{TH:CONVPME}. 
\item[$(iii)$] We use results from  \cite{carrillo2001entropy,carrillo2003kinetic,carrillo2006contractions} to prove that the latter solution converges exponentially fast to the global optimum, with rate $O(e^{-2\alpha t})$.
\end{enumerate}
Given Theorems \ref{TH:CONVPDE},  \ref{TH:CONVPME}, and the results of \cite{carrillo2001entropy,carrillo2003kinetic,carrillo2006contractions},
this proof is relatively direct. \comm{We emphasize that, unlike  Theorems \ref{TH:CONVPDE},  \ref{TH:CONVPME}, the proof Theorem \ref{TH:ENDTOEND}
relies in a crucial way on our structural assumptions, namely the concavity of $f$, and the structure of the bump-like activation $K_{\delta}(\bx-\bw_i)$. }
\begin{remark}
If we settle for the less ambitious goal of proving global convergence without the explicit dimension-independent rate $e^{-2\alpha k\eps}$, and there are no boundary conditions ($\Omega = \reals^d$),
we can achieve this goal using \cite[Theorem 5]{mei2018mean}. This result guarantees  convergence in a number of SGD steps that potentially depends on $\tau$
(the noise injected in SGD) as well as the dimensions $d$, and the width $\delta$, but does not require to assume strong concavity of $f$. 
On the other hand, numerical experiments are consistent with the conclusion that rates are independent of these parameters, cf. e.g. Fig. \ref{fig:SGD} where dependence
on $\delta$ is explored. 
\end{remark}

\section{Discussion}

{It is instructive to compare the general strategy followed in this paper (and in related work, e.g. \cite{mei2018mean,mei2019mean}) and
the results we obtain, to a more classical approach in theoretical statistics. For the sake of clarity, we will abstract away most of the details of 
the present problem, and focus on the most important differences.}

\comm{Consider a general setting in which we want to minimize the population risk $R(\bw) = \E_{y,\bx} L(\bw;y,\bx)$, where $L$ is a non-convex loss function
and $\bw\in\reals^D$ are parameters (in our problem $\bw=(\bw_1,\dots,\bw_N)$ are the first-layer weights and $D=dN$). We are given $n$ i.i.d.
samples $\{(y_j,\bx_j)\}_{j\le n}$.}

\comm{A standard theoretical analysis of this problem uses empirical risk minimization. Namely, we define the empirical risk 
$\hR_n(\bw)=  \hE_{y,\bx} L(\bw;y,\bx)$ (with $\hE_n$ denoting the empirical average), and compute the minimizer
$\hbw_n \in \arg\min_{\bw}\hR_n(\bw)$, for instance by gradient descent. Theoretical analysis proceeds --conceptually--
in two steps. First, one proves that the empirical risk minimizer  is a near-minimizer of the population risk. Namely
\begin{align}
R(\hbw_{n})\le \min_{\bw}R(\bw) +\err(D,n)\, .
\end{align}
This is normally proved through a uniform convergence argument to establish a bound $\sup_{\bw}|\hR_n(\bw)-R(\bw)|\le \err(D,n)/2$.
Here $\err(D,n)$ is an error term that (hopefully) vanishes as $n\to\infty$ for $D$ fixed. Second, one proves that gradient descent 
(with respect to the cost function $\hR_n$) converges to a minimizer $\hbw_n$. 
This is achieved by showing that, with high probability, the landscape $\bw\mapsto \hR_n(\bw)$ satisfies some strong conditions that guarantee convergence of gradient 
descent (or other algorithms). For instance, one desirable (although not sufficient)  property is that $\hR_n$ does not have local minima
other than the global minima, provided that the sample size is large enough.
A substantial literature applies this general scheme (with significant refinements) to a variety of non-convex problems in high-dimensional statistics, including
phase retrieval, clustering, matrix completion, error-in-variables models, and so on. We refer to \cite{mei2018landscape} for examples and a more
detailed survey.}

\comm{Unfortunately this approach runs into substantial difficulties when treating complex models such as multi-layer neural networks.
We can name at least two sources of difficulties. First of all, the number of parameters $D$ in the model is often comparable with  the sample size
$n$, and therefore uniform convergence of the empirical risk to population risk does not hold. For instance, in the present model, we could use
a number of parameters  $Nd\gtrsim n$: indeed, such an example is considered in Figure \ref{fig:SGDnew}-(a), where $Nd = 800$ and $n\in \{100,\dots, 2000\}$.
Of course this problem can be addressed by constraining other measures of complexity than the number of parameters \cite{bartlett1998sample}, 
but the common practice is not to add such regularizers in the training.}

\comm{The second source of difficulties is that studying the risk landscape, and ruling out local minima is extremely difficult, even if we limit ourselves to the $n=\infty$
limit, i.e. the population risk $R(\bw)$. In two-layers neural networks, part of this difficulty is due to the fact that the risk \eqref{eq:poprisk} is invariant under permutations
of the $N$ neurons, and hence it has (generically) at least $N!$  global minima related by permutations, and a large number of saddle points connecting them.}

\comm{The approach pursued in this paper builds on two simple remarks, which are connected to the previous difficulties:}
\begin{enumerate}
\item[\comm{$(i)$}] \comm{Uniform convergence of the empirical risk $\hR_n(\bw)$ to the population risk $R(\bw)$ is not necessary, 
nor it is necessary to control the random deviations of the whole landscape  of the empirical risk. What is instead important is to control the landscape of the
 empirical risk along the trajectory of gradient descent from a given initialization.}

\comm{A convenient way to implement this idea is to consider SGD  in a one-pass setting in which each sample is used 
only once. In the limit of small step size, this converges to gradient flow with respect to $R(\bw)$.}
\item[\comm{$(ii)$}] \comm{Absence of local minima in the population landscape $R(\bw)$ is not necessary either. What is instead important is absence of
local minima along the gradient flow trajectory for $R(\bw)$ or, more precisely, the fact that the gradient flow trajectory  converges to a global minimum.}
\end{enumerate}

\comm{These remarks suggest the following proof strategy. Let $\bw(t)$ denote the gradient flow trajectory
from a given initialization $\bw(0) = \bw_0$ (namely $\dot{\bw}(t) = -\nabla R(\bw(t))$), and $\bw^{k}$ be the (random) parameters produced after 
$k$ SGD steps. We first prove that gradient flow converges to a global optimum, possibly with explicit convergence rate $\Delta(t)$:
\begin{align}
R(\bw(t)) \le \min_{\bw} R(\bw)+ \Delta(t)\, ,\label{eq:GeneralGflow}
\end{align}
where $\Delta(t)\to 0$ as $t\to\infty$.
We then show that the SGD trajectory, after $k$ steps, is well approximated by the gradient flow for $R(\bw)$ provided the step size $\eps$ is small.
For instance we might prove that there exists a numerical constant $c_0$ such that, for any $k\eps\le T$, with high probability
\begin{align}
\big| R(\bw^k)-R(\bw(k\eps))\big| \le \eps^{c_0}\,  \err(T)\, . \label{eq:GeneralApprox}
\end{align}
The reader might recognize that the last estimate is analogous to the one obtained in Theorem \ref{TH:CONVPDE}, while 
the estimate \ref{eq:GeneralGflow}  is what we obtain from displacement convexity (after taking the limit $\delta\to 0$ using Theorem
\ref{TH:CONVPME}). Putting the two estimates together, and recalling that we can run a total of $n$ SGD steps (in the one-pass setting), we get
\begin{align}
R(\hbw) \le \min_{\bw} R(\bw)+ \Delta(n\eps)+ \eps^{c_0}\,  \err(n\eps) \, ,
\end{align}
where we set $\hbw=\bw^k$. The error is reminiscent of a bias-variance tradeoff: the first term is a bias due to early stopping; the second is 
instead the stochastic approximation error. We can now optimize $n$ as to minimize this error. For instance, if $\Delta(t) = e^{-c_1 t}$, and $\err(T) = e^{c_2T}$,
we can choose $\eps\propto (\log n/n)$, yielding  $R(\hbw) \le \min_{\bw} R(\bw)+ C (\log n)^{c_0}/n^{c'}$ where $c' = c_0c_1/(c_1+c_2)$.}

\comm{In summary, within the present approach, the generalization error is bounded via a tradeoff between the convergence rate of
gradient flow in the population risk, and the error of approximating the gradient flow by SGD. A side benefit of this proof strategy
is that it guarantees the existence of an efficient algorithm to compute the weights $\hbw$.}

\comm{As mentioned, the above discussion omits several challenges that are posed by the model treated in this paper. 
Most notably: $(1)$~We are trying to optimize $N$ weight vectors $\bw_1,\dots,\bw_N\in\reals^d$,  but the loss only depends on the empirical distribution
of these vectors $\hrho^{(N)} = N^{-1}\sum_{i=1}^N\delta_{\bw_i}$. It is therefore natural to define a gradient flow in the space of probability distributions,
which is nothing but the PDE \eqref{EQ:PDEDELTA}.
This also help addressing the challenge posed by by the fact that, as $N$ increases, the dimension of the parameter space increases
and convergence to the population behavior might fail. We are embedding all the values of $N$ in the space $\cuP(\reals^d)$.
$(2)$ We  cannot prove a bound of the form \eqref{eq:GeneralGflow} for the original PDE \eqref{EQ:PDEDELTA} and have to approximate
this by the porous medium equation \eqref{eq:PME}.}

\comm{Because of these additional challenges, our bounds are not nearly as neat as in Eqs.~\eqref{eq:GeneralGflow}, \ref{eq:GeneralApprox} and
depend on the additional parameters $d,\delta$: in particular, the approximation by the porous medium equation in Theorem \ref{TH:CONVPME}
is non-quantitative.
 We therefore refrain from optimizing the tradeoff between convergence rate of gradient flow, and error in stochastic approximation, which would result in 
suboptimal statistical guarantees, and defer this objective to future work.}

\section*{Acknowledgements}
A.~Javanmard was partially supported by an Outlier Research in Business (iORB) grant from the USC Marshall School of Business, a Google Faculty Research award and the NSF CAREER award DMS-1844481.  
M.~Mondelli was supported by an Early Postdoc.Mobility fellowship from the Swiss National Science Foundation and by the Simons Institute for the Theory of Computing. 
A.~Montanari was partially supported by grants NSF DMS-1613091,  CCF-1714305,
IIS-1741162 and ONR N00014-18-1-2729.
This work was carried out  in part while the authors were visiting the Simons Institute for the Theory of Computing.
\newpage

\appendix

\section{Uniqueness of weak solutions of limit PDE ($\delta=0$)}\label{app:PDEdelta0}

In this appendix, we prove that the limit PDE obtained for $\delta\to 0$, namely the porous medium equation 
(\ref{eq:PME}) has at most one solution in $\Lp^4(\Omega\times [0,T])$. Existence of such solutions will follow from the results 
of Appendix \ref{app:convglob}, and in particular from Lemma \ref{lemma:weaksol}.

For the sake of clarity, we repeat the definitions of Section \ref{subsec:PDEdelta0}. Let $\Omega\subseteq\reals^d$ be a compact convex set with $\Cont^2$ boundary. 
 We denote by $\cuP_2(\Omega)$ the space of probability measures on $\Omega$
endowed with Wasserstein's $W_2$ distance. Since $\Omega$ is compact, the induced topology is equivalent to weak convergence. We consider the following PDE:
\begin{align}
\partial_t \rho_t(\bw) &= -\nu_0\nabla\cdot\big(\rho_t(\bw)  \nabla f(\bw)\big) +
\frac{\nu_0}{2}\Delta(\rho_t^2(\bw)) +\tau\Delta \rho_t(\bw)\, ,\label{eq:PMEbis}
\end{align}
with initial and boundary conditions
\begin{equation}\label{eq:PME-BCbis}
\begin{split}
\rho_0&=\rho_{\sinit},\\
\Big\<\bn(\bw), \nu_0\rho_t(\bw)\, \nabla (f(\bw)-\rho_t(\bw))& - \tau\nabla\rho_t(\bw)\Big\> = 0\, \;\;\;\;\; \forall \bw\in\partial\Omega\, .\\
\end{split}
\end{equation}
Throughout this appendix, we adopt the notation $\Phi(\rho) = \tau\rho +\nu_0\,\rho^2/2$. Let us formally define the concept of weak solutions for the PDE \eqref{eq:PMEbis}.

For the next statement, it is useful to recall that $\Cont^{2,1}(\Omega\times[0,T])$ denotes the class of functions $f:\Omega\times[0,T]\to\reals$
with continuous partial derivatives $D_tf(\bx,t)$, $D_{\bx}^{\balpha} f(\bx,t)$ for all $\|\balpha\|_2\le 2$.
\begin{definition}[Weak solution of limit PDE]\label{def:weaklimitPDE}
We say that $\rho\in \Cont([0,T],$ $\cuP_2(\Omega))$ is a \emph{weak solution} of the PDE \eqref{eq:PMEbis},
with initial and boundary conditions \eqref{eq:PME-BCbis} if 
\begin{enumerate}
\item $\rho_t$ has density $\rho(\,\cdot\,,t)$ with respect to Lebesgue  measure,
and $\rho \in \Lp^2(\Omega\times[0,T])$.
\item For any test function $h\in \Cont^{2,1}(\Omega\times [0,T])$,
satisfying $\<\bn(\bx),\nabla h(\bx,t)\>=0$ for all $\bx\in\partial\Omega, t\in[0,T]$, 
we have 
\begin{align}
\int_{\Omega}&h(\bx,T)\,\rho(\bx,T)\, \de\bx - \int_{\Omega}h(\bx,0)\,\rho_{\sinit}(\bx)\, \de\bx  \label{eq:WeakPME_Def}\\
&= 
\int_{0}^T\int_{\Omega}\Big[\partial_th(\bx,t)+\nu_0\<\nabla f(\bx),\nabla h(\bx,t)\>
\Big] \,\rho(\bx, t)\, \de\bx\,\de t \nonumber\\
&+ \int_{0}^T\int_{\Omega} \Delta h(\bx,t)  \Phi(\rho)(\bx,t)\, \de \bx\,\de t\, . \nonumber
\end{align}
\end{enumerate}
\end{definition}

We now prove a uniqueness result, under a mild integrability condition. 

\begin{lemma}[Uniqueness of limit PDE]\label{lemma:unique}
\comm{Let $\rho, \tilde{\rho} \in \Lp^4(\Omega\times[0,T])$ be two weak solutions of the PDE~\eqref{eq:PMEbis}
with initial and boundary conditions \eqref{eq:PME-BCbis}, in the sense of Definition \ref{def:weaklimitPDE}.} Then, $\rho=\tilde{\rho}$, almost everywhere.
\end{lemma}

\begin{proof}
Note that setting $\nu_0=1$ corresponds to scaling time by a factor $\nu_0$ and to substituting $\tau$ with $\tau\, \nu_0$. Since the proof holds for any $\tau>0$, without loss of generality we can set $\nu_0=1$. 

The proof follows ideas from \cite[Theorem 6.5]{vazquez2007porous}. We write the identity~\eqref{eq:WeakPME_Def} for $\rho$ and $\tilde{\rho}$ and subtract them to get
\begin{align}
\int_{\Omega}&h_T(\bx)\,(\rho_T(\bx) -\tilde{\rho}_T(\bx))\, \de\bx  \label{eq:WeakPME_Def-2}\\
&= 
\int_{0}^T\int_{\Omega}\Big[\partial_th_t(\bx)+\<\nabla f(\bx),\nabla h_t(\bx)\>
\Big] \,(\rho_t(\bx)-\tilde{\rho}_t(\bx))\, \de\bx\,\de t \nonumber\\
& + \int_{0}^T\int_{\Omega} \Delta h_t(\bx)  (\Phi(\rho_t(\bx))-\Phi(\tilde{\rho}_t(\bx)))\, \de \bx\,\de t\, , \nonumber
\end{align}   
where we use the shorthand $\rho_t(\bx)\equiv \rho(\bx,t)$ and $h_t(\bx)\equiv h(\bx,t)$. Define $u_t = \rho_t-\tilde{\rho}_t$ and $\eta_t = \tau+(\rho_t+\tilde{\rho}_t)/2$. Then, 
\begin{align}
\int_{\Omega}h_T(\bx)\,u_T(\bx)\, \de\bx  &= 
\int_{0}^T\int_{\Omega}\Big[\partial_th_t(\bx)+\<\nabla f(\bx),\nabla h_t(\bx)\>
\Big] \,u_t(\bx)\, \de\bx\,\de t \label{eq:WeakPME_Def-3}\\
& + \int_{0}^T\int_{\Omega} \Delta h_t(\bx) \eta_t(\bx) u_t(\bx)\, \de \bx\,\de t\, . \nonumber
\end{align}
Note that $\eta_t(\bx)\ge \tau$ and define the truncated function $\eta_t^M = \min(M, \eta_t)$. We next choose a smooth test function 
$\theta:\Omega\times [0,T]\to \reals_{\ge 0}$, $(\bx,t)\mapsto \theta_t(\bx)$ and consider the following backward problem:
\begin{equation}\label{eq:parabolic}
\left\{\hspace{-.5em}\begin{array}{cc}
\partial_th_t(\bx) + \hat{\eta}_t(\bx)  \Delta h_t(\bx) +\<\nabla f(\bx), \nabla h_t(\bx)\> +\theta_t(\bx) =0\,,  &\hspace{-.5em}\forall \bx\in \Omega, t\in [0,T],\\
\<\bn(\bx),\nabla h_t(\bx)\> = 0\,, & \hspace{-1.5em}\forall \bx\in \partial\Omega, t\in [0,T],\\
h_T(\bx) = 0\,, &\hspace{-1em}\forall \bx\in \Omega\,.
\end{array}\right.
\end{equation}
Here, $\hat{\eta}_t$ is a smooth approximation of $\eta_t^M$, such that $\tau\le \hat{\eta}_t(\bx)\le M$. 
(We will make precise below in what sense $\hat{\eta}_t$ has to approximate $\eta_t^M$. For the moment, it can be a general smooth function
satisfying the bounds $\tau\le \hat{\eta}_t(\bx)\le M$.)
Note that~\eqref{eq:parabolic} is a backward parabolic problem with smooth coefficients and with Neumann boundary conditions. Hence, by classical results on quasilinear parabolic PDEs \cite{ladyzhenskaia1988linear}, it admits a solution $h_t\in \Cont^{2,1}(\Omega\times[0,T] )$.  Rewriting~\eqref{eq:WeakPME_Def-3} for such a test function $h_t$, we get
\[
\int_0^T\int_\Omega \Delta h_t(\bx) (\eta_t(\bx) - \hat{\eta}_t(\bx)) u_t(\bx)\,\de \bx\, \de t = \int_0^T \int_\Omega \theta_t(\bx) u_t(\bx)\,\de \bx\, \de t \,.
\] 
This immediately implies that
\begin{align}\label{eq:unique-0}
\int_0^T \int_\Omega \theta_t(\bx) u_t(\bx)\de \bx \de t \le \int_0^T\int_\Omega |\Delta h_t(\bx)|\, |\eta_t(\bx) - \hat{\eta}_t(\bx)| |u_t(\bx)| \,\de\bx \,\de t \,.
\end{align}
By applying Cauchy-Schwarz inequality, we have that
\begin{align}\label{eq:unique-2}
\int_0^T\int_\Omega & |\Delta h_t(\bx)|\, |\eta_t(\bx) - \hat{\eta}_t(\bx)| \, |u_t(\bx)|\, \de\bx \,\de t\nonumber\\
\le& \left(\int_{\Omega\times [0,T]} \hat{\eta}_t(\bx)\, (\Delta h_t(\bx))^2 \,\de \bx \,\de t\right)^{1/2}\nonumber\\
& \left(\int_{\Omega\times [0,T]} \frac{|\eta_t(\bx) - \hat{\eta}_t(\bx)|^2}{\hat{\eta}_t(\bx)} u_t^2(\bx)\, \de \bx\, \de t\right)^{1/2}\,.
\end{align}

To bound the first term on the right-hand side of~\eqref{eq:unique-2}, we consider a smooth positive bounded function $\mu(t)$, defined on $[0,T]$, whose properties will be discussed later. 
Define the shorthand $\tilde{\theta}_t(\bx)\equiv \theta_t(\bx) + \<\nabla f(\bx), \nabla h_t(\bx)\>$.
We multiply the parabolic PDE~\eqref{eq:parabolic} by $\mu(t)\Delta h_t(\bx)$ and integrate to obtain
\begin{align}\label{eq:unique-3}
\int_{\Omega\times[0,T]} \mu(t) \partial_t h_t(\bx) \Delta h_t(\bx) \de\bx\de t &+ \int_{\Omega\times[0,T]}  \mu(t)\hat{\eta}_t(\bx) (\Delta h_t(\bx))^2\de\bx\de t\nonumber \\
&+ \int_{\Omega\times[0,T]}  \mu(t) \tilde{\theta}_t(\bx) \Delta h_t(\bx) \de\bx\de t = 0.
\end{align}
We next write 
\begin{align}
&\int_{\Omega\times[0,T]} \mu(t) \partial_t h_t(\bx) \Delta h_t(\bx) \de\bx\de t \nonumber\\
 &\stackrel{(a)}{=}  - \int_{\Omega\times[0,T]} \mu(t) \<\nabla h_t(\bx), \nabla (\partial_t h_t(\bx)\> \de\bx\de t\nonumber\\
& =  - \int_{\Omega\times[0,T]} \mu(t) \frac{1}{2}\frac{\de}{\de t} |\nabla h_t(\bx)|^2 \de\bx\de t\nonumber\\
& \stackrel{(b)}= \frac{1}{2}\int_\Omega \mu(0)|\nabla h_0(\bx)|^2\de \bx- \frac{1}{2}\int_\Omega \mu(T)|\nabla h_T(\bx)|^2\de\bx\nonumber\\
&\hspace{11em}+  \frac{1}{2}\int_{\Omega\times[0,T]} \mu'(t)  |\nabla h_t(\bx)|^2 \de\bx\de t\nonumber\\
&\stackrel{(c)}\ge \frac{1}{2}\int_{\Omega\times[0,T]} \mu'(t)  |\nabla h_t(\bx)|^2 \de\bx\de t
\end{align}
 Here $(a)$ follows from integration by parts in the integral over $\Omega$ and using the fact that $\<\bn(\bx),\nabla h_t(\bx)\> = 0$ for $\bx\in \partial \Omega$ and $t\in [0,T]$. Also, $(b)$ follows from integration by parts in the integral over $t$.
 Finally $(c)$ holds because $h_T(\bx) = 0$ for $\bx\in \Omega$ and  $\mu(0) \ge 0$. 
 
 Getting back to~\eqref{eq:unique-3} and using the properties of function $\mu(t)$, we have
 \begin{align}
 &\frac{1}{2}\int_{\Omega\times[0,T]} \mu'(t)  |\nabla h_t(\bx)|^2 \de\bx\de t + \int_{\Omega\times[0,T]} \mu(t)  \hat{\eta}(\bx) (\Delta h_t(\bx))^2\de\bx\de t\nonumber\\
 & \le -\int_{\Omega\times[0,T]}  \mu(t) \tilde{\theta}_t(\bx) \Delta h_t(\bx) \de\bx\de t  \nonumber\\
 & = -\int_{\Omega\times[0,T]}  \mu(t) {\theta}_t(\bx) \Delta h_t(\bx) \de\bx\de t  \nonumber\\
 &\hspace{2em}-  \int_{\Omega\times[0,T]}  \mu(t) \<\nabla f(\bx), \nabla h_t(\bx)\> \Delta h_t(\bx) \de\bx\de t   \nonumber\\
 & = \int_{\Omega\times[0,T]}  \mu(t) \<\nabla\theta_t(\bx), \nabla h_t(\bx)\> \de\bx\de t\nonumber \\
 &\hspace{2em}-  \int_{\Omega\times[0,T]}  \mu(t) \<\nabla f(\bx), \nabla h_t(\bx)\> \Delta h_t(\bx) \de\bx\de t \nonumber\\
 &\le\int_{\Omega\times[0,T]}\hspace{-1em}  \mu(t) \<\nabla\theta_t(\bx), \nabla h_t(\bx)\> \de\bx\de t + \int_{\Omega\times[0,T]}  \frac{\mu(t)}{2\tau} |\nabla f(\bx)|^2 |\nabla h_t(\bx)|^2\de\bx \de t\nonumber\\
 & + \int_{\Omega\times[0,T]} \mu(t) \frac{\tau}{2} (\Delta h_t(\bx))^2 \de\bx\de t 
 \,,\label{eq:unique-4}
 \end{align}
The penultimate step follows from integration by parts and the constraint $\<\nabla h_t(\bx), \bn(\bx)\> = 0$, for $\bx\in \partial \Omega$ and $t\in [0,T]$, and the last step follows by applying 
Cauchy-Schwartz inequality. We continue by applying Cauchy-Schwartz inequality  again to get
 \begin{align}
  &\int_{\Omega\times[0,T]}  \mu(t) \<\nabla\theta_t(\bx), \nabla h_t(\bx)\> \de\bx\de t\nonumber\\
  &\le \int_{\Omega\times[0,T]}\bigg[\frac{\tau\mu(t)}{2C^2}  |\nabla \theta_t(\bx)|^2+ \frac{C^2 \mu(t)}{2\tau} |\nabla h_t(\bx)|^2 \bigg] \de\bx \,\de t\,,\label{eq:unique-5}
   \end{align}
   where $C = \sup_{\bx\in \Omega} |\nabla f(\bx)|$.
   Combining Equations~\eqref{eq:unique-4} and \eqref{eq:unique-5}, we get
   \begin{align}\label{eq:unique-6}
   &\frac{1}{2}\int_{\Omega\times[0,T]} \left(\mu'(t) - \frac{\mu(t)}{\tau} \left(|\nabla f(\bx)|^2 +C^2\right) \right)  |\nabla h_t(\bx)|^2 \de\bx\de t\nonumber\\
   & +\int_{\Omega\times[0,T]} \mu(t) \left(\hat{\eta}(\bx)- \frac{\tau}{2}\right) (\Delta h_t(\bx))^2\de\bx\de t \le \frac{\tau\mu_{\max}}{2C^2} \int_{\Omega\times[0,T]}  |\nabla \theta_t(\bx)|^2\de\bx\de t\,,
   \end{align} 
 where $\mu_{\max} = \sup_{t\in[0,T]} \mu(t)$.
We find a smooth function $\mu(t)$ such that
 \begin{enumerate}
 \item $\mu(t) \ge \mu_{\min}>0$, for $t\in [0,T]$,
 \item $\mu'(t) - \frac{2C^2}{\tau} \mu(t) \ge 0$.
 \end{enumerate}
 A particular choice is 
 \[
 \mu(t) = \mu_{\min}e^{\frac{2C^2}{\tau}t}.
 \]
 
 We then obtain from~\eqref{eq:unique-6} that
 \begin{align}\label{eq:unique-7}
 \int_{\Omega\times[0,T]}  \hat{\eta}(\bx) (\Delta h_t(\bx))^2\de\bx\de t\le \frac{\tau \mu_{\max}}{\mu_{\min}C^2} \int_{\Omega\times[0,T]}  |\nabla \theta_t(\bx)|^2\de\bx\de t\, .
 \end{align}
   
   Now by employing~\eqref{eq:unique-7} in bound \eqref{eq:unique-2} combined with \eqref{eq:unique-0} we get
   \begin{align}
\int_0^T \int_\Omega \theta_t(\bx) u_t(\bx)\de \bx \de t\le & \frac{1}{C}\sqrt{\frac{\tau\mu_{\max}}{\mu_{\min}}}  \|\nabla \theta\|_{\Lp^2(\Omega\times[0,T])} \nonumber\\
&\left(\int_{\Omega\times [0,T]} \frac{|\eta_t(\bx) - \hat{\eta}_t(\bx)|^2}{\hat{\eta}_t(\bx)} u_t^2(\bx)\, \de \bx\, \de t\right)^{1/2}\,.\label{eq:unique-8}
\end{align}
Next we note that
\begin{equation*}
\begin{split}
&\int_{\Omega\times[0,T]} |\eta_t(\bx) - \hat{\eta}_t(\bx)|^2 u_t^2(\bx) \de \bx \de t \\
&\le 2 \int_{\Omega\times[0,T]} |\eta^M_t(\bx) - \hat{\eta}_t(\bx)|^2 u_t^2(\bx) \de \bx \de t \\
&+ 2 \int_{\Omega\times[0,T]} (({\eta}_t(\bx)- M)_+)^2 u_t^2(\bx) \de \bx \de t.
\end{split}
\end{equation*}
Call the first integral $I_1$ and denote the second one by $I_2$.
The integrand in $I_2$ is pointwise bounded by 
\[
2\eta_t^2(\bx) u_t^2(\bx)  \,\mathbb{I}(\eta_t(\bx)> M ) \le
2(\Phi(\rho_t(\bx)) - \Phi(\tilde{\rho}_t(\bx)))^2 \,\mathbb{I}(\eta_t(\bx)> M ) \,.
\]
Since $\rho_t, \tilde{\rho}_t \in \Lp^4$, we have that $(\Phi(\rho_t) - \Phi(\tilde{\rho}_t))^2$ has bounded integral. Hence, we can choose $M$
large enough such that  $I_2$ is arbitrarily small. Moreover we can choose the smooth approximation $\hat{\eta}_t$ such that $I_1$ is also arbitrarily small. Putting everything together, we obtain that
\[\int_{\Omega\times[0,T]} |\eta_t(\bx) - \hat{\eta}_t(\bx)|^2 u_t^2(\bx) \de \bx \de t\le \eps\,, \]   
where $\eps$ is an arbitrary small fixed constant. 

In addition, since $\hat{\eta}_t(\bx) \ge \tau$, invoking~\eqref{eq:unique-8} we have
\[
\int_0^T \int_\Omega \theta_t(\bx) u_t(\bx)\de \bx \de t\le  \frac{1}{C}\sqrt{\frac{\mu_{\max}}{\mu_{\min}}}  \|\nabla \theta\|_{\Lp^2(\Omega\times[0, T])}  \sqrt{\eps}\,.
\] 
Since $\frac{\mu_{\max}}{\mu_{\min}} = e^{\frac{2C^2}{\tau}T}<\infty$ and $\theta$ are independent of $\eps$, by choosing $\eps$ arbitrarily small, we conclude that
\[
\int_0^T \int_\Omega \theta_t(\bx) u_t(\bx)\de \bx \de t\le 0 \,.
\] 
Since $\theta_t(\bx) \ge 0$ was an arbitrary smooth function supported on $\Omega\times [0,T]$, this implies that $u\le 0$, almost everywhere. By repeating a similar argument, we get
$u\ge 0$, almost everywhere. The result follows.
\end{proof}

\section{General results on the PDE \eqref{EQ:PDEDELTA} ($\delta>0$)}\label{app:genPDE}

This appendix contains some basic results on the PDE (\ref{EQ:PDEDELTA}). Although these facts are standard, we collect them here for the reader's convenience. 

In fact, we will consider a more general PDE, which also includes as a special case the one studied in \cite{mei2018mean}.
We consider a compact convex domain $D$, with a non-empty interior.
 The general PDE is parametrized by two functions
$V\in \Cont^{2}(D)$ and $U\in \Cont^{2}(D\times D)$, with $U(\bx_1,\bx_2)=U(\bx_2,\bx_1)$.
(Unlike in \cite{mei2018mean},  we consider the case of a compact domain with Neumann boundary conditions.)
Given $\rho\in\cuP_2(D)$, 
we define
\begin{align}
\Psi(\bx,\rho)\equiv V(\bx)+\int U(\bx,\tbx)\, \rho(\de\tbx) \, ,
\end{align}
and consider the PDE
\begin{align}
\partial_t\rho(\bx,t) = \nabla \cdot(\rho(\bx,t)\nabla\Psi(\bx,\rho_t))+ \tau\, \Delta\rho(\bx,t)\, ,
\label{eq:PDEbis}
\end{align}
with initial and boundary conditions
\begin{equation}\label{eq:BCbis}
\begin{split}
\rho_{0} & = \rho_{\sinit}\, ,\\
\<\bn(\bx),\rho_t(\bx)\nabla\Psi(\bx,\rho_t)&+\, \tau\, \nabla\rho_t(\bx)\> = 0\, ,\;\;\;\;\; \forall \bx\in\partial D\,.\\
\end{split}
\end{equation}
We will typically write  $\rho_t(\,\cdot\,)$ for a solution of this
equation, in order to emphasize that it is a function of $t$ that takes values in $\cuP_2(D)$, and $\rho(\bx,t)$
for the corresponding density, viewed as a function on $D\times [0,T]$. Let us formally define the concept of weak solutions for the PDE \eqref{eq:PDEbis}.

Note that the PDE (\ref{EQ:PDEDELTA}) is a special case of this setting with $D=\Omega^{\delta}$, and $V(\bw)$ and $U(\bw_1,\bw_2) = U(\bw_1-\bw_2)$ defined as follows:
\begin{equation}\label{eq:KtoUV}
\begin{split}
V(\bw) &  \equiv -\nu_0\,K^{\delta}*f(\bw) =  -\nu_0\int K^{\delta}(\bw-\bx)\, f(\bx)\, \de\bx\, ,\\
U(\bw) &  \equiv \nu_0\,K^{\delta}*K^{\delta}(\bw)= \nu_0 \int K^{\delta}(\bw-\bx)\, K^{\delta}(\bx)\, \de\bx\,.
\end{split}
\end{equation}

\begin{remark} For the special choice of $V$ and $U$ given by \eqref{eq:KtoUV} the following properties hold:
\begin{enumerate}\label{rem1}
\item $V:\Omega^\delta\to \reals$ is convex for any $\delta>0$. 

\item $\lim_{\delta\to 0}\sup_{\bw\in\Omega^\delta}|V(\bw)+\nu_0\,f(\bw)|=0$.
\item $U(\bw) = \nu_0\, \delta^{-2d}K^{(2)}(\bw/\delta)$, where $K^{(2)} = K*K$.
\end{enumerate}
\end{remark}
\begin{proof}
We have $V(\bw) = -\nu_0\int K^{\delta}(\bx) f(\bw-\bx) \de\bx$. Hence,
\begin{align*}
V(\lambda \bw + (1-\lambda)\bw') &= -\nu_0\int K^{\delta}(\bx) f(\lambda \bw + (1-\lambda)\bw' -\bx) \de\bx\\
&= -\nu_0\int K^{\delta}(\bx) f(\lambda (\bw-\bx) + (1-\lambda)(\bw' - \bx)) \de\bx\\
&\le -\nu_0\int K^{\delta}(\bx) \left(\lambda f(\bw-\bx) + (1-\lambda) f(\bw'-\bx)\right)\de\bx \\
&= \lambda V(\bw) + (1-\lambda) V(\bw')\,.
\end{align*}
This proves that $V(\bw)$ is convex. The next two properties are straightforward.
\end{proof}

\begin{definition}[Weak solution of PDE]\label{def:weaksolPDE}
We say that $\rho:[0,T]\to\cuP_2(D)$ is a \emph{weak solution} of \eqref{eq:PDEbis} with initial and boundary
conditions \eqref{eq:BCbis} if $\rho\in \Cont([0,T],\cuP_2(D))$ and,
 for any test function $h\in \Cont^{2,1}(D\times [0,T])$,
satisfying $\<\bn(\bx),\nabla h(\bx,t)\>=0$ for all $\bx\in\partial D, t\in[0,T]$,
we have 
\begin{align}
&\int_{D}h(\bx,T)\,\rho_T(\de\bx) - \int_{D}h(\bx,0)\,\rho_{\sinit}(\de\bx)  \label{eq:WeakDef}\\
&= 
\int_{0}^T\int_{D}\left[\partial_th(\bx,t) +\tau\Delta h(\bx,t) -\<\nabla\Psi(\bx,\rho_t),\nabla h(\bx,t)\>
\right] \,\rho_t(\de\bx)\, \de t\, . \nonumber
\end{align}
\end{definition}

We now state and prove Duhamel's principle for the PDE \eqref{eq:PDEbis}. Duhamel's principle follows from the fact that the right-hand side of \eqref{eq:PDEbis} contains the linear diffusion term $\tau \Delta \rho$, and it will be crucial for the proofs that will follow. 
\begin{lemma}[Duhamel's principle]
Assume $\tau>0$. Let $G^{D}(\bx,\by;t)$ denote the heat kernel with Neumann boundary conditions, defined in \eqref{eq:defheatbd1}-\eqref{eq:defheatbd3}. Let $\rho$ be a weak solution of the PDE \eqref{eq:PDEbis} with initial and boundary
conditions \eqref{eq:BCbis}. Then, for any $t>0$, $\rho_t(\de\bx)$
has a density, denoted by $\rho(\;\cdot\;, t)$, which satisfies, for any $t>0$,
\begin{align}
\rho(\bx,t) &= \int_{D} G^{D}(\bx,\by;\tau t)\, \rho_{\sinit}(\de\by) \nonumber\\
&-
\int_{0}^t \int_{D} \<\nabla_{\by} G^{D}(\bx,\by;\tau(t-s)),\nabla_{\by} \Psi(\by;\rho_s)\>\, \rho(\by,s)\, \de\by\,\de s\, . 
\label{eq:Duhamel}
\end{align}
\end{lemma}
\begin{proof}
By rescaling time, without loss of generality, we set $\tau=1$.
Let $\varphi\in \Cont^2(D)$, and define 
\begin{align}
G_{\varphi}(\bx;t) = \int_{D} G^{D}(\bx,\by;t)\, \varphi(\by)\, \de\by\, .
\end{align}
By the properties of the heat kernel, we have:
\begin{align}
&\Big(\partial_t-\Delta\Big) G_{\varphi}(\bx;t) = 0\;\;\;\;\; \forall t>0\, ,\label{heat1}\\
&\<\bn(\bx), \nabla G_{\varphi}(\bx;t) \> =0 \;\;\; \forall \bx\in\partial D\, ,\label{heat2}\\
&\lim_{t\to 0} G_{\varphi}(\bx;t) = G_{\varphi}(\bx;0) = \varphi(\bx)\, .
\end{align}

Let $\rho_t$ be a weak solution. We choose the test function $h(\bx,s) = G_{\varphi}(\bx;t-s)$ in \eqref{eq:WeakDef} with $T=t$. Note that by~\eqref{heat1}, this test function satisfies the 
Neumann boundary condition. In addition, by~\eqref{heat2} we obtain 
\begin{align}
\int_{D}\varphi(\by)\,\rho_{t}(\de\by)& = \int_{D}G_{\varphi}(\bx;t)\,\rho_{\sinit}(\de\bx) \nonumber\\
&-
\int_{0}^{t}\int_{D}\<\nabla\Psi(\bx,\rho_s),\nabla G_{\varphi}(\bx;t-s)\>
 \,\rho_s(\de\bx)\, \de s\, . 
\end{align}
By an application of Fubini's theorem, this implies
\begin{align}
\int_{D}\varphi(\by)\,\rho_{t}(\de\by)& = \int_{D}\int_{D}G^{D}(\bx,\by;t)\,\rho_{\sinit}(\de\bx) \,\varphi(\by)\, \de\by\nonumber\\
&- \int_{0}^{t}\int_{D}\int_{D}\<\nabla\Psi(\bx,\rho_s),\nabla G^{D}(\bx,\by;t-s)\>
 \,\varphi(\by)\,\rho_s(\de\bx) \de s  \, \de\by\, . 
\end{align}
Since $\varphi\in \Cont^{2}(D)$ is arbitrary, we obtain that $\rho_{t}$ admits a density and \eqref{eq:Duhamel} follows.
\end{proof}

As an intermediate step towards proving existence and uniqueness, we consider a linearized problem
\begin{align}
\partial_t\rho(\bx,t) = \nabla \cdot(\rho(\bx,t)\nabla\Psi_*(\bx,t))+ \tau\, \Delta\rho(\bx,t)\, ,
\label{eq:LinPDE}
\end{align}
with initial and boundary conditions
\begin{equation}\label{eq:BCbis2}
\begin{split}
\rho_{0} & = \rho_{\sinit}\, ,\\
\<\bn(\bx),\rho_t(\bx)\nabla\Psi_*(\bx,t)&+\, \tau\, \nabla\rho_t(\bx)\> = 0\, ,\;\;\;\;\; \forall \bx\in\partial D\,.\\
\end{split}
\end{equation}
Here, $\Psi_*:D\times\reals\to\reals$ is independent of $\rho$, and weak solutions are defined as for 
the original problem (with Neumann boundary conditions).
\begin{corollary}[Uniqueness of linearized problem]\label{lemma:UniqLinear}
Assume that $\tau>0$ and also that
$$\|\nabla\Psi_*\|_{\Lp^\infty(D\times [0, T])}\equiv \sup_{t\in [0, T]}\sup_{\bx\in D}|\nabla\Psi_*(\bx,t)| < \infty.$$
Then, the PDE \eqref{eq:LinPDE} with initial and boundary
conditions \eqref{eq:BCbis2} has at most one weak solution.
\end{corollary}
\begin{proof}
Without loss of generality, we will set $\tau=1$.
Assume by contradiction that $\rho^{(1)}$, $\rho^{(2)}$ are two solutions. Fix arbitrary $0\le t'\le t$. Then, by an application of \eqref{eq:Duhamel} to $\Psi_*(\bx,t)$, we have
\begin{align*}
&\big|\rho^{(1)}(\bx,t')- \rho^{(2)}(\bx,t')\big|\le \big\|\rho^{(1)}(\,\cdot\,,0)- \rho^{(2)}(\,\cdot\, ,0)\big\|_{\Lp^\infty(D)}\\
&+\|\nabla\Psi_*\|_{\Lp^\infty(D\times[0, T])}\int_{0}^{t'} \hspace{-.5em}\int_{D} |\nabla_{\by} G^{D}(\bx,\by;t'-s)|\, 
\big|\rho^{(1)}(\by,s)- \rho^{(2)}(\by,s)\big|\, \de\by\de s\\
&\le  \big\|\rho^{(1)}(\,\cdot\,,0)- \rho^{(2)}(\,\cdot\, ,0)\big\|_{\Lp^\infty(D)}\\
&\quad+C(D) \|\nabla\Psi_*\|_{\Lp^\infty(D\times [0, T])}\int_{0}^{t'} \frac{1}{\sqrt{t'-s}} \big\|\rho^{(1)}(\,\cdot\, ,s)- \rho^{(2)}(\, \cdot\, ,s)\big\|_{\Lp^\infty(D)}\de s\\
&\le  \big\|\rho^{(1)}(\,\cdot\,,0)- \rho^{(2)}(\,\cdot\, ,0)\big\|_{\Lp^\infty(D)}\\
&\quad+C(D) \|\nabla\Psi_*\|_{\Lp^\infty(D\times [0, T])} \sqrt{t}\,
\,\sup_{s\le t}\big\|\rho^{(1)}(\,\cdot\, ,s)- \rho^{(2)}(\, \cdot\, ,s)\big\|_{\Lp^\infty(D)}\,, 
\end{align*}
where we used the estimates of Theorem \ref{thm:Kernel}. By taking supremum over $0\le t'\le t$ form both sides, we obtain that for $t<1/(C(D)^2 \|\nabla\Psi_*\|_{\Lp^\infty(D\times[0, T])}^2)$,
\begin{align}
\sup_{s\le t}\big\|\rho^{(1)}(\,\cdot\, ,s)- \rho^{(2)}(\, \cdot\, ,s)\big\|_{\Lp^\infty(D)}\le 
\frac{\big\|\rho^{(1)}(\,\cdot\,,0)- \rho^{(2)}(\,\cdot\, ,0)\big\|_{\Lp^\infty(D)}}{1-C(D)\|\nabla\Psi_*\|_{\Lp^\infty(D\times [0, T])}\sqrt{t}}\, .
\end{align}
Therefore, the two solutions coincide if we fix the initial condition $\rho^{(1)}(\,\cdot\,,0)=\rho^{(2)}(\,\cdot\,,0) = \rho_{\sinit}$.
For larger $t$, the claim follows by iterating the above argument.
\end{proof}

\section{Nonlinear dynamics}
\label{app:nonlinear}

The `nonlinear dynamics' plays an important role in our proof of Theorem \ref{TH:CONVPDE}. 
In this section  we adopt the same general setting as in Appendix \ref{app:genPDE}, remembering that for our application we set $D=\Omega^{\delta}$
and $U,V$ as per Eq.~(\ref{eq:KtoUV}).

Given $\rho:[0,T]\to \cuP_2(D)$,
consider the following stochastic differential equation for a process $(\bX_t)_{t\in [0,T]}$,  with a reflecting boundary condition
(known as `Skorokhod problem')
\begin{align}
\bX_0 & \sim \rho_{\sinit}\, \label{eq:SDE1}\\
\de\bX_t & =-\nabla \Psi(\bX_t,\rho_t) \, \de t+ \sqrt{2\tau}\,\de\bB_t + \de\bPhi_t\, ,   \label{eq:SDE2}
\end{align}
where $(\bB_t)_{t\ge 0}$ is a standard $d$-dimensional Brownian motion and $(\bPhi_t)_{t\ge 0}$ enforces the reflecting boundary
by satisfying the following constraints (recall that $\bn(\bx)$ is the normal to $\partial D$ at $\bx\in\partial D$, directed inside):
\begin{itemize}
\item[$(i)$] $(\bPhi_t)_{t\ge 0}$ is adapted (and hence so is $(\bX_t)_{t\ge 0}$).
\item[$(ii)$] $t\mapsto \bPhi_t$ has (almost surely) bounded variation. Denoting by $\|\bPhi\|_{\sTV}(t)$ the total variation 
of $\bPhi$ on the interval $[0,t]$, we define the measure $\mu_{\Phi}$ on $[0,T]$ by $\mu_{\Phi}([0,t]) = \|\bPhi\|_{\sTV}(t)$.
\item[$(iii)$] $\mu_{\Phi}(\{t:\, \bX_t\in D^\circ\}) =0$, where $D^\circ$ denotes the interior of $D$.
\item[$(iv)$] We have that, for $t\in [0,T]$,
\begin{align}
\bPhi_t = \int_0^t \bN_s \, \mu_{\Phi}(\de s)\, ,
\end{align}
where $\bN_s=\bn(\bX_s)$, for $\mu_{\Phi}$-almost every $s$.
\end{itemize}
Then, $(\bX_t,\bPhi_t)_{t\in [0,T]}$ is said to solve the Skorokhod problem.

\begin{lemma}[Existence, uniqueness and continuity of Skorokhod problem]\label{lemma:UniqueSko}
Fix $\rho_{\sinit}\in \cuP_2(D)$ and let $\rho:[0,T]\to \cuP_2(D)$ with $\rho_0=\rho_{\sinit}$. 
Then, the Skorokhod problem  (\ref{eq:SDE1}),  (\ref{eq:SDE2}) admits a unique solution $(\bX_t)_{t\ge 0}$ with continuous paths.
Define $\cuF(\rho)_t\in \cuP_2(D)$, for $t\in[0,T]$, by letting  $\cuF(\rho)_t=\Law(\bX_t)$. Then, 
$\cuF(\rho)\in \Cont([0,T],\cuP_2(D))$.
\end{lemma}
\begin{proof}
Let  $\bb(\bx,t) \equiv -\nabla \Psi(\bx,\rho_t)$ and notice that, by the smoothness of $U,V$, and compactness of $D$,
this is a Lipschitz continuous function of $\bx$. Hence the problem (\ref{eq:SDE1}),  (\ref{eq:SDE2}) admits a unique solution 
by  \cite[Theorem 4.1]{tanaka1979stochastic}.

We are left with the task of proving that $t \mapsto \cuF(\rho)_t$ is continuous in $W_2$ metric. Notice that
\begin{align}
\bX_t & =\bX_0 +\int_0^t\bb(\bX_s,s)\, \de s+ \sqrt{2\tau}\,\bB_t + \bPhi_t\, .
\end{align}
By \cite[Lemma 2.2]{tanaka1979stochastic}, we have, for any $s\le t$,
\begin{align}
|\bX_t-\bX_s|^2 & \le \left| \int_s^t\bb(\bX_r,r)\, \de r+ \sqrt{2\tau}\,(\bB_t-\bB_s)\right|^2+\nonumber\\
&2 \int_s^t\<  \int_r^t\bb(\bX_u,u)\, \de u+ \sqrt{2\tau}\,(\bB_t-\bB_r) ,\bN_r\>\mu_{\Phi}(\de r)\, .\nonumber
\end{align}
Taking expectation, we get 
\begin{equation}\label{eq:continuityF}
\begin{split}
\E\{|\bX_t-\bX_s|^2\} &\le \sup_{\bx,t}|\bb(\bx,t)|^2(t-s)^2 + 2\tau (t-s)\\
&\quad\quad+2 \sup_{\bx,t}|\bb(\bx,t)|
 \int_s^t (t-r)\mu_{\Phi}(\de r)\\
&\le  \sup_{\bx,t}|\bb(\bx,t)|^2(t-s)^2 + 2\tau (t-s)\\
&\quad \quad +2 \sup_{\bx,t}|\bb(\bx,t)|
(t-s)\mu_{\Phi}([0,t])\, ,
\end{split}
\end{equation}
whence the continuity follows.
\end{proof}

\begin{definition}[Solution of nonlinear dynamics]
We say that $\rho\in \Cont([0,T];\cuP_2(D))$ is a solution of the nonlinear dynamics if $\cuF(\rho) =\rho$,
namely
\begin{align}
\rho_t = {\rm Law}(\bX_t)\, \;\;\; \forall t\in[0,T]\, .\label{eq:NonlinearDynFP}
\end{align}
\end{definition}

\begin{lemma}
Assume $\tau>0$. If $\rho:[0,T]\to \cuP_2(D)$ is a weak solution of the PDE \eqref{eq:PDEbis} with initial and boundary
conditions \eqref{eq:BCbis}, then it is a solution of the nonlinear dynamics. 
Vice versa, if $\rho:[0,T]\to \cuP_2(D)$ is a solution of the nonlinear dynamics, then it is a weak solution of
 PDE \eqref{eq:PDEbis} with initial and boundary
conditions \eqref{eq:BCbis}.
\end{lemma}
\begin{proof}
Let $\rho$ be a weak solution of the PDE \eqref{eq:PDEbis}, and assume $\tau>0$. 
 Let $(\bX_t)_{t\ge 0}$ be the unique solution of the
Skorokhod problem (\ref{eq:SDE1}), (\ref{eq:SDE2}), cf. Lemma  \ref{lemma:UniqueSko}.
Let $\trho_t\equiv \Law(\bX_t)$, $t\ge 0$, i.e. $\trho\equiv \cuF(\rho)$. For $g\in \Cont^{2}(D)$,  satisfying 
$\<\bn(\bx),\nabla g(\bx)\>=0$ for all $\bx\in\partial D$, compute
\begin{align}
\int g(\bx) &\, \trho_t(\de\bx) =  \E\{g(\bX_t)\} \\
&\stackrel{(a)}{=} \E\{g(\bX_0)\}+\int_{0}^t \E\left\{-\<\nabla \Psi(\bX_s,\rho_s), \nabla g(\bX_s)\>
+\tau\Delta g(\bX_s)\right\}\de s \nonumber\\
&\phantom{AAAA}+ \E\int _0^t \<\nabla g(\bX_s),\bN_s\>\, \mu_{\Phi}(\de s)\nonumber\\
&\stackrel{(b)}{=}\E\{g(\bX_0)\} +\int_{0}^t \E\left\{-\<\nabla \Psi(\bX_s,\rho_s), \nabla g(\bX_s)\>
+\tau\Delta g(\bX_s)\right\}\de s \nonumber\\
&\stackrel{(c)}{=}\int_{D} g(\bx)\, \rho_{\sinit}(\de\bx)\nonumber\\
&\phantom{AAAA} +\int_{0}^t \int_{D}\left\{-\<\nabla \Psi(\bx,\rho_s), \nabla g(\bx)\>
+\tau\Delta g(\bx)\right\} \, \trho_s(\de\bx)\, \de s\, .\nonumber
\end{align}
Here $(a)$ follows from Ito's formula for continuous semimartingales \cite{rogers1994diffusions}, 
$(b)$ since $\bX_s\in \partial D$ and  $\bN_s=\bn(\bX_s)$ for $\mu_{\Phi}$-almost every $s$, and $(c)$ by the definition of $\trho$. 
We conclude that $\trho$ is a weak solution of the linearized PDE \eqref{eq:LinPDE}, with  $\Psi_*(\bx,t) = \Psi(\bx,\rho_t)$.
Since $\rho$ also solves the same linearized PDE, we conclude by Lemma \ref{lemma:UniqLinear}
that $\trho_t = \rho_t$ for all $t\in [0,T]$,
and therefore $\rho$ is a solution of the nonlinear dynamics.

Next,  assume that $\rho:[0,T]\to \cuP_2(D)$ is a solution of the nonlinear dynamics. Then by the same application of
Ito's formula to the process $\bX_t$, we have
\begin{align}
\int g(\bx) \, \rho_t(\de\bx) &=\E\{g(\bX_0)\} \\
&+\int_{0}^t \E\left\{-\<\nabla \Psi(\bX_s,\rho_s), \nabla g(\bX_s)\>
+\tau\Delta g(\bX_s)\right\}\de s \, ,\nonumber
\end{align}
which coincides with the claim that $\rho$ is a weak solution of the PDE \eqref{eq:PDEbis}.
\end{proof}

\begin{theorem}[Existence and uniqueness of nonlinear dynamics]
For any initial condition $\rho_{\sinit}\in\cuP_2(D)$, and any $T>0$, the nonlinear dynamics (\ref{eq:NonlinearDynFP}) admits
a unique solution $\rho:[0,T]\to \cuP_2(D)$ with $\rho_0=\rho_{\sinit}$. As a consequence, the 
PDE \eqref{eq:PDEbis} with initial and boundary
conditions \eqref{eq:BCbis} has a unique solution.
\end{theorem}
\begin{proof}
Note that it is sufficient to prove the claim for $T\le T_0$, where $T_0>0$ is a small enough constant,
since this implies the claim for arbitrary $T$ by breaking $[0,T]$ into intervals of size smaller than $T_0$.

We claim that $\cuF$ is a contraction on $\Cont([0,T],\cP_2(D))$ endowed with the metric 
$d(\rho,\trho)\equiv \sup_{t\in [0,T]} W_2(\rho,\trho)$. To show that this is the case,
define $\bb(\bx,t) \equiv -\nabla \Psi(\bx,\rho_t)$, $\tbb(\bx,t) \equiv -\nabla \Psi(\bx,\trho_t)$.
By the smoothness of $U$, $V$ and by the compactness of $D$, we have that $\bb$ and $\tbb$ are Lipschitz continuous in $\bx$, with Lipschitz constant $L$
independent of $t,\rho,\trho$. Further,
\begin{equation}\label{eq:lipconb}
\begin{split}
|\bb(\bx,t) -\tbb(\bx,t)| &= \left|\int \nabla U(\bx,\tbx)\, \rho_t(\de\tbx)-\int \nabla U(\bx,\tbx)\, \trho_t(\de\tbx)\right|\\
& \le C\, W_1(\rho_t,\trho_t)\le C\, d(\rho,\trho)\, .
\end{split}
\end{equation}
\comm{Let $(\bX_t,\bPhi_t)$ and $(\tbX,\tbPhi_t)$ are be solution of  the Skorokhod problem~\eqref{eq:SDE2}, with  
drift coefficients $\bb(\bx,t)$, $\tbb(\bx,t)$. We couple the processes $\bX_t$ and $\tbX_t$ by using the same initial condition 
$\bX_0$ and same Brownian motion $\bB_t$:} 
\begin{align}
\bX_t & =\bX_0 +\int_0^t\bb(\bX_s,s)\, \de s+ \sqrt{2\tau}\,\bB_t + \bPhi_t\, ,\\
\tbX_t & =\bX_0 +\int_0^t\tbb(\tbX_s,s)\, \de s+ \sqrt{2\tau}\,\bB_t + \tbPhi_t\, .
\end{align}
Define 
\begin{align}
\bD_t \equiv \int_0^t\bb(\bX_s,s)\, \de s -\int_0^t\tbb(\tbX_s,s)\, \de s\, ,
\end{align}
and notice that, by the above remarks, 
\begin{align}
|\bD_t| \le  L\int_0^t|\bX_s-\tbX_s|\, \de s + C\, t\, d(\rho,\trho)\, .
\end{align}
Further, by  \cite[Remark 2.2]{tanaka1979stochastic}, we have
\begin{align}
|\bX_t-\tbX_t|^2&\le 2\int_0^t \<\bX_s-\tbX_s,\bD_s\>\, \de s\\
& \le 2\int_0^t |\bX_s-\tbX_s| \left(L\left(\int_0^s|\bX_r-\tbX_r|\, \de r\right)+ C\, s\, d(\rho,\trho)\right)\, \de s\nonumber\\
& \le 2L\left(\int_0^t |\bX_s-\tbX_s|  \de s\right)^2 +2C\, t\, d(\rho,\trho) \int_0^t |\bX_s-\tbX_s|  \de s\nonumber\\
& \le 2 L t \int_0^t |\bX_s-\tbX_s|^2  \de s +2C\, t^{3/2}\, d(\rho,\trho)\left(\int_0^t |\bX_s-\tbX_s|^2  \de s\right)^{1/2}\, \hspace{-1em}.\nonumber
\end{align}
\comm{Define $\Delta(t) \equiv \E\{|\bX_t-\tbX_t |^2\}$ and $\oDelta(t) \equiv \sup_{s\le t}\Delta(s)$. By taking the 
expectation of the last inequality and using Jensen's inequality, we get}
\begin{equation}
\Delta(t) \le 2  L t \int_0^t \Delta(s) \de s +2C\, t^{3/2}\, d(\rho,\trho)\left(\int_0^t \Delta(s)  \de s\right)^{1/2}\, ,
\end{equation}
which immediately implies
\begin{equation}
\oDelta(t)  \le 2 L t^2\, \oDelta(t) + 2C\, t^2\,  d(\rho,\trho)\, \oDelta(t)^{1/2}\, .
\end{equation}
Hence, for $T_0 < (2L)^{-1/2}$, 
\begin{align}
\oDelta(t) &\le\left(\frac{2CT_0^2}{1-2LT_0^2}\right)^{2} d(\rho,\trho)^{2}\, .
\end{align}
Selecting $T_0$ small enough, so that $(2CT_0^2)/(1-2LT_0^2)\le 1/2$, we obtain
\begin{align}
d(\cuF(\rho),\cuF(\trho)) \le \sqrt{\oDelta(T_0)} &\le \frac{1}{2}\, d(\rho,\trho)\, .
\end{align}
This proves that $\cuF$ is a contraction as claimed. By Lemma \ref{lemma:UniqueSko}, $\cuF$ maps $\Cont([0,T],\cP_2(D))$ into itself. Furthermore, $\Cont([0,T],\cP_2(D))$ is complete with respect to the metric $d$. As a result, there exists a unique fixed point.
\end{proof}

We conclude this section by stating a result about the discretization of the nonlinear dynamics.
Fix a solution $(\rho_t)_{t\ge 0}$ of the PDE \eqref{eq:PDEbis} with initial condition $\rho_0=\rho_{\sinit}$, a step size $\eps>0$ and define recursively 
the random variables $(\bX^{\eps})_{k\in \naturals}$ by
\begin{align}
\bX^{\eps}_0 & \sim \rho_{\sinit}\, \label{eq:Euler1}\\
\bX^{\eps}_{k+1} & =\Proj\Big(\bX_k^{\eps}-\eps\nabla \Psi(\bX^\eps_k,\rho_{k\eps}) + \sqrt{2\tau\eps}\,\bg_k \Big)\, .   \label{eq:Euler2}
\end{align}
This can be viewed as an Euler discretization of the  stochastic differential equation \eqref{eq:SDE1}, \eqref{eq:SDE2},
and the next theorem establishes that this is indeed a close
approximation of the original process. It is just an immediate consequence  of a result of Slomi\'nski 
\cite{slominski1994approximation,slominski2001euler}.
\begin{theorem}[Theorem 3.2 in \cite{slominski2001euler}]\label{thm:Slominski}
Consider the nonlinear dynamics defined by Eqs.~\eqref{eq:SDE1}, \eqref{eq:SDE2}. Assume $\Ball(\bfzero;r)\subseteq D$, 
and $\|\nabla V\|_{\Lp^\infty(D)}$, $\|\nabla U\|_{\Lp^\infty(D\times D)}$, $\|\nabla V\|_{{\rm Lip}}$, $\|\nabla U\|_{{\rm Lip}}
\le L$. Also assume that $\supp(\rho_{\sinit})\subseteq \Ball(\bfzero,r)$.
\comm{Construct the Euler scheme  
(\ref{eq:Euler1}), (\ref{eq:Euler2}) on the same probability  space by letting $\bX^{\eps}_0= \bX_0$ and $\bg_k = (\bB((k+1)\eps)-\bB(k\eps))/\sqrt{\eps}$. }
Then, for any $p\in\naturals$, $T\in\reals_{\ge 0}$, 
\begin{align}
\E\Big\{\max_{k\in [0,T/\eps]\cap\naturals} \big|\bX^{\eps}_k-\bX_{k\eps}\big|^{2p}\Big\}^{1/(2p)} \le C_*L r\, e^{C_*pLT} (d^2\eps \log(1/\eps))^{1/4}\, .
\end{align}
\end{theorem}
\begin{proof}
The proof is obtained simply by chasing the constants in the proof of Theorem 3.2 (part (ii)) of \cite{slominski2001euler}, and using the optimal constant 
in the Burkholder-Davis-Gundy inequality (which yields $C(p)\le (C_* p)^{2p}$ in  \cite[Eq.~(2.7)]{slominski2001euler}). 
\end{proof}

\section{Convergence of SGD to the PDE: Proof of Theorem \ref{TH:CONVPDE}}\label{app:conv}

The proof is a `propagation of chaos' argument \cite{sznitman1991topics}. While the basic idea is  similar to the one 
used in \cite{mei2018mean}, implementing it requires different estimates because of the reflecting boundary conditions. 
In particular, we rely on tools developed in the study of discretizations of reflecting stochastic differential equations.

We will prove a more general theorem that implies  Theorem~\ref{TH:CONVPDE} as a special case, and also applies to the setting of
\cite{mei2018mean}. Namely, we consider data $\{\bz_i=(y_i,\bx_i)\}_{i\ge 1}$ i.i.d. with common distribution $\prob$ on $\reals\times\reals^{d_0}$,
and parameters $\bw_i\in D\subseteq \reals^d$. These parameters are initially sampled independently from distribution $\rho_0\in \cuP_2(D)$, and then evolve according to 
\begin{align}
\bw_i^{k+1} &= \Proj\left\{\bw_i^k+\bF_i(\bz_{k+1};\bw^k)\right\}\, ,\label{eq:GeneralUpdate}\\
\bF_i(\bz_{k+1};\bw^k) & =  -\eps\nabla\sigma(\bx_{k+1};\bw_i^k) \Big(y_{k+1}-\frac{1}{N}\sum_{i=1}^N\sigma(\bx_{k+1};\bw_i^k)\Big)+\sqrt{2\tau\eps}\
\bg^{k+1}_i \, .
\end{align}
Here $\Proj$ is the projection on the closed convex domain $D\subseteq\reals^d$ with non-empty interior.
The setting of Theorem \ref{TH:CONVPDE} is recovered by taking $\sigma(\bx;\bw) = K_{\delta}(\bx-\bw)$, $D=\Omega^{\delta}$, $\bx_k\sim\Unif(\Omega)$,
$\E\{y_k|\bx_k\} = f(\bx_k)$. 

We make the following assumptions:
\begin{itemize}
\item[{\sf (G1)}] $\|y\|_{\infty}$, $\|\sigma\|_{\infty} =\ess\sup_{\bw\in D,\bx}|\sigma(\bx;\bw)|\le \sigma_{\infty}$, and $\nabla_{\bw}\sigma(\bx;\bw)$ is $\gamma$-subgaussian.
\item[{\sf (G2)}] Letting $V(\bw) = -\E\{y\sigma(\bx;\bw)\}$, $U(\bw_1,\bw_2) \equiv \E\{\sigma(\bx;\bw_1)\sigma(\bx;\bw_2)\}$, both $V$ and $U$ 
are differentiable with Lipschitz continuous derivative, namely
$\|\nabla V\|_{{\rm Lip}}, \|\nabla U\|_{{\rm Lip}}\le L$. 
Further, we assume $\|\nabla U\|_{\Lp^\infty(D\times D)}<\infty$. 
\end{itemize}

\begin{theorem}\label{thm:convPDE-GN}
Consider the general update (\ref{eq:GeneralUpdate}) with initialization $(\bw_i^0)_{i\le N}\sim_{iid}\rho_0=\rho_{\sinit}$, under the conditions
${\sf (G1)}$, ${\sf (G2)}$ above.
For $t\ge 0$, let $\rho_t$ be the unique solution of the PDE \eqref{eq:PDEbis} with initial and boundary conditions \eqref{eq:BCbis}. 
Assume $\supp(\rho_{\sinit})\subseteq \Ball(\bfzero,r)$.

Then, for $T\ge 0$ $TL\ge 1$, any $g:\mathbb R^d\to \mathbb R$ with $\|g\|_{\rm Lip}\le 1$ and for $\varepsilon \le 1$, $p\in\naturals$,
the following holds with probability at least $1-z^{-2p}$:
\begin{equation}\label{eq:ErrorBoundGeneral}
\begin{split}
&\sup_{k\in [0, T/\varepsilon]\cap \mathbb N} \bigg| \sum_{i=1}^N g(\bw_i^k)-\int g(\bw)\rho_{k\varepsilon}(\de \bw)\bigg|\le z\, \err(N,d,\eps) \,\,e^{C_*pL\, T},\\
&\sup_{k\in [0, T/\varepsilon]\cap \mathbb N} | R_N(\bw^k)-R^\delta(\rho_{k\epsilon})|\le z\, \err(N,d,\eps) \,\,\,e^{C_*pL\, T} ,
\end{split}
\end{equation}
where
\begin{equation}\label{eq:errNGeneral}
\err(N,d,\eps) =\sqrt{\frac{d}{N}}\vee
\big(\sigma_{\infty}\gamma \sqrt{d\eps}\big)\vee \bigg(Lr(d^2\eps\log(1/\eps))^{1/4}\bigg)\, .
\end{equation}
\end{theorem}
%
Theorem \ref{TH:CONVPDE} follows as a special case of Theorem~\ref{thm:convPDE-GN} by considering $\sigma(\bx;\bw) = K_{\delta}(\bx-\bw)$ and letting $\sigma_\infty \le C_*\delta^{-d}$, $\gamma = C_*\delta^{-d-1}$ and $L = C_*\delta^{-2d-1}$.
\begin{proof}
Let $\cF_k$ denote the sigma algebra generated by $(\bz_j)_{j\le k}$ and denote the empirical distribution of $(\bw_i^k)_{i\le N}$ by $\rho^{(N)}_k \equiv\sum_{i=1}^n\delta_{\bw_i^k}$. 
Note that
\begin{align}
\E\big\{ \bF_i(\bz_{k+1};\bw^k) \big|\cF_k\big\} & = \eps \bG(\bw^k_i;\rho^{(N)}_k)\, ,\nonumber\\
\bG(\bw;\rho)& \equiv -\nabla\Psi(\bw;\rho) = -\nabla V(\bw) - \int \nabla U(\bw,\bw') \, \rho(\de\bw')\, .
\end{align}
We introduce two auxiliary processes  $(\obw_i^k)_{i\le N}$  $(\hbw_i^k)_{i\le N}$, with initial conditions $\obw_i^{0} =\hbw_i^0=\bw_i^0$, as follows:
\begin{itemize}
\item The trajectories $(\hbw_i^k)_{k\ge 0}$ are i.i.d. copies of the nonlinear dynamics introduced in Appendix \ref{app:nonlinear}, sampled at times $t=k\eps$. Namely,
for any $k\in \reals$
\begin{align}
\hbw_i^k = \bw_i^0 + \int_0^{k\eps} \bG(\bw_i^{s/\eps};\rho_s)\, \de s +\sqrt{2\tau}\, \bB_i(k\eps) + \bPhi_i(k\eps)\, .
\end{align}
In particular, for any $k$, $(\hbw_i^k)_{i\le N}\sim_{iid} \rho_{k\eps}$. 
\item The trajectories $(\obw_i^{k})_{k\ge 0}$ are obtained by the Euler discretization of the non-linear dynamics:
\begin{align}
\obw_i^{k+1} & = \Proj\Big(\obw_i^k+\eps \bG(\obw^k_i;\rho_{k\eps})+\sqrt{2\tau\eps}\bg_i^{k+1}\Big)\, .
\end{align}
As above, $(\rho_s)_{s\ge 0}$ is the solution of the PDE \eqref{eq:PDEbis}. Note that, again, the $(\obw_i^k)_{i\le N}$ are i.i.d. although their distribution does not coincide with 
$\rho_{k\eps}$.  
\end{itemize} 
We construct these three processes on the same space by letting $\bB_i((k+1)\eps) = \bB_i(k\eps)+\sqrt{\eps}\bg_i^{k+1}$,
and define the distances (for $q\ge 1$)
\begin{align}
\cD_q(k) & \equiv \E\Big\{\max_{j\le k}\big|\bw_i^j-\obw_i^j\big|^q\Big\}^{1/q} = \E\left\{\frac{1}{N}\sum_{i=1}^N \max_{j\le k}\big|\bw_i^j-\obw_i^j\big|^q\right\}^{1/q}\label{Dq}\\
& \ge \E\left\{\max_{j\le k}\frac{1}{N}\sum_{i=1}^N \big|\bw_i^j-\obw_i^j\big|^q\right\}^{1/q} \, ,\nonumber\\
\hcD_q(k) & \equiv \E\Big\{\max_{j\le k}\big|\hbw_i^j-\obw_i^j\big|^q\Big\}^{1/q}\, .\label{hDq}
\end{align}
Theorem \ref{thm:Slominski} yields, for $p\in\naturals$,
\begin{align}
\hcD_{2p}(k) \le  C_*Lr\, e^{C_*pL(k\eps)}(d^2\eps \log(1/\eps))^{1/4}\,. \label{eq:hcd}
\end{align}

Note that $\bw_i^k$, $\obw_i^k$ take the form 
\begin{align}
\bw_i^k & = \bw_i^0 +\bM_i^k+\bV_i^k+\bphi_i^k\, ,\\
\obw_i^k& = \bw_i^0+\obM_i^k+\obV_i^k+\obphi_i^k\, , 
\end{align}
where $\bM_i^k, \obM_i^k$ are  martingales with respect to the filtration $\cF_k$: $\E\{\bM_i^k|\cF_{k-1}\} = \bM_i^{k-1}$, $\E\{\obM_i^k|\cF_{k-1}\} = \obM_i^{k-1}$,
and  $\bV_i^k$, $\obV_i^k$ are $\cF_{k-1}$-measurable. Explicitly
\begin{align}
\bM_i^k & = \sum_{\ell=0}^{k-1} \big(\bF_i(\bz_{\ell+1};\bw^{\ell}) - \E\{\bF_i(\bz_{\ell+1};\bw^{\ell}) |\cF_{\ell}\}\big)\, ,\\
\obM_i^k & = \sum_{\ell=0}^{k-1} \sqrt{2\tau\eps}\, \bg_i^{\ell+1}\, ,\\
\bV_i^k & = \sum_{\ell=0}^{k-1} \E\{\bF_i(\bz_{\ell+1};\bw^{\ell}) |\cF_{\ell}\} =  \sum_{\ell=0}^{k-1} \eps \, \bG(\bw^{\ell}_i;\rho^{(N)}_\ell)\, ,\\
\obV_i^k & =  \sum_{\ell=0}^{k-1} \eps \, \bG(\obw^{\ell}_i;\rho_{\ell\eps})\, .
\end{align}
Finally, $\bphi^k_i$, $\obphi^k_i$ are corrections to satisfy the constraint $\bw_i^k,\obw_i^k\in D$. Indeed the above can
be viewed as Skorokhod problems with unknowns $(\bw_i,\bphi_i)$ and  $(\obw_i,\obphi_i)$.  

Using \cite[Theorem 1]{slominski1994approximation} (where we can set $C_p =(C_*p)^{2p}$ which is the tight constant in the Burkholder-Davis-Gundy inequality),
we get 
\begin{align}
\cD_{2p}(k) \le C_* p \left\{\E\big([\bM_i-\obM_i]_k^p\big)^{1/(2p)} + \E\big(|\bV_i-\obV_i|_k^{2p}\big)^{1/(2p)}\right\} \, ,\label{eq:BasicBoundSlominski}
\end{align}
where $[\bM]_k$ denotes the quadratic variation of the martingale $\bM$, and $|\bV|_k$ is the total variation of the process $\bV$. We then have
\begin{align}
A_{1,p}(k) &\equiv \E\big([\bM_i-\obM_i]_k^p\big)^{1/(2p)} 
 = \E\left\{\left[\sum_{\ell=0}^{k-1} |\bZ_i^\ell|^2\right]^p\right\}^{1/(2p)}\nonumber\\
\bZ_i^{\ell} & = \eps \nabla\sigma(\bx_{\ell+1};\bw_i^k) \Big(y_{\ell+1}-\frac{1}{N}\sum_{i=1}^N\sigma(\bx_{\ell+1};\bw_i^\ell)\Big) +
\eps\bG(\bw_i^{\ell};\rho_\ell^{(N)})\, ,
\end{align}
Note that under the stated assumption the martingale increments $\bZ^{\ell}_i$ are sub-Gaussian with variance proxy upper bounded by 
$v^2 = C_*\eps^2\sigma_{\infty}^2\gamma^2$.
Therefore, by using the moment generating function of $\chi^2_d$ distribution, we have
\begin{align}
\E \exp\bigg\{\frac{\alpha^2}{2dv^2} |\bZ^{\ell}_i|^2 \bigg\} \le {\Big(1-\frac{\alpha^2}{d}\Big)^{-d/2}}\,.
\end{align} 
Hence,
\begin{align}
\E \exp\bigg\{\frac{\alpha^2}{2dv^2} \sum_{\ell=0}^{k-1}|\bZ^{\ell}_i|^2 \bigg\} \le \Big(1-\frac{\alpha^2}{d}\Big)^{-dk/2}\,.
\end{align}
By using the inequality $x^p\le e^x p!$, this implies, for $\alpha\le \sqrt{d/2}$,
\begin{align}
\E \bigg\{\bigg[\frac{\alpha^2}{2dv^2} \sum_{\ell=0}^{k-1}|\bZ^{\ell}_i|^2\bigg]^p \bigg\} \le p! \Big(1-\frac{\alpha^2}{d}\Big)^{-dk/2}\le p^p e^{\alpha^2k}\,.
\end{align}
Equivalently, 
\begin{align}
\E \bigg\{\Big[\sum_{\ell=0}^{k-1}|\bZ^{\ell}_i|^2\Big]^p \bigg\}^{1/(2p)} \le \Big(\frac{\sqrt{2d}v}{\alpha}\Big) \sqrt{p} e^{\alpha^2k/(2p)}\,.
\end{align}
By taking $\alpha = \sqrt{p/k}$ (which is allowed provided $p\le \sqrt{kd/2}$), we obtain that
\begin{align}
A_{1,p}(k) \le 10\sqrt{k d} v\le C_* \sqrt{k d}\,\eps \sigma_\infty \gamma \, . \label{eq:A1}
\end{align}
%
We next consider the total variation of the process $\bV_i$ in Eq.~(\ref{eq:BasicBoundSlominski}). We have
\begin{align*}
\E\big(|\bV_i-\obV_i|_k^{2p}\big)^{1/(2p)} &=\E\left\{ \left(\sum_{\ell=0}^{k-1}\eps\Big|\bG(\bw_i^{\ell};\rho_\ell^{(N)}) - \bG(\obw_i^{\ell};\rho_{\ell\eps}) \Big|\right)^{2p}\right\}^{1/(2p)}\\
&\le \E\left\{ \left(\sum_{\ell=0}^{k-1}\eps\Big|\bG(\bw_i^{\ell};\rho_\ell^{(N)}) - \bG(\obw_i^{\ell};\rho_{\ell}^{(N)}) \Big|\right)^{2p}\right\}^{1/(2p)}\\
&\phantom{AA}+
\E\left\{ \left(\sum_{\ell=0}^{k-1}\eps\Big|\bG(\obw_i^{\ell};\rho_\ell^{(N)}) - \bG(\obw_i^{\ell};\rho_{\ell\eps}) \Big|\right)^{2p}\right\}^{1/(2p)}\\
& \equiv A_{2,p}(k)+A_{3,p}(k)\, .
\end{align*}
Using the Lipschitz property of $\nabla V$, $\nabla U$, we get
\begin{align}
A_{2,p}(k) &\le L \eps \E\left\{ \left(\sum_{\ell=0}^{k-1}\big|\bw_i^{\ell}- \obw_i^{\ell} \big|\right)^{2p}\right\}^{1/(2p)}\nonumber\\
& \le L\eps \sum_{\ell=0}^{k-1} \cD_{2p}(\ell)\, .\label{eq:A2}
\end{align}
For the second term, we get, by triangular inequality,
\begin{align}
A_{3,p}(k) \le \eps \sum_{\ell=0}^{k-1} \E\left\{ \Big|\bG(\obw_i^{\ell};\rho_\ell^{(N)}) - \bG(\obw_i^{\ell};\rho_{\ell\eps}) \Big|^{2p}\right\}^{1/(2p)}\, .
\end{align}
We next use the expression $\bG(\bw;\rho)  = \nabla V(\bw) +\int \nabla U(\bw,\bw') \, \rho(\de\bw')$, and the fact that $\hbw_i^{\ell}\sim \rho_{\ell\eps}$, to get 
\begin{align*}
A_{3,p}(k) &\le \eps \sum_{\ell=0}^{k-1} \E\left\{ \left|\frac{1}{N}\sum_{j=1}^N\left[\nabla U(\obw_i^{\ell},\bw_j^{\ell})
-\E_{\hbw_j^{\ell}} \nabla U(\obw_i^{\ell},\hbw_j^{\ell})\right] \right|^{2p}\right\}^{1/(2p)}\\
& \le \eps \sum_{\ell=0}^{k-1} \E\left\{ \left|\frac{1}{N}\sum_{j=1}^N\left[\nabla U(\obw_i^{\ell},\bw_j^{\ell})
-\nabla U(\obw_i^{\ell},\hbw_j^{\ell})\right]
  \right|^{2p}\right\}^{1/(2p)}\\
&\phantom{AA}+\eps \sum_{\ell=0}^{k-1} \E\left\{ \left|\frac{1}{N}\sum_{j=1}^N\left[\nabla U(\obw_i^{\ell},\hbw_j^{\ell})
-\E_{\hbw_j^{\ell}}\nabla U(\obw_i^{\ell},\hbw_j^{\ell})\right]
  \right|^{2p}\right\}^{1/(2p)}\\
&\equiv A^{(1)}_{3,p}(k)+A^{(2)}_{3,p}(k)\, .
\end{align*}
Using once more the Lipschitz property of $\nabla U$, and the symmetry of the distributions of $(\bw^{\ell})_{i\le N}$, $(\hbw^{\ell})_{i\le N}$
under permutations, we obtain
\begin{align}
A_{3,p}^{(1)}(k) &\le  L\eps \sum_{\ell=0}^{k-1} \E\left\{ \left|\bw_j^{\ell}-\hbw_j^{\ell}\right|^{2p}\right\}^{1/(2p)}\\
&\le L\eps \sum_{\ell=0}^{k-1} \cD_{2p}(\ell) + L\eps \sum_{\ell=0}^{k-1} \hcD_{2p}(\ell)  \, .\label{eq:A3_1}
\end{align}
Finally, $|\nabla U(\obw_i^{\ell},\hbw_j^{\ell})|\le L$ and therefore the vector $$\bW = \frac{1}{N}\sum_{j=1}^N\left[\nabla U(\obw_i^{\ell},\hbw_j^{\ell})
-\E_{\hbw_j^{\ell}}\nabla U(\obw_i^{\ell},\hbw_j^{\ell})\right]$$ is sub-Gaussian, with variance proxy upper bounded by $v^2=L^2/N$. 
This implies that $\E\{|\bW|^{2p}\}^{1/(2p)}\le C_*\sqrt{dp}\, v$, and therefore 
\begin{align}
A^{(2)}_{3,p}(k) \le C_*(k\eps) L\sqrt{\frac{dp}{N}} \,.\label{eq:A3_2}
\end{align}
Substituting \eqref{eq:A1}, \eqref{eq:A2}, \eqref{eq:A3_1}, \eqref{eq:A3_2} in Eq.~\eqref{eq:BasicBoundSlominski}, we obtain
\begin{align}
\cD_{2p}(k) &\le C_*Lp\eps \sum_{\ell=0}^{k-1} \cD_{2p}(\ell) +C_*Lp(k\eps) \hcD_{2p}(k) \nonumber\\
&+ C_*p\sqrt{kd}\, \eps \sigma_{\infty}\gamma +C_*Lp(k\eps)\sqrt{\frac{dp}{N}} \, .
\end{align}
Using Eq.~(\ref{eq:hcd}) and  Gronwall inequality, along with the fact that $k\eps\le T$, this yields 
\begin{align*}
\cD_{2p}(T/\eps)& \le C_* \, e^{C_*pLT}\left[\sqrt{\frac{d}{N}}\vee
\big(\sigma_{\infty}\gamma \sqrt{d\eps}\Big)\vee \Big(Lr(d^2\eps\log(1/\eps))^{1/4}\Big)\right]\,.
\end{align*}
By using Eq.~(\ref{eq:hcd}) again, we get
\begin{align}
&\cD_{2p}(T/\eps) + \hcD_{2p}(T/\eps) \\
&\quad\le  C_* \, e^{C_*pLT}\left[\sqrt{\frac{d}{N}}\vee
\big(\sigma_{\infty}\gamma \sqrt{d\eps}\Big)\vee \Big(Lr(d^2\eps\log(1/\eps))^{1/4}\Big)\right] \nonumber\\
&\quad\equiv e^{C_*pLT} \err(N,d,\eps)\, .\nonumber
\end{align}
By Markov inequality along with the Jensen inequality applied to the convex function $x^{2p}$, we have
\begin{align*}
\prob\left\{\frac{1}{N}\sum_{i=1}^N|\bw_i^k-\hbw_i^k|\ge  \Delta \right\}& \le 
\frac{1}{\Delta^{2p}}\E\left\{\left[\frac{1}{N}\sum_{i=1}^N|\bw_i^k-\hbw_i^k|\right]^{2p} \right\} \\
&\le \frac{1}{\Delta^{2p}}\E\left\{\frac{1}{N}\sum_{i=1}^N|\bw_i^k-\hbw_i^k|^{2p} \right\} \\
&\le \frac{1}{\Delta^{2p}}\left[\cD_{2p}(T/\eps)+ \hcD_{2p}(T/\eps)\right]^{2p} \\
&\le  \frac{1}{\Delta^{2p}} e^{2C_*p^2LT} \err(N,d,\eps)^{2p}\,,
\end{align*}
where in the third step we used~\eqref{Dq} and~\eqref{hDq}.
Set $\Delta = z\, e^{C_*pLT}\err(N,d,\eps)$. Thus, we obtain 
\begin{align}
\frac{1}{N}\sum_{i=1}^N|\bw_i^k-\hbw_i^k|\le z\, e^{C_*pLt}\err(N,d,\eps)\, ,\label{eq:KeyCouplingBound}
\end{align}
with probability at least $1-z^{-2p}$. 

The bounds in Eq.~\eqref{eq:ErrorBoundGeneral} follow straightforwardly from Eq.~\eqref{eq:KeyCouplingBound} as in the
proofs of Lemma 3.3 and 3.4 in the supplementary material of \cite{mei2018mean}.
\end{proof}

\section{Regularity of the solutions of the PDE (\ref{EQ:PDEDELTA}) ($\delta>0$)}\label{app:reg}

In this section we prove some standard regularity properties of the solutions of the  PDE (\ref{EQ:PDEDELTA}), for $\delta>0$, 
and indeed for the more general  PDE (\ref{eq:PDEbis}). 
First of all, we show that the weak solution of the PDE \eqref{eq:PDEbis} is in fact strong, i.e., $\rho\in \Cont^{2,1}(\Omega^{\delta},[0,T])$ and the equation \eqref{eq:PDEbis} 
holds pointwise.  We will then prove upper bounds on $\nabla K^{\delta}\ast\rho$ and $\nabla U^{\delta}\ast\rho$ that are uniform in $\delta$. These will be crucial in order to
take the $\delta\to 0$ limit in the next section.

We start by proving a  bound on the $\Lp^\infty$ norm of $\rho$.
 In the proofs of the two lemmas that follow, we assume without loss of generality that $\tau=1$. 
\begin{lemma}[Bound on $\Lp^\infty$ norm]\label{lemma:boundlinf}
Let $\rho_t$ be a weak solution of the PDE \eqref{eq:PDEbis} with initial and boundary
conditions \eqref{eq:BCbis}.
Recall that  $\rho_t$ has a density with respect to Lebesgue measure, denoted  by $\rho(\,\cdot\,,t)$.
Then, there exists a constant $C(\Omega)$ such that, by letting $L= (\|\nabla V\|_{\Lp^\infty(\Omega)} \vee \|\nabla U\|_{\Lp^\infty(\Omega\times \Omega)})$, we have
%
\begin{align}
\|\rho(\,\cdot\,,t)\|_{\Lp^\infty(\Omega)}\le \|\rho_{\sinit}\|_{\Lp^\infty(\Omega)}\, e^{C(\Omega) L^2\, t}\, .
\end{align}
\end{lemma}
\begin{proof}
Any solution the PDE \eqref{eq:PDEbis}  satisfies Eq.~\eqref{eq:Duhamel}. Given a measurable (Borel)  function 
$\rho\in m\cB(\Omega \times [0,T])$, denote by $\cuD(\rho)\in m\cB(\Omega \times [0,T])$ the function given by the right-hand side of \eqref{eq:PDEbis}.
Let $C(\Omega)$ be the constant in the statement of Theorem \ref{thm:Kernel} (part 3) and let $C_{U,V}\equiv C(\Omega) (\|\nabla V\|_{\Lp^\infty(\Omega)}+\|\nabla U\|_{\Lp^\infty(\Omega\times \Omega)})$. We then have
\begin{align}
\|\cuD(\rho)(\,\cdot,t)&\|_{\Lp^\infty(\Omega)}\le \|\rho_{\sinit}\|_{\Lp^\infty(\Omega)} + \big(\|\nabla V\|_{\Lp^\infty(\Omega)}+\|\nabla U\|_{\Lp^\infty(\Omega\times \Omega)}\big)\nonumber\\
&\hspace{4em}\int_{0}^t \sup_{\bx\in\Omega}\big\|\nabla G^\Omega(\bx,\, \cdot\, ;t-s)\big\|_{\Lp^1(\Omega)}
\|\rho(\,\cdot\,,s)\|_{\Lp^\infty(\Omega)}\, \de s \nonumber\\
&\hspace{3em} \le \|\rho_{\sinit}\|_{\Lp^\infty(\Omega)} + C_{U,V} \int_{0}^t(t-s)^{-1/2}  \|\rho(\,\cdot\,,s)\|_{\Lp^\infty(\Omega)}\, \de s. 
\end{align}
Hence
\begin{align}
\|\cuD(\rho)\|_{\Lp^\infty(\Omega\times [0,T])}\le \|\rho_{\sinit}\|_{\Lp^\infty(\Omega)}  + C_{U,V} \sqrt{T}\,   \|\rho\|_{\Lp^\infty(\Omega\times [0,T])}\, .  \label{eq:Linfty}
\end{align}
Proceeding analogously for two different densities $\rho, \trho$, we get
\begin{align}
\|\cuD(\rho)-\cuD(\trho)\|_{\Lp^\infty(\Omega\times [0,T])}\le  C_{U,V} \sqrt{T}\,   \|\rho-\trho\|_{\Lp^\infty(\Omega\times [0,T])}\, . 
\end{align}
Hence $\cuD$ maps $\Lp^\infty(\Omega\times [0,T])$ into itself, and is a contraction for $C_{U,V} \sqrt{T}<1$. Therefore, it must have a unique 
fixed point in $\Lp^{\infty}$ that coincides with the unique solution of PDE \eqref{eq:PDEbis}. Let $T_0= 1/(4C_{UV}^2)$. Then for that fixed point $\rho\in \Lp^\infty(\Omega\times [0,T])$ we have from Eq.~(\ref{eq:Linfty})
\begin{align}
\|\rho\|_{\Lp^\infty(\Omega\times [0,T_0])}\le 2\, \|\rho_{\sinit}\|_{\Lp^\infty(\Omega)}  \, .
\end{align}
The desired claim follow by iterating this inequality $\lceil t/T_0\rceil$ times.
\end{proof}

\begin{lemma}[Strong solutions of PDE]
Let $\rho_t$ be a weak solution of the PDE \eqref{eq:PDEbis} with initial and boundary
conditions \eqref{eq:BCbis},
and recall that, for any $t\le T<\infty$,  this has a density $\rho(\,\cdot\,,t)$, with $\rho\in \Lp^{\infty}(\Omega\times [0,T])$. Fix $q\in\mathbb N$. If $\rho_{\sinit}\in \Cont^q(\Omega)$, then $\rho\in \Cont^{q,1}(\Omega,[0,T])$.
\end{lemma}
\begin{proof}
We prove the claim for $q=2$. For larger values of $q$, the proof is similar and it only requires to iterate the argument. 

The proof uses the same bootstrap technique of  \cite{mei2018mean}[Supplementary material, Lemma 6.7]. The only difference is that the Duhamel 
formula of Eq.~(\ref{eq:Duhamel}) involves the Neumann heat kernel in $\Omega$ instead of the heat kernel in $\reals^d$.

Let $S = \Omega\times [0,T]$ and, for $u:\Omega\times[0,T]\to \reals$. For $r\in \naturals$, $\balpha = (\alpha_1,\dots,\alpha_d)\in\naturals^d$,
let $D^r_tD^{\balpha}_{\bx}u$ be the generalized derivative of $u$, and define the parabolic seminorm
\begin{align}
\<\<u\>\>_{\Lp^p(S)}^{(j)} \equiv \sum_{|\balpha|+2r = j} \|D^r_tD^{\balpha}_{\bx}u\|_{\Lp^p(S)}.
\end{align}
The proof  of \cite{mei2018mean}[Supplementary material, Lemma 6.7] uses the following inequality from \cite{ladyzhenskaia1988linear}[Chapter IV, Section 3, Eq. (3.1)]
\begin{align}
&\<\<\<G *_2 u\>\>^{(2m+2)}_{\Lp^p(S)}\le C\, \<\<u\>\>^{(2m)}_{\Lp^p(S)}\, ,\label{eq:Lady}\\
&G *_2 u (\bx,t) \equiv \int_{\reals^d}\int_0^t G(\bx,\by,\;t-s) \, u(\by,s)\, \de\by\,\de s\, .
\end{align}
Furthermore, \eqref{eq:HeatDecomposition} of Theorem \ref{thm:Kernel} yields
\begin{align}
\<\<G^{\Omega}*_2 u\>\>^{(2m+2)}_{\Lp^p(S)} \le \<\<G *_2 u\>\>^{(2m+2)}_{\Lp^p(S)}+ \<\<G_R^{\Omega}*_2 u\>\>^{(2m+2)}_{\Lp^p(S)} \, .
\end{align}
Since $G_R^{\Omega}\in \Cont^{\infty}(\Omega\times\Omega\times[0,T])$, we have that
$$\<\<G_R^{\Omega}*_2 u\>\>^{(2m+2)}_{\Lp^p(S)} \le C\|u\|_{\Lp^p(S)},$$
which immediately implies that
\begin{align}
\<\<G^{\Omega}*_2 u\>\>^{(2m+2)}_{\Lp^p(S)} \le C\<\<u\>\>^{(2m)}_{\Lp^p(S)}+ C\| u\|_{\Lp^p(S)} \, .\label{eq:SmoothingModified}
\end{align}
The  proof  of \cite{mei2018mean}[Supplementary material, Lemma 6.7]  can be repeated verbatimly with \eqref{eq:Lady} replaced by \eqref{eq:SmoothingModified}.
\end{proof}

As a consequence of the last lemma, the PDE \eqref{eq:PDEbis} admits unique strong solutions 
$\rho\in \Cont^{2,1}(\Omega,[0,T])$ with initial condition $\rho_{\sinit}$ and Neumann boundary condition. We will use $\rho(t)$ as shortcut for  $\rho(\,\cdot\,,t)$. The rest of this appendix is devoted to prove further regularity results for $\rho(t)$, which will be crucial in the proofs provided in Appendix \ref{app:convglob}.
To emphasize the dependence of $\rho$ on $\delta$, we will denote this solution by $\rho^{\delta}$.

In what follows, we will set  the initial condition $\rho^\delta(0)\equiv \rho_{\sinit}^\delta$ at $\delta>0$ to be defined  via
 $\rho_{\sinit}^\delta(\bw)=\lambda_{\delta}^{-d}\rho_{\sinit}(\bw/\lambda_{\delta})$, with $\lambda_{\delta}$ given by Eq.~(\ref{eq:lambda}) 

It is useful to recall the definition of free energy, which is given by
\begin{align}
F^\delta(\rho) &= \frac{1}{2}\, R^\delta(\rho) -\tau\, S(\rho)\nonumber\\
& = \frac{\nu_0}{2}\|f-K^\delta\cv \rho\|_{\Lp^2(\Omega)}^2+\tau\int \rho(\bx) \log\rho(\bx)\, \de\bx\, .
\end{align}

The following lemma provides an expression for the derivative of the free energy with respect to time. Such an expression immediately yields an upper bound on the $\Lp^2(\Omega)$ norm of $K^\delta *\rho^\delta(t)$ which is independent of $\delta$.  
\begin{lemma}\label{lemma:FreeEnergyDecrease}
Let $\rho^\delta\in \Cont^{2,1}(\Omega^\delta,[0,T])$ be the solution of the PDE \eqref{eq:PDEbis} with initial and boundary
conditions \eqref{eq:BCbis}. Then,
\begin{align}
\frac{\de\phantom{t}}{\de t}F^\delta(\rho^\delta(t)) = -\int \big|\nabla\big( \Psi(\bx;\rho^\delta(t))+\tau\log\rho^\delta(\bx,t)\big)\big|^2\, \rho^\delta(\bx,t)\, \de\bx\, .
\end{align}
\end{lemma}
\begin{proof}
By definition
\begin{align*}
F^\delta(\rho^\delta(t)) = \frac{\nu_0}{2}&\|f\|_{\Lp^2(\Omega)}^2 + \int V(\bw) \rho^\delta(\bw,t)\de\bw \nonumber\\
& + \frac{1}{2} \int U(\bw_1-\bw_2) \rho^\delta(\bw_1,t) \de \bw_1  \rho^\delta(\bw_2,t) \de \bw_2 \\
&+ \tau \int \rho^\delta(\bw,t) \log \rho^\delta(\bw,t) \de \bw\,.
\end{align*}
By differentiating $F^\delta(\rho^\delta(t))$ along the solution of \eqref{eq:PDEbis}, we obtain
\begin{align}
\frac{\de}{\de t} F^\delta(\rho^\delta(t)) &= \int V(\bw) \partial_t \rho^\delta(\bw,t) \de \bw\nonumber\\
&\phantom{AAAA}+\int U(\bw_1-\bw_2) \rho^\delta(\bw_1,t) \partial_t\rho^\delta(\bw_2,t)\de\bw_1\de\bw_2 \nonumber\\
&\phantom{AAAA}+ \tau\int (1+\log \rho^\delta(\bw,t)) \partial_t \rho^\delta(\bw,t)\de\bw\nonumber\\
& = \int (\Psi(\bw,\rho^\delta(t)) + \tau \log \rho^\delta(\bw,t) +\tau) \partial_t\rho^\delta(\bw,t) \de\bw\nonumber\\
& = \int (\Psi(\bw,\rho^\delta(t)) + \tau \log \rho^\delta(\bw,t) +\tau)\nonumber\\
& \phantom{AAAA} \nabla\cdot \Big(\rho^\delta(\bw,t) \nabla \left(\Psi(\bw,\rho^\delta(t)) + \tau\log \rho^\delta(\bw,t)\right)\Big)\de\bw\nonumber\\
&= - \int \<\nabla\left(\Psi(\bw,\rho^\delta(t) + \tau \log \rho^\delta(\bw,t) \right), \nonumber\\
&\phantom{AAAA} \nabla (\Psi(\bw,\rho^\delta(t)) +\tau \log \rho^\delta(\bw,t))\> \rho^\delta(\bw,t) \de\bw\nonumber\\
& = - \int \big|\nabla\big( \Psi(\bw;\rho^\delta(t))+\tau\log\rho^\delta(\bw,t)\big)\big|^2\, \rho^\delta(\bw,t)\, \de\bw
\end{align}
\end{proof}

\begin{corollary}\label{coro:ell2Bound}
Let $\rho^\delta\in \Cont^{2,1}(\Omega^\delta,[0,T])$ be the solution of the PDE \eqref{eq:PDEbis} with initial and boundary
conditions \eqref{eq:BCbis}. Then,
\begin{align}\label{eq:ell2bound}
\nu_0\|K^{\delta} \cv \rho^\delta(t)-f\|^2_{\Lp^2(\Omega)}\le 2F^\delta(\rho^\delta(0)) +2\tau\log|\Omega^\delta|\, ,
\end{align}
where $|\Omega^\delta|$ denotes the volume of the set $\Omega^\delta$.
\end{corollary}
\begin{proof}
By Lemma \ref{lemma:FreeEnergyDecrease} we have $F^\delta(\rho^\delta(t)) \le F^\delta(\rho^\delta(0))$. The claim follows by substituting the definition of $F^\delta(\rho^\delta)$
and using $S(\rho^\delta) \le \log|\Omega^\delta|$.
\end{proof}
\begin{remark}\label{rem:delta-ind1}
By Corollary~\ref{coro:ell2Bound}, we are able to provide a $\delta$-free upper bound on $\nu_0\|K^{\delta} \cv \rho^\delta(t)-f\|^2_{\Lp^2(\Omega)}$. Specifically, $\Omega^\delta\subseteq \Omega$ and hence $|\Omega^\delta| \le |\Omega|$.  We also have
\[
F^\delta (\rho^\delta(0)) \le \nu_0 \|f\|^2_{\Lp^2(\Omega)} + \nu_0 \|K^\delta \cv \rho^\delta(0)\|_{\Lp^2(\Omega)} -\tau S(\rho^\delta(0))\,.
\]

Note that
\begin{equation*}
S(\rho^\delta(0)) = S(\rho_{\sinit}) + d\log \lambda_\delta.
\end{equation*}
Since $\lambda_\delta \to 1$ as $\delta\to 0$, there exists a $C_* >0$ such that for $\delta<C_*$, $\lambda_\delta \ge 1/2$. Thus, the term $S(\rho^\delta(0))$ has a $\delta$-free upper bound. 

By Young's inequality it only remains to give a $\delta$-free upper bound on the quantity $\|\rho^\delta(0)\|_{\Lp^2(\Omega)}$. Let us write
\begin{align*}
\|\rho^\delta(0)\|^2_{\Lp^2(\Omega)} &\le \int_{\Omega^\delta} \lambda_\delta^{-2d} \rho^2_{\sinit} (\bw/\lambda_\delta) \de\bw \\
& =  \int_{\Omega} \lambda_\delta^{-2d} \rho^2_{\sinit} (\bx) \lambda_\delta^d \de\bx = \lambda_\delta^{-d} \|\rho^2_{\sinit}\| ^2_{\Lp^2(\Omega)}. 
\end{align*}
Again, for $\delta<C_*$, $\lambda_\delta \ge 1/2$. Also, by Assumption {\sf (A5)} and the fact that $\Omega$ is compact, we have $\|\rho^2_{\sinit}\| ^2_{\Lp^2(\Omega)}<\infty$, which concludes the claim. 
\end{remark}
%

We next prove $\delta$-free upper bound on the gradient of $\nabla K^\delta\cv \rho^\delta$.

\begin{lemma}\label{lemma:SpeedBound}
Let $\rho^\delta\in \Cont^{2,1}(\Omega^\delta,[0,T])$ be the solution of the PDE \eqref{eq:PDEbis} with initial and boundary
conditions \eqref{eq:BCbis}. Then, the following bound holds:
\begin{equation}\label{eq:intfin}
\begin{split}
\int_0^T \int \big|\nabla U\cv&\rho^\delta(\bx,t)|^2 \rho^\delta(\bx,t)\,\de\bx\, \de t+2\tau \int_0^T \int |\nabla(K^{\delta}\cv\rho^\delta)(\bx,t)|^2 \de\bx\, \de t\\
&\le T\|\nabla V\|_{\Lp^\infty(\Omega)}^2+2\nu_0\|f\|^2_{\Lp^2(\Omega)}+4F^\delta(\rho^\delta(0))+
4\tau\log|\Omega^\delta|\, .
\end{split}
\end{equation}
\end{lemma}
\begin{proof}
Denote by $\<f,g\>=\int f(\bx)g(\bx) \de\bx$ the standard scalar product in $\Lp^2$. Then,
\begin{equation}\label{eq:exprlong}
\begin{split}
\frac{1}{2}\frac{\de\phantom{t}}{\de t}&\<U\cv\rho^\delta(t),\rho^\delta(t)\>  = \<U\cv \rho^\delta(t),\partial_t \rho^\delta(t)\> \\
&= \<U\cv \rho^\delta(t),\nabla\cdot (\rho^\delta(t) \nabla (V+ U\cv \rho^\delta(t))) + \tau \Delta \rho^\delta(t) \>  \\
& =\< U\cv\rho^\delta(t),\nabla\cdot(\rho^\delta(t)\nabla V)\> +
\< U\cv\rho^\delta(t),\nabla\cdot(\rho^\delta(t)\nabla(U\cv\rho^\delta(t))\> \\
&\phantom{AAAA}+\tau \<U\cv\rho^\delta(t),\Delta\rho^\delta(t)\>\\
& = -\int\<\nabla(U\cv\rho^\delta)(\bx,t),\nabla V(\bx) \> \, \rho^\delta(\bx,t)\,\de\bx\\
&\phantom{AAAA}- \int |\nabla(U\cv\rho^\delta)(\bx,t)|^2 \, \rho^\delta(\bx,t)\,\de\bx-
\tau \int |\nabla(K^{\delta}\cv\rho^\delta)(\bx,t)|^2 \de\bx\\
& \le \frac{1}{2} \int|\nabla V(\bx) |^2 \, \rho^\delta(\bx,t)\,\de\bx- \frac{1}{2}\int |\nabla(U\cv\rho^\delta)(\bx,t)|^2 \, \rho^\delta(\bx,t)\,\de\bx\\
&\phantom{AAA}-
\tau \int |\nabla(K^{\delta}\cv\rho^\delta)(\bx,t)|^2 \de\bx\, .
\end{split}
\end{equation}
By integrating \eqref{eq:exprlong} between $0$ and $T$, we obtain
\begin{equation}
\begin{split}
\int_0^T \int\big|\nabla U\cv\rho^\delta(\bx,t)|^2& \rho^\delta(\bx,t)\,\de\bx\, \de t+2\tau \int_0^T \int |\nabla(K^{\delta}\cv\rho^\delta)(\bx,t)|^2 \de\bx\, \de t\\
&\le T\|\nabla V\|_{\infty}^2-\<\rho^\delta(T),U\cv\rho^\delta(T)\>+\<\rho^\delta(0),U\cv\rho^\delta(0)\>\\
&\le  T\|\nabla V\|_{\Lp^\infty(\Omega)}^2+\nu_0\|K^{\delta}\cv\rho^\delta(0)\|_{\Lp^2(\Omega)}^2\, .
\end{split}
\end{equation}
Hence, \eqref{eq:intfin} follows from Corollary \ref{coro:ell2Bound}.
\end{proof}

\begin{remark}
Note that by virtue of Lemma~\ref{lemma:SpeedBound}, we are able to get a $\delta$-free upper bound on the left-hand side of \eqref{eq:intfin}.  
Indeed, by definition of $\nabla V$ as per~\eqref{eq:KtoUV} and using Assumption ({{\sf A3}}), we have the $\delta$-free bound:
\begin{align}
\|\nabla V\|_{\Lp^\infty(\Omega)} \le \nu_0 \|K^\delta\|_{\Lp^1(\Omega)} \|\nabla f\|_{\Lp^\infty(\Omega)}  =  \|\nabla f\|_{\Lp^\infty(\Omega)} < C_*\,.
\end{align}
In addition, by Remark~\ref{rem:delta-ind1}, $\|K^{\delta}\cv\rho^\delta(0)\|_{\Lp^2(\Omega)}^2$ has $\delta$-free bound.
\end{remark}

\section{Global convergence: Proof of Theorems \ref{TH:CONVPME} and \ref{TH:ENDTOEND}}\label{app:convglob}

We start by showing that $\rho^\delta$ admits a limit in a suitable functional space as $\delta \to 0$.

\begin{lemma}[Existence of converging subsequence]\label{lemma:rcseq}
Let $\rho^\delta\in \Cont^{2,1}(\Omega^\delta,$ $[0,T])$ be the unique solution of the PDE \eqref{eq:PDEbis} with initial and boundary
conditions \eqref{eq:BCbis}.
Then, the family $(\rho^\delta)_{\delta>0}$ is relatively compact in
the space $\Cont([0,T],\cuP_2(\Omega))$. In particular any sequence
$(\rho^{\delta_n})_{n\ge 1}$, admits
a converging subsequence.
\end{lemma}
\begin{proof}
This follows from the Ascoli-Arzel\'a's theorem. Notice that $\cuP_2(\Omega)$ is compact due to the compactness of $\Omega$. Therefore, it is sufficient to prove that
the family is equicontinuous. Using the representation in terms of nonlinear dynamics (cf. Appendix \ref{app:nonlinear}), we have 
\begin{align}
W_2(\rho_t,\rho_s)^2 \le\E\big\{\big|\bX_t-\bX_s\big|^2\big\}\, .
\end{align}
Note that we omit for simplicity the dependence on $\delta$. Recall that the nonlinear dynamic satisfies (for $\bb(\bx,t) \equiv -\nabla\Psi(\bx,\rho_t)$)
\begin{align}
\bX_t &= \int_0^t\bb(\bX_r,r)\, \de r+ \sqrt{2\tau}\,\bB_t+ \bPhi_t\\
& \equiv \bV_{t}+\sqrt{2\tau}\,\bB_t+ \bPhi_t\, .
\end{align}
By \cite[Theorem 2.2]{slominski2001euler}, we have
\begin{align}
\E\big\{|\bX_t-\bX_s|^2\big\} & \le C_*\tau \E\{[\bB]_{s}^t\} +C_* \E\{(|\bV|_{s}^t)^2\}\, .
\end{align}
where $[\bB]_{s}^t$ denotes the quadratic variation of $\bB$, and $|\bV|_{s}^t$ the total variation of $\bV$ between times
$s$ and $t$. We thus have
\begin{align}
W_2(\rho_t,\rho_s)^2&\le \E\big\{|\bX_t-\bX_s|^2\big\} \le C_*\tau (t-s)+C_* \E\left\{\left(\int_s^t |\bb(\bX_r,r)| \,\de r\right)^2\right\}\nonumber\\
&\le  C_*\tau (t-s)+C_*(t-s) \int_{s}^t \E\left\{ |\bb(\bX_r,r)|^2 \right\} \,\de r\, .
\end{align}
Hence, in order to prove uniform continuity, it is sufficient to show that, for $s,t\le T$, $\int_s^t\E\big\{|\bb(\bX_r,r)|^2\big\}\, \de r \le C$
where $C$ is bounded uniformly in $\delta$.
In order to show that this is the case, notice that
\begin{align}
\int_s^t&\E\big\{|\bb(\bX_r,r)|^2\big\}\, \de r = \int_s^t\E\big\{|\nabla V(\bX_r) +\nabla U\cv \rho(\bX_r,r)|^2\big\}\, \de r\\
& \le 2(t-s) \|\nabla V\|_{\Lp^\infty(\Omega^\delta)}^2 + 2 \int_s^t\int|\nabla U\cv \rho(\bx,r)|^2\rho(\bx,r) \, \de\bx\, \de r\, ,
\end{align}
and the claim follows from Lemma \ref{lemma:SpeedBound}.
\end{proof}

We have now proved that the sequence $(\rho^{\delta_n})_{n\ge 1}$ admits a converging subsequence, where $\delta_n\to 0$ as $n\to \infty$. Fix such a convergent subsequence and, with an abuse of notation, also denote it by $(\rho^{\delta_n})_{n\ge 1}$. Let $\rho^\infty\in \Cont([0,T],\cuP_2(\Omega))$ be its limit. 

Recall that  $\rho^{\delta_n}$ is supported in $\Omega^{\delta_n}$. Hence, $K^{\delta_n} \ast \rho^{\delta_n}$ is supported in $\Omega$ and $K^{\delta_n} \ast \rho^{\delta_n}\in \cuP_2(\Omega)$. We will now show that $(K^{\delta_n} \ast \rho^{\delta_n})_{n \ge 1}$ has the same limit as $(\rho^{\delta_n})_{n \ge 1}$ in $\Cont([0,T],\cuP_2(\Omega))$.

\begin{lemma}\label{lemma:rcseq2}
The sequence $(K^{\delta_n} \ast \rho^{\delta_n})_{n \ge 1}$ also converges in  $\Cont([0,T],$ $\cuP_2(\Omega))$ to $\rho^\infty$. 
\end{lemma}

\begin{proof}
By Lemma \ref{lemma:rcseq}, the result is implied by the following claim:
\begin{equation}\label{eq:claimsup}
\lim_{n\to \infty}\sup_{0\le t\le T} W_2(K^{\delta_n} \ast \rho_t^{\delta_n}, \rho_t^{\delta_n})=0.
\end{equation}
Note that, for bounded $\Omega$,
\[
\left(\int |\bx-\by|^2 \gamma(\de\bx,\de\by)\right)^{1/2} \le {{\rm diam}}(\Omega)^{1/2} \left(\int |\bx-\by| \gamma(\de\bx,\de\by)\right)^{1/2}\,,
\] 
for any coupling $\gamma$ of the probability distributions of $\bx$ and $\by$. Hence,
\begin{align}
W_2(\rho_1,\rho_2) \le {{\rm diam}}(\Omega)^{1/2} W_1(\rho_1,\rho_2)\,.
\end{align}
 As an application,
\begin{equation}
W_2^2(K^{\delta_n} \ast \rho_t^{\delta_n}, \rho_t^{\delta_n})\le{{\rm diam}}(\Omega)^{1/2} W_1(K^{\delta_n} \ast \rho_t^{\delta_n}, \rho_t^{\delta_n}).
\end{equation}
Thus, it suffices to show that $\sup_{0\le t\le T} W_1(K^{\delta_n} \ast \rho_t^{\delta_n}, \rho_t^{\delta_n})\to 0$ as $n\to \infty$.

\comm{Note that
\begin{equation*}
W_1(K^{\delta_n} \ast \rho_t^{\delta_n}, \rho_t^{\delta_n})\le \mathbb E \{|(\bK_{\delta_n}+\bX_{\delta_n})-\bX_{\delta_n}|\} = \mathbb E\{|\bK_{\delta_n}|\},
\end{equation*}
where the random variables $\bK_{\delta_n}$ and $\bX_{\delta_n}$ have distributions $K^{\delta_n}$ and $\rho_t^{\delta_n}$, respectively. The quantity $\mathbb E\{|\bK_{\delta_n}|\}$ is $O(\delta)$, since $K$ has bounded absolute first  moment, which completes the proof.}
\end{proof}

We will now prove a stronger convergence result. 

\begin{lemma}[Convergence in $\Lp^2$]\label{lemma:convL2}
The measure $\rho^\infty$ has a density, which is the limit in $\Lp^2(\Omega\times [0, T])$ of the sequence $(K^{\delta_n} \ast \rho^{\delta_n})_{n \ge 1}$. 
\end{lemma}
\begin{proof}
By Corollary \ref{coro:ell2Bound}, we have that, for any $n\ge 1$, $K^{\delta_n} \ast \rho^{\delta_n}\in \Lp^2(\Omega \times [0,T])$. Let us show that $(K^{\delta_n} \ast \rho^{\delta_n})_{n \ge 1}$ is a Cauchy sequence in $\Lp^2(\Omega \times [0,T])$.

As $K^{\delta_n} \ast \rho_t^{\delta_n}\in \Lp^2(\Omega)$ for every $t\in[0,T]$, its Fourier transform exists and we denote it by $\reallywidehat{K^{\delta_n} \ast \rho^{\delta_n}}$. Hence, by applying Parseval's theorem, we have 
\begin{equation}\label{eq:parsedec}
\begin{split}
\limsup_{n, n'\to\infty} & \| K^{\delta_n} \ast \rho^{\delta_n}-K^{\delta_{n'}} \ast \rho^{\delta_{n'}}\|^2_{\Lp^2(\Omega\times [0, T])}\\
&=\limsup_{n, n'\to\infty}\int_0^T\| K^{\delta_n} \ast \rho_t^{\delta_n}-K^{\delta_{n'}} \ast \rho_t^{\delta_{n'}}\|^2_{\Lp^2(\Omega)}\,\de t\\
&=\limsup_{n, n'\to\infty}\int_0^T\int_{\mathbb R^d}| \reallywidehat{K^{\delta_n} \ast \rho_t^{\delta_n}}(\blambda)-\reallywidehat{K^{\delta_{n'}} \ast \rho_t^{\delta_{n'}}}(\blambda)|^2\,\de\blambda\,\de t.\\
\end{split}
\end{equation}
Fix $\Lambda>1$ and decompose the integral in the right-hand side of \eqref{eq:parsedec} as 
\begin{equation}\label{eq:parsedec2}
\begin{split}
\limsup_{n, n'\to\infty}&\int_0^T\int_{|\blambda|<\Lambda}| \reallywidehat{K^{\delta_n} \ast \rho_t^{\delta_n}}(\blambda)-\reallywidehat{K^{\delta_{n'}} \ast \rho_t^{\delta_{n'}}}(\blambda)|^2\,\de\blambda\,\de t\\
&+\limsup_{n, n'\to\infty}\int_0^T\int_{|\blambda|\ge \Lambda}| \reallywidehat{K^{\delta_n} \ast \rho_t^{\delta_n}}(\blambda)-\reallywidehat{K^{\delta_{n'}} \ast \rho_t^{\delta_{n'}}}(\blambda)|^2\,\de\blambda\,\de t.
\end{split}
\end{equation}
Consider the first term of \eqref{eq:parsedec2}. By Lemma \ref{lemma:rcseq2}, and since by Jensen's inequality $W_1(\rho_1,\rho_2) \le W_2(\rho_1,\rho_2)$ for any two distributions $\rho_1,\rho_2$, we have $W_1(K^{\delta_n}*\rho_t^{\delta_n} - K^{\delta_{n'}}*\rho_t^{\delta_{n'}}) \to 0$, as $n,n' \to \infty$. Since for the complex exponential functions $\|e^{i\<\blambda,\bx\>}\|_{{\rm Lip}}\le |\blambda|$, by definition of 1-Wasserstein distance, the integrand in the first term converges pointwise to $0$. Furthermore, the integrand is upper bounded by an integrable function, since $|\reallywidehat{K^{\delta_n} \ast \rho_t^{\delta_n}}(\blambda)|\le \|K^{\delta_n} \ast \rho_t^{\delta_n} \|_{\Lp^2(\Omega)}\le C$ 
for all $n$ and every $t\in[0,T]$. Hence, by dominated convergence, the first integral in  \eqref{eq:parsedec2} converges to $0$.

As for the second term of \eqref{eq:parsedec2}, the following chain of inequalities holds:
\begin{equation}\label{eq:parsedec3}
\begin{split}
\limsup_{n, n'\to\infty}&\int_0^T\int_{|\blambda|\ge \Lambda}| \reallywidehat{K^{\delta_n} \ast \rho_t^{\delta_n}}(\blambda)-\reallywidehat{K^{\delta_{n'}} \ast \rho_t^{\delta_{n'}}}(\blambda)|^2\,\de\blambda\,\de t\\
&\le 4\sup_{n\ge 1}\int_0^T\int_{|\blambda|\ge \Lambda}| \reallywidehat{K^{\delta_n} \ast \rho_t^{\delta_n}}(\blambda)|^2\,\de\blambda\,\de t\\
&\le \frac{4}{\Lambda^2}\sup_{n\ge 1}\int_0^T\int_{|\blambda|\ge \Lambda}|\blambda|^2\,| \reallywidehat{K^{\delta_n} \ast \rho_t^{\delta_n}}(\blambda)|^2\,\de\blambda\,\de t\\
&\le \frac{4}{\Lambda^2}\sup_{n\ge 1}\int_0^T\int_{\mathbb R^d}|\blambda|^2\,| \reallywidehat{K^{\delta_n} \ast \rho_t^{\delta_n}}(\blambda)|^2\,\de\blambda\,\de t\\
&= \frac{4}{\Lambda^2}\sup_{n\ge 1}\int_0^T\int_{\mathbb R^d}| \reallywidehat{\nabla K^{\delta_n} \ast \rho_t^{\delta_n}}(\blambda)|^2\,\de\blambda\,\de t\\
&= \frac{4}{\Lambda^2}\sup_{n\ge 1}\int_0^T\int_{\Omega}| \nabla K^{\delta_n} \ast \rho_t^{\delta_n}(\bx)|^2\,\de\bx\,\de t,\\
\end{split}
\end{equation}
where in the last equality we have applied again Parseval's theorem. 
By Lemma \ref{lemma:SpeedBound}, the integral in the right-hand side of \eqref{eq:parsedec3} is upper bounded by a constant independent of $n$. Therefore, as $\Lambda\to \infty$, the second term of \eqref{eq:parsedec2} converges to $0$. 

As a result, $(K^{\delta_n} \ast \rho^{\delta_n})_{n \ge 1}$ is a Cauchy sequence in $\Lp^2(\Omega \times [0,T])$. Let $\tilde{\rho}^{\infty} \in \Lp^2(\Omega \times [0,T])$ be its limit. Furthermore, by Lemma \ref{lemma:rcseq2}, $(K^{\delta_n} \ast \rho^{\delta_n})_{n \ge 1}$ has limit $\rho^{\infty}$ in $\Cont([0,T],\cuP_2(\Omega))$. Therefore, the measures $\rho^{\infty}_t(\de \bx)\de t$ and $\tilde{\rho}^{\infty}_t(\bx)\de\bx\,\de t$ coincide. This implies that  the measure $\rho^{\infty}_t$ has for almost every $t\in[0, T]$ the density  $\tilde{\rho}^{\infty}_t \in \Lp^2(\Omega \times [0,T])$, and the proof is complete.
\end{proof}

From now on, with an abuse of notation, we will use $\rho^\infty$ to denote also the density which is the limit in $\Lp^2(\Omega\times [0, T])$ of the sequence $(K^{\delta_n} \ast \rho^{\delta_n})_{n \ge 1}$. 

\begin{lemma}[Convergence to a weak solution of the limit PDE]\label{lemma:weaksol}
Let $\rho^\infty$ be the limit in $\Cont([0,T],\cuP_2(\Omega))$ of the converging sequence $(\rho^{\delta_n})_{n \ge 1}$. Then, $\rho^\infty$ is a weak solution of the PDE \eqref{eq:PMEbis} with initial and boundary conditions \eqref{eq:PME-BCbis}. 
\end{lemma}
\begin{proof}
By Lemma \ref{lemma:convL2}, we have that $\rho^\infty\in \Lp^2(\Omega \times [0,T])$. Choose a test function $h\in \Cont^{2,1}(\Omega\times [0,T])$,
satisfying $\<\bn(\bx),\nabla h(\bx,t)\>=0$ for all $\bx\in\partial\Omega, t\in[0,T]$. In order to prove the claim, we need to show that \eqref{eq:WeakPME_Def} holds. 
Throughout the proof, we will let $\lambda_n \equiv \lambda_{\delta_n}$.

Recall that, for any $n\ge 1$, $\rho^{\delta_n}$ is a weak solution of the PDE \eqref{eq:PDEbis} with initial and boundary
conditions \eqref{eq:BCbis}. Hence, by Definition \ref{def:weaksolPDE}, we have that
\begin{align}
\int_{\Omega^{\delta_n}}h^{\delta_n}(\bx,T)&\,\rho^{\delta_n}_T(\de\bx) - \int_{\Omega^{\delta_n}}h^{\delta_n}(\bx,0)\,\rho_0^{\delta_n}(\de\bx)  \label{eq:Weakappl}\\
&= 
\int_{0}^T\int_{\Omega^{\delta_n}}\Big[\partial_th^{\delta_n}(\bx,t) +\tau\Delta h^{\delta_n}(\bx,t)\nonumber \\
&\phantom{AAAA}-\<\nabla\Psi(\bx,\rho^{\delta_n}_t),\nabla h^{\delta_n}(\bx,t)\>
\Big] \,\rho^{\delta_n}_t(\de\bx)\, \de t\, , \nonumber
\end{align}
for any $h^{\delta_n}\in \Cont^{2,1}(\Omega^{\delta_n}\times [0,T])$
satisfying $\<\bn(\bx),\nabla h^{\delta_n}(\bx,t)\>=0$ for all $\bx\in\partial\Omega^{\delta_n}, t\in[0,T]$. Now, we set
\begin{equation}
h^{\delta_n}(\bx, t) = h(\bx/\lambda_n, t).
\end{equation}
By definition of $\Omega^\delta_n$, we have that $h^{\delta_n}\in \Cont^{2,1}(\Omega^{\delta_n}\times [0,T])$ since $h\in \Cont^{2,1}(\Omega\times [0,T])$. Furthermore, $\<\bn(\bx),\nabla h(\bx,t)\>=0$  for all $\bx\in\partial\Omega$, $t\in[0,T]$
 immediately implies that $\<\bn(\bx),\nabla h^{\delta_n}(\bx,t)\>=0$ for all $\bx\in\partial\Omega^{\delta_n}$, $t\in[0,T]$. 

Recall that
\begin{equation}
\Psi(\bx,\rho^{\delta_n}_t) =  V^{\delta_n}(\bx)+U^{\delta_n}\ast\rho^{\delta_n}_t(\bx) = -\nu_0 \,K^{\delta_n}\ast f(\bx) + \nu_0\,K^{\delta_n}\ast K^{\delta_n}\ast\rho^{\delta_n}_t(\bx).
\end{equation}
Thus, \eqref{eq:Weakappl} can be rewritten as 
\begin{align}
&\int_{\Omega^{\delta_n}}h^{\delta_n}(\bx,T)\,\rho^{\delta_n}_T(\de\bx) - \int_{\Omega^{\delta_n}}h^{\delta_n}(\bx,0)\,\rho_0^{\delta_n}(\de\bx)  \label{eq:Weakappl2}\\
&= 
\int_{0}^T\int_{\Omega^{\delta_n}}\left[\partial_th^{\delta_n}(\bx,t) +\nu_0\,\<\nabla K^{\delta_n}\ast f(\bx),\nabla h^{\delta_n}(\bx,t)\>
\right] \,\rho^{\delta_n}_t(\de\bx)\, \de t\,  \nonumber\\
&+
\int_{0}^T\hspace{-.5em}\int_{\Omega^{\delta_n}}\hspace{-.25em}\left[\tau\Delta h^{\delta_n}(\bx,t) -\nu_0\<\nabla K^{\delta_n}\ast K^{\delta_n}\ast  \rho^{\delta_n}_t(\bx),\nabla h^{\delta_n}(\bx,t)\>
\right] \,\rho^{\delta_n}_t(\de\bx)\, \de t\, . \nonumber
\end{align}

Since $(\rho^{\delta_n})_{n \ge 1}$ converges in  $\Cont([0,T],\cuP_2(\Omega))$ to $\rho^\infty$ by Lemma \ref{lemma:rcseq}, we have that
\begin{equation}\label{eq:fin1}
\begin{split}
&\lim_{n\to \infty}\int_{\Omega^{\delta_n}}h^{\delta_n}(\bx,T)\,\rho^{\delta_n}_T(\de\bx) = \int_{\Omega}h(\bx,T)\,\rho^{\infty}_T(\de\bx),\\
&\lim_{n\to \infty}\int_{0}^T\int_{\Omega^{\delta_n}}\left[\partial_th^{\delta_n}(\bx,t) +\tau\Delta h^{\delta_n}(\bx,t) 
\right] \,\rho^{\delta_n}_t(\de\bx)\, \de t \\
&\phantom{AAAA}= \int_{0}^T\int_{\Omega}\left[\partial_th(\bx,t) +\tau\Delta h(\bx,t) 
\right] \,\rho^{\infty}_t(\de\bx)\, \de t.\\
\end{split}
\end{equation}
Furthermore, since $\rho_0^{\delta_n}(\bx)= \lambda_{n}^{-d}\rho_{\sinit}(\bx/\lambda_n)$, we have that
\begin{equation}
\int_{\Omega^{\delta_n}}h^{\delta_n}(\bx,0)\,\rho_0^{\delta_n}(\de\bx)= \int_{\Omega}h(\bx,0)\,\rho_{\sinit}(\bx)\, \de\bx.
\end{equation}
 
Let us use the notation $h_t(\bx)= h(\bx, t)$ and $h_t^{\delta_n}(\bx)= h^{\delta_n}(\bx, t)$. Again, we set $\rho_t^{\delta_n}(\bx)=0$ for $\bx\not \in \Omega^{\delta_n}$. 
We further define $\trho^{\delta_n}_t(\bx) = \lambda_{n}^d\rho^{\delta_n}_t(\lambda_{n}\bx)$, which is a probability density on $\Omega$. 
Since $\rho^{\delta_n}_t(\,\cdot)\to \rho^{\infty}_t(\,\cdot)$ in $\cuP_2(\Omega)$ and $\lambda_{n}\to 1$, we have 
$\trho^{\delta_n}_t(\,\cdot)\to \rho^{\infty}_t(\,\cdot)$ in $\cuP_2(\Omega)$ as well.
Hence
\begin{align}
&\lim_{n\to \infty}\int_{0}^T\int_{\Omega^{\delta_n}}\<\nabla K^{\delta_n}\ast f(\bx),\nabla h^{\delta_n}_t(\bx)\> \,\rho^{\delta_n}_t(\bx)\, \de\bx\de t\\
&=
\lim_{n\to \infty}\lambda_{n}^{-1}\int_{0}^T\int_{\Omega}\<\nabla K^{\delta_n}\ast f(\lambda_n\bx),\nabla h_t(\bx)\> \,\trho^{\delta_n}_t(\bx)\, \de\bx\de t
\nonumber\\
&=\int_{0}^T\int_{\Omega}\< \nabla f(\bx),\nabla h^{\delta_n}_t(\bx)\> \,\rho^{\infty}_t(\bx)\, \de\bx\de t\, ,\label{eq:fin2}
\end{align}
where the last equality follows since $\lambda_n\to 1$, and $\nabla K^{\delta_n}\ast f(\lambda_n\bx)\to \nabla f(\bx)$
uniformly in $\Omega$.

Furthermore, we have that 
\begin{equation}\label{eq:intconvf2}
\begin{split}
&\bigg| \int_{0}^T\int_{\Omega^{\delta_n}}\<\nabla K^{\delta_n}\ast K^{\delta_n}\ast  \rho^{\delta_n}_t(\bx),\nabla h^{\delta_n}_t(\bx)\> \rho^{\delta_n}_t(\bx)\,\de\bx\, \de t\\
&\phantom{AAAA}+\frac{1}{2}\int_{0}^T\int_{\Omega} \Delta h_t(\bx)  (\rho_t^\infty(\bx))^2\,\de\bx\,\de t\bigg|\\
&\le \bigg| \int_{0}^T\int_{\Omega^{\delta_n}}\<\nabla K^{\delta_n}\ast K^{\delta_n}\ast  \rho^{\delta_n}_t(\bx),\nabla h_t^{\delta_n}(\bx)\> \rho^{\delta_n}_t(\bx)\, \de\bx\,\de t\\
&\phantom{AAAA}- \int_{0}^T\int_{\Omega}\<\nabla K^{\delta_n}\ast  \rho^{\delta_n}_t(\bx),\nabla h_t(\bx)\> K^{\delta_n}\ast \rho^{\delta_n}_t(\bx)\, \de\bx\,\de t\bigg|\\
&+ \bigg| \int_{0}^T\int_{\Omega}\<\nabla K^{\delta_n}\ast  \rho^{\delta_n}_t(\bx),\nabla h_t(\bx)\> K^{\delta_n}\ast \rho^{\delta_n}_t(\bx)\, \de\bx\,\de t\\
&\phantom{AAAA}+\frac{1}{2}\int_{0}^T\int_{\Omega} \Delta h_t(\bx)  ( K^{\delta_n}\ast\rho_t^{\delta_n}(\bx))^2\,\de\bx\,\de t\bigg|\\
&+ \bigg| \frac{1}{2}\int_{0}^T\int_{\Omega} \Delta h_t(\bx)  ( K^{\delta_n}\ast\rho_t^{\delta_n}(\bx))^2\,\de\bx\,\de t\\
&\phantom{AAAA}- \frac{1}{2}\int_{0}^T\int_{\Omega} \Delta h_t(\bx)  (\rho_t^{\infty}(\bx))^2\,\de\bx\,\de t\bigg|.\\
\end{split}
\end{equation} 
The second term in the right-hand side of \eqref{eq:intconvf2} is equal to $0$ by integration by parts. 
The third integral in the right-hand side of \eqref{eq:intconvf2} is upper bounded as follows:
\begin{equation}
\begin{split}
&\left| \frac{1}{2}\int_{0}^T\int_{\Omega} \Delta h_t(\bx)  ( K^{\delta_n}\ast\rho_t^{\delta_n}(\bx))^2\,\de\bx\,\de t- \frac{1}{2}\int_{0}^T\int_{\Omega} \Delta h_t(\bx)  (\rho_t^{\infty}(\bx))^2\,\de\bx\,\de t\right|\\
&\le C \int_{0}^T\int_{\Omega}  \left|( K^{\delta_n}\ast\rho_t^{\delta_n}(\bx))^2-(\rho_t^{\infty}(\bx))^2\right|\,\de\bx\,\de t\\
&\le C \int_{0}^T\int_{\Omega}  \left| K^{\delta_n}\ast\rho_t^{\delta_n}(\bx)-\rho_t^{\infty}(\bx)\right|\,\left| K^{\delta_n}\ast\rho_t^{\delta_n}(\bx)+\rho_t^{\infty}(\bx)\right|\,\de\bx\,\de t\\
&\le C \left(\int_{0}^T\int_{\Omega}  \left| K^{\delta_n}\ast\rho_t^{\delta_n}(\bx)-\rho_t^{\infty}(\bx)\right|^2\,\de\bx\,\de t\right)^{1/2}\\
&\phantom{AAAA}\left(\int_{0}^T\int_{\Omega}  \left| K^{\delta_n}\ast\rho_t^{\delta_n}(\bx)+\rho_t^{\infty}(\bx)\right|^2\,\de\bx\,\de t\right)^{1/2},\\
\end{split}
\end{equation}
which converges to $0$, as $(K^{\delta_n} \ast \rho^{\delta_n})_{n \ge 1}$ converges in $\Lp^2(\Omega\times [0, T])$ to $\rho^\infty$. The first term in the 
right-hand side of \eqref{eq:intconvf2} is upper bounded as follows:
\begin{equation}\label{eq:intconvf3_more}
\begin{split}
&\bigg| \int_{0}^T\int_{\Omega^{\delta_n}}\<\nabla K^{\delta_n}\ast K^{\delta_n}\ast  \rho^{\delta_n}_t(\bx),\nabla h_t^{\delta_n}(\bx)\> \rho^{\delta_n}_t(\bx)\, \de\bx\,\de t\\
&\phantom{AAAA}- \int_{0}^T\int_{\Omega}\<\nabla K^{\delta_n}\ast  \rho^{\delta_n}_t(\bx),\nabla h_t(\bx)\> K^{\delta_n}\ast \rho^{\delta_n}_t(\bx)\, \de\bx\,\de t\bigg|\\
&\le \bigg| \int_{0}^T\int_{\Omega^{\delta_n}}\<\nabla K^{\delta_n}\ast K^{\delta_n}\ast  \rho^{\delta_n}_t(\bx),\nabla h_t^{\delta_n}(\bx)\> \rho^{\delta_n}_t(\bx)\, \de\bx\,\de t\\
&\phantom{AAAA}- \int_{0}^T\int_{\Omega^{\delta_n}}\<\nabla K^{\delta_n}\ast K^{\delta_n}\ast  \rho^{\delta_n}_t(\bx),\nabla h_t(\bx)\> \rho^{\delta_n}_t(\bx)\, \de\bx\,\de t\bigg|\\
&+\bigg| \int_{0}^T\int_{\Omega^{\delta_n}}\<\nabla K^{\delta_n}\ast K^{\delta_n}\ast  \rho^{\delta_n}_t(\bx),\nabla h_t(\bx)\> \rho^{\delta_n}_t(\bx)\, \de\bx\,\de t\\
&\phantom{AAAA}- \int_{0}^T\int_{\Omega}\<\nabla K^{\delta_n}\ast  \rho^{\delta_n}_t(\bx),\nabla h_t(\bx)\> K^{\delta_n}\ast \rho^{\delta_n}_t(\bx)\, \de\bx\,\de t\bigg|.\\
\end{split}
\end{equation}
The first term is upper bounded using
\begin{align}
&\bigg| \int_{0}^T\int_{\Omega^{\delta_n}}\<\nabla K^{\delta_n}\ast K^{\delta_n}\ast  \rho^{\delta_n}_t(\bx),\nabla h_t^{\delta_n}(\bx)\> \rho^{\delta_n}_t(\bx)\, \de\bx\,\de t\nonumber\\
&\phantom{AAAA}- \int_{0}^T\int_{\Omega}\<\nabla K^{\delta_n}\ast K^{\delta_n}\ast  \rho^{\delta_n}_t(\bx),\nabla h_t(\bx)\> \rho^{\delta_n}_t(\bx)\, \de\bx\,\de t\bigg|\nonumber\\ 
&\le\|\nabla h^{\delta_n}-\nabla h\|_{\Lp^\infty(\Omega\times[0, T])} \bigg\|\,\Big|\nabla K^{\delta_n}\ast K^{\delta_n}\ast  \rho^{\delta_n}_t\Big| \times \rho^{\delta_n}_t \bigg\|_{\Lp^1(\Omega^{\delta_n}\times[0, T])}\, .
\label{eq:InftyOne}
\end{align}
Notice that
\begin{equation}\label{eq:neweqaa}
\begin{split}
& \bigg\|\,\Big|\nabla K^{\delta_n}\ast K^{\delta_n}\ast  \rho^{\delta_n}_t\Big| \times \rho^{\delta_n}_t \bigg\|_{\Lp^1(\Omega^{\delta_n}\times[0, T])}\\
& \stackrel{(a)}{\le}  \bigg\|\,\Big|\nabla K^{\delta_n}\ast K^{\delta_n}\ast  \rho^{\delta_n}_t\Big| \times \sqrt{\rho^{\delta_n}_t} \bigg\|_{\Lp^2(\Omega^{\delta_n}\times[0, T])}\,\bigg\|\sqrt{\rho^{\delta_n}_t} \bigg\|_{\Lp^2(\Omega^{\delta_n}\times[0, T])}\\
& \le \sqrt{T} \bigg\|\,\Big|\nabla K^{\delta_n}\ast K^{\delta_n}\ast  \rho^{\delta_n}_t\Big| \times \sqrt{\rho^{\delta_n}_t} \bigg\|_{\Lp^2(\Omega^{\delta_n}\times[0, T])},\\
\end{split}
\end{equation}
where $(a)$ follows from an application of Cauchy-Schwartz. By Lemma \ref{lemma:SpeedBound}, we deduce that the right-hand side of \eqref{eq:neweqaa} is bounded uniformly in $\delta_n$. Thus, the first term of \eqref{eq:intconvf3_more} converges to $0$ because of Eq.~\eqref{eq:InftyOne}.
As concerns the second term of \eqref{eq:intconvf3_more}, we have that 
\begin{equation}\label{eq:intconvf3}
\begin{split}
&\bigg| \int_{0}^T\int_{\Omega^{\delta_n}}\<\nabla K^{\delta_n}\ast K^{\delta_n}\ast  \rho^{\delta_n}_t(\bx),\nabla h_t(\bx)\> \rho^{\delta_n}_t(\bx)\, \de\bx\,\de t\\
&\phantom{AAAA}- \int_{0}^T\int_{\Omega}\<\nabla K^{\delta_n}\ast  \rho^{\delta_n}_t(\bx),\nabla h_t(\bx)\> K^{\delta_n}\ast \rho^{\delta_n}_t(\bx)\, \de\bx\,\de t\bigg|\\
&\le \bigg|\int_{0}^T\int_{\Omega}\int_{\Omega}\<\nabla K^{\delta_n}\ast  \rho^{\delta_n}_t(\by),\nabla h_t(\bx)\>K^{\delta_n}(\bx-\by) \rho^{\delta_n}_t(\bx)\, \de\by\, \de\bx\,\de t\\
&\phantom{AAAA}-\int_{0}^T\int_{\Omega}\int_{\Omega}\<\nabla K^{\delta_n}\ast  \rho^{\delta_n}_t(\by),\nabla h_t(\by)\>K^{\delta_n}(\bx-\by) \rho^{\delta_n}_t(\bx)\, \de\by\, \de\bx\,\de t\bigg|\\
&=\left|\int_{0}^T\int_{\Omega}\int_{\Omega}\<\nabla K^{\delta_n}\ast  \rho^{\delta_n}_t(\by),\nabla h_t(\bx)-\nabla h_t(\by)\>K^{\delta_n}(\bx-\by) \rho^{\delta_n}_t(\bx)\, \de\by\, \de\bx\,\de t\right|\\
&\le C \int_{0}^T\int_{\Omega}\int_{\Omega} |\bx-\by| \,|\nabla K^{\delta_n}\ast  \rho^{\delta_n}_t(\by)| K^{\delta_n}(\bx-\by) \rho^{\delta_n}_t(\bx)\,\de\by\, \de\bx\,\de t.\\
\end{split}
\end{equation}
Recall that $\rho^{\delta_n}_t$ is supported on $\Omega^{\delta_n}\subseteq \Omega$, and $\Omega$ is bounded. In addition, since the kernel $K$ has bounded support, the diameter of the support of $K^{\delta_n}$ is at most $\delta_n$ times a constant. Consequently, the last term in the right-hand side
of \eqref{eq:intconvf3} is upper bounded by
\begin{equation}\label{eq:intconvf4}
\begin{split}
&\delta_n C_1 \int_{0}^T\int_{\Omega}\int_{\Omega} |\nabla K^{\delta_n}\ast  \rho^{\delta_n}_t(\by)| K^{\delta_n}(\bx-\by) \rho^{\delta_n}_t(\bx)\,\de\by\, \de\bx\,\de t\\
&=\delta_n C_1\int_{0}^T\int_{\Omega} |\nabla K^{\delta_n}\ast  \rho^{\delta_n}_t(\by)| K^{\delta_n}\ast \rho^{\delta_n}_t(\by)\,\de\by\, \de t\\
&\le \delta_n C_1\left(\int_{0}^T\int_{\Omega} |\nabla K^{\delta_n}\ast  \rho^{\delta_n}_t(\by)|^2 \,\de\by\, \de t\right)^{1/2}\\
&\phantom{AAAA}\left(\int_{0}^T\int_{\Omega} ( K^{\delta_n}\ast \rho^{\delta_n}_t(\by))^2\,\de\by\, \de t\right)^{1/2}.
\end{split}
\end{equation}
By using that $K^{\delta_n} \ast \rho^{\delta_n}\in \Lp^2(\Omega\times [0, T])$ and the result of Lemma \ref{lemma:SpeedBound}, we have that the two last integrals are bounded uniformly in $\delta$. As a result, the right-hand side of \eqref{eq:intconvf4} converges to $0$, which implies that the right-hand side of \eqref{eq:intconvf2} also converges to $0$. By putting this fact together with \eqref{eq:fin1} and \eqref{eq:fin2}, the desired result follows.
\end{proof}

We have now proved that $(\rho^{\delta_n})_{n \ge 1}$ converges to \emph{a} weak solution of the limit PDE \eqref{eq:PMEbis}. 
In order to prove the uniqueness of the weak solutions of the limit PDE, we next prove a bound on $\|\rho^{\delta_n}_t\|_{\Lp^4(\Omega)}$, which along with Lemma \ref{lemma:unique} proves the uniqueness claim.

\begin{lemma}[Uniform bound in $\Lp^4$]\label{lemma:convL4}
Assume that $\rho_{\sinit}, f \in \Cont^\infty(\Omega)$ and consider the sequence $(\rho^{\delta_n})_{n \ge 1}$. Then, 
\begin{align}
\sup_{n\ge1, t\in [0,T_0]} \|\rho^{\delta_n}_t\|_{\Lp^4(\Omega)}\le C(\Omega)(1+T)\,,
\end{align}
where 
\begin{equation}
T_0 = \frac{1}{C(\Omega)(1+T)}, 
\end{equation}
for some bounded constant $0<C(\Omega)<\infty$.
\end{lemma}

\begin{proof}
For simplicity, we indicate the norms $\Lp^p(\Omega)$ by $\|\cdot\|_{p}$. For a function $g\in \Cont^m(\Omega)$, we let $\nabla^{\otimes m} g$ be the vector with coordinates ${\partial ^m}g/{(\partial_{i_1} \dotsc \partial_{i_m})}$, with $1\le i_1, i_2, \dotsc, i_m\le d$.  
The proof strategy to prove this lemma is to first bound $\|\nabla^{\otimes m} \rho^{\delta_n}_t\|_2$, for some $m\ge d/4$, and then apply the Gagliardo-Nirenberg interpolation inequality (cf. Lemma~\ref{lem:GN-interpolation}) to bound $\|\rho^{\delta_n}_t\|_4$. Throughout this proof, we will use $C$, $C_k$ and so on to denote
constants that can depend on the domain $\Omega$, but do not depend on $t$ or $\delta$.

Before proceeding, we need to establish some notations and definitions. 

For a function $g$ and an integer $k\ge 0$, we denote its Sobolev norms by
\begin{align}\label{eq:sobolev-norm}
\|g\|_{(k)} = \left(\sum_{m=0}^k \|\nabla^{\otimes m} g \|_2^2 \right)^{1/2}\,.
\end{align}
We will use the following relations on Sobolev norms (see~\cite[Equation (1.14)]{oelschlager2001sequence}):
\begin{align}\label{sobolev-relation}
\|g\|_{(k)} \le \begin{cases}
\bar{C}_k (\|g\|_2^2 + \|(-\Delta)^{k/2} g\|_2^2)^{1/2} \quad & \text{if } k \text{ is even\,,}\\
\bar{C}_k (\|g\|_2^2 + \|\nabla(-\Delta)^{(k-1)/2} g\|_2^2)^{1/2} \quad & \text{if } k \text{ is odd\,.}
\end{cases}
\end{align}
Instead of bounding $\|\nabla^{\otimes m} \rho^{\delta_n}_t\|_2$, we will bound the dominating quantity $\|\rho^{\delta_n}_t\|_{(m)}$. To this end, we follow a similar strategy as in~\cite{oelschlager2001sequence}. Namely,
we derive descriptions of the evolution of $\|(-\Delta)^m \rho^{\delta_n}_t\|_2$ and $\|(-\Delta)^m (\rho^{\delta_n}_t\ast K^{\delta_n} - f)\|_2$. More precisely, we derive a recursive equation (on $m$) for the evolution of a suitably chosen linear combination of these two quantities.  

Since $\rho^{\delta_n}$ is a solution of the PDE~\eqref{eq:PDEbis}, we have
\begin{align}\label{eq:Dmrho}
\partial_t (-\Delta)^m \rho^{\delta_n}_t(\bx) = -\tau (-\Delta)^{m+1}  \rho^{\delta_n}_t(\bx) + (-\Delta)^m\nabla\cdot\left(\rho^{\delta_n}_t(\bx) \nabla(V + U \ast \rho^{\delta_n}_t) \right)
\end{align}
Following along the same lines as in derivation of~\cite[Equation (3.12)]{oelschlager2001sequence}, we obtain
\begin{align}
\frac{\de}{\de t} &\|(-\Delta)^m \rho^{\delta_n}_t\|_2^2\nonumber\\
&\le (CC_m-2\tau) \|\nabla(-\Delta)^m\rho^{\delta_n}_t\|_2^2\nonumber\\
\,&+\frac{2}{C} \|\rho^{\delta_n}_t\|_\infty \Big\<\nabla(-\Delta)^m (V + U\ast \rho^{\delta_n}_t ), \rho^{\delta_n}_t \nabla(-\Delta)^m (V + U\ast \rho^{\delta_n}_t ) \Big\>\nonumber\\
\,&+\frac{\tilde{C}_m}{C} \sum_{q=1}^m\Big( \Big \|V + U\ast \rho^{\delta_n}_t  \Big\|_{(2m+1-q)}^2 \|\rho^{\delta_n}_t\|^2_{(q+1+d/2)} \nonumber\\
\,&+ \Big \|V + U\ast \rho^{\delta_n}_t  \Big\|_{(q+1+d/2)}^2 \|\rho^{\delta_n}_t\|^2_{(2m+1-q)} \Big)\,,\label{d-rho-1}
\end{align} 
where $C_m$ and $\tilde{C}_m$ are positive constants that depend on $m$ and $C>0$ is a constant which can be chosen arbitrarily.  

We set $m = \lceil 1+d/2\rceil$ for which we can upper bound the right-hand side of~\eqref{d-rho-1} as 
\begin{align}
\frac{\de}{\de t} &\Big\|(-\Delta)^m \rho^{\delta_n}_t\Big\|_2^2\nonumber\\
&\le (CC_m-2\tau) \Big\|\nabla(-\Delta)^m\rho^{\delta_n}_t\Big\|_2^2\nonumber\\
\,&+\frac{2}{C} \|\rho^{\delta_n}_t\|_\infty \Big\<\nabla(-\Delta)^m (V + U \ast \rho_t ), \rho^{\delta_n}_t \nabla(-\Delta)^m (V + U\ast \rho^{\delta_n}_t ) \Big\>\nonumber\\
\,&+\frac{2m\tilde{C}_m}{C} \Big \|V + U\ast \rho^{\delta_n}_t  \Big\|_{(2m)}^2  \|\rho^{\delta_n}_t\|^2_{(2m)} \,.\label{d-rho-1-1}
\end{align} 
We next move to the next quantity. Write
\begin{align}
&\frac{\de}{\de t} \Big\|(-\Delta)^m (K^{\delta_n}\ast \rho^{\delta_n}_t - f )\Big\|_2^2\nonumber\\
&\,\le 2\Big\<(-\Delta)^m (K^{\delta_n}\ast \rho^{\delta_n}_t - f) , \partial_t (-\Delta)^m \rho^{\delta_n}_t \ast K^{\delta_n} \Big\>\nonumber\\
&\,= 2\Big\<(-\Delta)^m (V^{\delta_n} + U^{\delta_n}\ast \rho^{\delta_n}_t ) , \partial_t (-\Delta)^m \rho^{\delta_n}_t \Big\>\nonumber\\
&\,= -2\tau\Big\<(-\Delta)^m (V^{\delta_n} + U^{\delta_n}\ast \rho^{\delta_n}_t ) , (-\Delta)^{m+1} \rho^{\delta_n}_t \Big\> \nonumber\\
&\,+ 2\Big\<(-\Delta)^m (V^{\delta_n} + U^{\delta_n}\ast \rho^{\delta_n}_t ) ,  (-\Delta)^m\nabla\cdot\left(\rho^{\delta_n}_t(\bx) \nabla(V^{\delta_n} + U^{\delta_n}\ast \rho^{\delta_n}_t) \right)  \Big\>\,,\label{d-rho-2}
\end{align}
where the last step follows from~\eqref{eq:Dmrho}. Note that the first term on the right-hand side can be bounded as
\begin{align}
&-2\tau\Big\<(-\Delta)^m (V^{\delta_n} + U^{\delta_n}\ast \rho^{\delta_n}_t ) , (-\Delta)^{m+1} \rho^{\delta_n}_t \Big\>\nonumber\\
&= -2\tau\Big\<(-\Delta)^{m+1} (K^{\delta_n}\ast f) , (-\Delta)^{m} \rho^{\delta_n}_t \Big\> -2\tau \Big\|\nabla(-\Delta)^m  (K^{\delta_n}\ast \rho^{\delta_n}_t )\Big\|_2^2\nonumber\\
&\le 2\tau \|(-\Delta)^{m+1} f\|_2 \| (-\Delta)^{m} \rho^{\delta_n}_t \|_2-2\tau \Big\|\nabla(-\Delta)^m  (K^{\delta_n}\ast \rho^{\delta_n}_t )\Big\|_2^2\,, \label{d-rho-3}
\end{align}
where the last step follows from Young's convolution inequality and the fact that $\|K^{\delta_n}\|_1 = 1$.

The second term in~\eqref{d-rho-2} can be bounded following the same lines as in derivation of~\cite[Equations (3.3) and (3.16)]{oelschlager2001sequence}, which along with~\eqref{d-rho-3} gives 
\begin{align}
\frac{\de}{\de t}  &\Big\|(-\Delta)^m (K^{\delta_n}\ast \rho^{\delta_n}_t - f )\Big\|_2^2\nonumber\\
&\le CC_m \|\nabla(-\Delta)^m(K^{\delta_n}\ast \rho^{\delta_n}_t - f )\|_2^2 +2\tau \|(-\Delta)^{m+1} f\|_2 \| (-\Delta)^{m} \rho^{\delta_n}_t \|_2\nonumber\\
&\phantom{AAAA}-2\tau \Big\|\nabla(-\Delta)^m  (K^{\delta_n}\ast \rho^{\delta_n}_t )\Big\|_2^2 \nonumber\\
\,&-2 \Big\<\nabla(-\Delta)^m (V + U\ast \rho^{\delta_n}_t ), \rho^{\delta_n}_t \nabla(-\Delta)^m (V + U\ast \rho^{\delta_n}_t ) \Big\>\nonumber\\
\,&+\frac{\tilde{C}_m}{C} \sum_{q=1}^m\Big( \Big \|V + U\ast \rho^{\delta_n}_t  \Big\|_{(2m+1-q)}^2 \|\rho^{\delta_n}_t\|^2_{(q+1+d/2)} \nonumber\\
\,&+ \Big \|V + U\ast \rho^{\delta_n}_t  \Big\|_{(q+1+d/2)}^2 \|\rho^{\delta_n}_t\|^2_{(2m+1-q)} \Big)\,.\label{d-rho-4}
\end{align}
Since $f\in \Cont^\infty(\Omega)$, there exists constant $M>0$, such that $\|(-\Delta)^{m+1} f\|_2\le M$, $\|\nabla(-\Delta)^{m} f\|_2\le M$. Using the particular choice of $m$, we can upper bound the right-hand side of~\eqref{d-rho-4} as
\begin{align}
\frac{\de}{\de t}  &\Big\|(-\Delta)^m (K^{\delta_n}\ast \rho^{\delta_n}_t - f )\Big\|_2^2\nonumber\\
&\le (2CC_m-2\tau) \|\nabla(-\Delta)^m(K^{\delta_n}\ast \rho^{\delta_n}_t )\|_2^2 + 2\tau M \| (-\Delta)^{m} \rho^{\delta_n}_t \|_2+ 2M^2 CC_m  \nonumber\\
\,&-2 \Big\<\nabla(-\Delta)^m (V + U\ast \rho^{\delta_n}_t ), \rho^{\delta_n}_t \nabla(-\Delta)^m (V + U\ast \rho^{\delta_n}_t ) \Big\>\nonumber\\
\,&+\frac{2m\tilde{C}_m}{C} \Big \|V + U\ast \rho^{\delta_n}_t  \Big\|_{(2m)}^2 \|\rho^{\delta_n}_t\|^2_{(2m)} \,.\label{d-rho-5}
\end{align}
Define $C_1\equiv 2 \|\rho_{\sinit}\|_{(2m)}$ and let
\[
T_n \equiv \inf\Big\{t\in[0,T]:\, \|\rho^{\delta_n}_t\|_{(2m)} > C_1\Big\} \wedge T\,,
\]
for $n\ge 1$. Clearly, $T_n>0$ by choice of $C_1$. In addition, by applying Sobolev's inequality (see e.g.~\cite[Equation (1.12)]{oelschlager2001sequence}), we have 
\[
\|\rho^{\delta_n}_t\|_\infty \le C_2\| \rho^{\delta_n}_t\|_{(2m)}\le C_1C_2 \,,\quad \text{for } t\in [0,T_n],\, n\ge 1\,.
\] 
where $C_2>0$ is a constant depending on $d$. We let $C_\ast \equiv C_1 C_2/C$. Recall that the constant $C>0$ in~\eqref{d-rho-1-1} and~\eqref{d-rho-5} was arbitrary. We choose it in a way that $C<\tau/(2C_m)$.  We then consider the evolution of the following linear combination of the two quantities we analyzed above. Note that by Equations \eqref{d-rho-1-1} and~\eqref{d-rho-5}, we have for $t\in [0, T_n]$,
\begin{align}
\frac{\de}{\de t} &\left(\Big\|(-\Delta)^m \rho^{\delta_n}_t\Big\|_2^2 + C_\ast \Big\|(-\Delta)^m (K^{\delta_n}\ast \rho^{\delta_n}_t - f )\Big\|_2^2 \right)\nonumber\\
&\le -\frac{3\tau}{2}\Big\|\nabla(-\Delta)^m\rho^{\delta_n}_t\Big\|_2^2 - C_\ast \tau \|\nabla(-\Delta)^m(K^{\delta_n}\ast \rho^{\delta_n}_t )\|_2^2\nonumber \\
&\phantom{AAAA}+ C_\ast\left(2\tau M \| (-\Delta)^{m} \rho^{\delta_n}_t \|_2+ \tau M^2\right)\nonumber\\
&\phantom{AAAA}+\frac{2m\tilde{C}_m}{C} (1+C_\ast) \Big \|V + U\ast \rho^{\delta_n}_t  \Big\|_{(2m)}^2 \|\rho^{\delta_n}_t\|^2_{(2m)} \nonumber\\
&\le C_\ast\left(2\tau M \| (-\Delta)^{m} \rho^{\delta_n}_t \|_2+ \tau M^2\right) \nonumber\\
&\phantom{AAAA}+\frac{2m\tilde{C}_m}{C} (1+C_\ast) \left( \|\rho^{\delta_n}_t\|^2_{(2m)}  +  \Big \|V+ U\ast \rho^{\delta_n}_t  \Big\|_{(2m)}^2 \right)^2\nonumber\\
&\stackrel{(a)}{\le} \tau M^2 C_{\ast} + C_3 \left( \|\rho^{\delta_n}_t\|^2_{(2m)}  +  C_\ast \Big \|V + U\ast \rho^{\delta_n}_t  \Big\|_{(2m)}^2 \right)^2\nonumber\\
&\stackrel{(b)}{\le} \tau M^2 C_{\ast} + C_3 \left( \|\rho^{\delta_n}_t\|^2_{(2m)}  +  C_\ast \Big \|-f+ K^{\delta_n}\ast \rho^{\delta_n}_t  \Big\|_{(2m)}^2 \right)^2 \,,\label{d-rho-6}
\end{align}
where in $(a)$ we use the fact that $\| (-\Delta)^{m} \rho^{\delta_n}_t \|_2\le \|\rho^{\delta_n}_t\|_{(2m)}$, which follows immediately from~\eqref{eq:sobolev-norm}; $(b)$ follows from the 
fact that for any function $g\in \Lp^2(\Omega)$, $\|g\ast K^{\delta_n}\|_2\le \|K^{\delta_n}\|_1 \|g\|_2 = \|g\|_2$, by Young's inequality for convolution.

Another observation that will be used later is that
\begin{align}\label{d-rho-7}
\frac{\de}{\de t} \left( \|\rho^{\delta_n}_t\|_{2}^2 + C_{\ast} \| K^{\delta_n}\ast \rho^{\delta_n}_t - f \|_2^2  \right)\le 0\,.
\end{align}
This claim follows by repeating the same argument we had to derive~\eqref{d-rho-6}, for $m=0$. In this case, we have analogous equations to~\eqref{d-rho-1-1} and~\eqref{d-rho-5},
where only the first two terms appear.
   
Next note that by~\eqref{sobolev-relation}, we have for $t\in [0,T_n]$,
\begin{align}
&\|\rho^{\delta_n}_t\|_{(2m)}^2 + C_\ast \|K^{\delta_n}\ast \rho^{\delta_n}_t - f \|_{(2m)}^2\nonumber\\
&\le \bar{C}_{m}\Bigg(\|\rho^{\delta_n}_t\|_{2}^2 + \|(-\Delta)^{m}\rho^{\delta_n}_t\|_{2}^2 \nonumber\\
&\phantom{AAAA}+ C_\ast \left(\| K^{\delta_n}\ast \rho^{\delta_n}_t - f \|_2^2 + \Big\|(-\Delta)^m (K^{\delta_n}\ast \rho^{\delta_n}_t - f )\Big\|_2^2\right) \Bigg) \nonumber\\
&\le \bar{C}_{m} \left(\|\rho_{\sinit}\|_{(2m)}^2  + C_\ast \| K^{\delta_n}\ast \rho_{\sinit} - f \|_{(2m)}^2 \right) \nonumber\\
&+ \bar{C}_{m} \tau M^2 C_\ast t + \bar{C}_{m}  C_3 \int_0^t \left( \|\rho^{\delta_n}_s\|^2_{(2m)}  +  C_\ast \Big \|K^{\delta_n}\ast \rho^{\delta_n}_s -f  \Big\|_{(2m)}^2 \right)^2 \de s\,, \label{d-rho-7-1}
%
\end{align} 
where the last step is a result of~\eqref{d-rho-7} and~\eqref{d-rho-6}.  Let us stress that $\bar{C}_m$, $C_\ast$, $C_3$ are constants that are independent of $n$.

We further note that
\begin{align}
\|\rho_{\sinit}\|_{(2m)}^2  + C_\ast \| K^{\delta_n}\ast \rho_{\sinit} - f \|_{(2m)}^2 &\le \|\rho_{\sinit}\|_{(2m)}^2 (1+ 2C_\ast)+ 2C_\ast \|f\|_{(2m)}^2\nonumber\\
& \le (1+2C_{\ast}) \frac{C_1^2}{4}+ 2C_\ast \|f\|_{(2m)}^2\,,\label{d-rho-8}
\end{align}
for $n\ge 1$. Here, the first step is a result of triangle inequality and the Young's inequality for convolution along with the fact that $\|K^{\delta_n}\|_1 = 1$. The second step follows from definition of 
$C_1$. Since $f\in \Cont^\infty(\Omega)$, $\|f\|_{(2m)}^2$ is uniformly bounded over $\Omega$. We denote the right-hand side of \eqref{d-rho-8} by the constant $C_4$. Using bound~\eqref{d-rho-8}
into~\eqref{d-rho-7-1} results in
\begin{align}
&\|\rho^{\delta_n}_t\|_{(2m)}^2 + C_\ast \|K^{\delta_n}\ast \rho^{\delta_n}_t - f \|_{(2m)}^2\nonumber\\
&\le \bar{C}_{m} C_4 + \bar{C}_{m} \tau M^2 C_\ast t \nonumber\\
&\phantom{AAAA}+ \bar{C}_{m}  C_3 \int_0^t \left( \|\rho^{\delta_n}_s\|^2_{(2m)}  +  C_\ast \Big \|K^{\delta_n}\ast \rho^{\delta_n}_s -f  \Big\|_{(2m)}^2 \right)^2 \de s\,, \label{d-rho-9}
\end{align}
for $t\in [0,T_n]$. By employing a generalization of Gronwall's inequality (cf. Lemma~\ref{lem:gronwall} and Remark~\ref{rem:gronwall}) we get
\begin{align}
\|\rho^{\delta_n}_t\|_{(2m)}^2 + C_\ast \|K^{\delta_n}\ast \rho^{\delta_n}_t - f \|_{(2m)}^2 \le \frac{\bar{C}_{m} C_4 + \bar{C}_{m} \tau M^2 C_\ast T_n}{1 - (\bar{C}_{m} C_4 + \bar{C}_{m} \tau M^2 C_\ast T_n) \bar{C}_m C_3 t}\,. \label{d-rho-10}
\end{align}
Therefore, for $t\in [0, T_0]$, with
\begin{equation}
T_0 = \frac{1}{2\bar{C}_{m} C_4 + 2\bar{C}_{m} \tau M^2 C_\ast T},
\end{equation}
we have that
\begin{align}
\|\rho^{\delta_n}_t\|_{(2m)}^2 + C_\ast \|K^{\delta_n}\ast \rho^{\delta_n}_t - f \|_{(2m)}^2 \le C_5 + C_6 T\,, \label{d-rho-11}
\end{align}
with $C_5 \equiv \bar{C}_m C_4$ and $C_6\equiv \bar{C}_m \tau M^2  C_\ast$. Note that $C_5$, $C_6$ and $T_0$ are independent of $n$, but depend on $d$.
Let $m_0 = \lceil d/4\rceil$. Then, by the choice of $m = 1+\lceil d/2\rceil$ we have $\|\nabla^{\otimes {m_0}} \rho^{\delta_n}_t\|_2\le \|\rho^{\delta_n}_t\|_{(2m)}$, and hence  as a result of~\eqref{d-rho-11}, we obtain
\begin{align}
\sup_{n\ge1, t\in [0,T_0]} \|\nabla^{\otimes {m_0}} \rho^{\delta_n}_t\|_2  \le C_5 + C_6 T\,, \label{d-rho-12}
\end{align}
Finally, by applying Gagliardo-Nirenberg interpolation inequality (cf. Lemma~\ref{lem:GN-interpolation}) we get 
\[
\sup_{n\ge1, t\in [0,T_0]} \|\rho^{\delta_n}_t\|_4 \le C_7+C_8 T\,,
\]
for some constant $C_7, C_8>0$, which completes the proof. 
\end{proof}

\begin{lemma}[Convergence to the unique weak solution of limit PDE]\label{lemma:convuniquePDE}
Let $\rho^\infty$ be the limit in $\Cont([0,T],\cuP_2(\Omega))$ of the converging sequence $(\rho^{\delta_n})_{n \ge 1}$. Then, $\rho^\infty$ is the unique weak solution of the PDE \eqref{eq:PMEbis}  in $\Lp^4(\Omega\times [0,T])$  with initial and boundary conditions \eqref{eq:PME-BCbis}. 
\end{lemma}

\begin{proof}
From Lemma \ref{lemma:convL2}, we have that the sequence $(K^{\delta_n} \ast \rho^{\delta_n})_{n \ge 1}$ converges in $\Lp^2(\Omega\times [0,T])$ to $\rho^\infty$. Furthermore, by Lemma \ref{lemma:convL4}, $\|\rho^{\delta}_t\|_{\Lp^4(\Omega)}\le C(1+T)$ for any $t\in [0, T_0]$, where $C$ is a universal constant. By using Young's convolution inequality, we also deduce that $\|K^{\delta_n} \ast\rho^{\delta}_t\|_{\Lp^4(\Omega)}\le C(1+T)$ for any $t\in [0, T_0]$.

Note that $\Lp^4(\Omega)$ is a reflexive Banach space. Thus, by applying the Banach-Alaoglu theorem, every bounded sequence in $\Lp^4(\Omega)$ has a weakly convergent subsequence. This means that there exist a subsequence $K^{\delta_{n_k}} \ast\rho^{\delta_{n_k}}$ and a function $\tilde{\rho}\in \Lp^4(\Omega)$ such that, for any $g\in \Lp^{4/3}(\Omega)$, we have 
\begin{equation}
\int_{\Omega} (K^{\delta_{n_k}} \ast\rho^{\delta_{n_k}}(\bx) - \tilde{\rho}(\bx)) g(\bx)\,\de\bx \to 0.
\end{equation}
Now, since $\Omega$ is bounded, $K^{\delta_{n_k}} \ast\rho^{\delta_{n_k}}$ and $\tilde{\rho}$ are also in $\Lp^2(\Omega)$ (as they are in $\Lp^4(\Omega)$). Thus, $K^{\delta_{n_k}} \ast\rho^{\delta_{n_k}} - \tilde{\rho}$ is in $\Lp^2(\Omega)$, hence it is also in $\Lp^{4/3}(\Omega)$. As a result, we can pick $g = K^{\delta_{n_k}} \ast\rho^{\delta_{n_k}}- \tilde{\rho}$ and obtain 
\begin{equation}
\int_\Omega |K^{\delta_{n_k}} \ast\rho^{\delta_{n_k}}(\bx) - \tilde{\rho}(\bx)|^2\,\de \bx \to 0.
\end{equation}
Therefore, $\tilde{\rho}$ is the limit in $\Lp^2(\Omega)$ of the sequence $K^{\delta_{n_k}} \ast\rho^{\delta_{n_k}}$. By uniqueness of the limit, we conclude that $\tilde{\rho} = \rho^\infty$. As a result, $\rho^\infty\in \Lp^4(\Omega)$ for any $t\in [0, T_0]$, which implies that $\rho^\infty\in \Lp^4(\Omega\times [0, T_0])$. Thus, by Lemma \ref{lemma:weaksol} and Lemma \ref{lemma:unique}, $\rho^\infty$ is the unique weak solution of the PDE \eqref{eq:PMEbis} for $t\in [0, T_0]$. Note that $T_0$ is decreasing with $T$. Thus, we can repeat the same argument with $T-T_0$ instead of $T$ and obtain that  $\rho^\infty$ is the unique weak solution of the PDE \eqref{eq:PMEbis} for $t\in [T_0, 2\,T_0]$. By iterating this procedure $T/T_0$ times, the result follows.  
\end{proof}

\comm{At this point, we state and prove a lemma showing that the sequence $(\rho_t^{\delta_n})_{n \ge 1}$ converges in $\Lp^2(\Omega)$ to $\rho^\infty_t$.}
 
\comm{\begin{lemma}\label{lemma:newL2}
For almost all $t\in [0, T]$, the measure $\rho^\infty_t$ is the limit in $\Lp^2(\Omega)$ of the sequence $(\rho_t^{\delta_n})_{n \ge 1}$. 
\end{lemma}}

\begin{proof}
\comm{The proof is similar to that of Lemma \ref{lemma:convL2}. Suppose that $t\in [0, T_0]$, where $T_0$ is defined in the statement of Lemma \ref{lemma:convL4}. Note that, for any $n\ge 1$, $\rho^{\delta_n}\in \Lp^2(\Omega)$. Let us show that $(\rho^{\delta_n})_{n \ge 1}$ is a Cauchy sequence in $\Lp^2(\Omega)$.}

\comm{As $\rho_t^{\delta_n}\in \Lp^2(\Omega)$ for every $t\in[0,T_0]$, its Fourier transform exists and we denote it by $\widehat{\rho^{\delta_n}}$. Hence, by applying Parseval's theorem, we have 
\begin{equation}\label{eq:parsedecnew}
\begin{split}
\limsup_{n, n'\to\infty}  \| \rho^{\delta_n}- \rho^{\delta_{n'}}\|^2_{\Lp^2(\Omega)}&=\limsup_{n, n'\to\infty}\| \rho_t^{\delta_n}- \rho_t^{\delta_{n'}}\|^2_{\Lp^2(\Omega)}\\
&=\limsup_{n, n'\to\infty}\int_{\mathbb R^d}| \widehat{ \rho_t^{\delta_n}}(\blambda)-\widehat{\rho_t^{\delta_{n'}}}(\blambda)|^2\,\de\blambda.\\
\end{split}
\end{equation}
Fix $\Lambda>1$ and decompose the integral in the right-hand side of \eqref{eq:parsedecnew} as 
\begin{equation}\label{eq:parsedec2new}
\begin{split}
\limsup_{n, n'\to\infty}&\int_{|\blambda|<\Lambda}| \widehat{\rho_t^{\delta_n}}(\blambda)-\widehat{\rho_t^{\delta_{n'}}}(\blambda)|^2\,\de\blambda+\limsup_{n, n'\to\infty}\int_{|\blambda|\ge \Lambda}| \widehat{\rho_t^{\delta_n}}(\blambda)-\widehat{\rho_t^{\delta_{n'}}}(\blambda)|^2\,\de\blambda.
\end{split}
\end{equation}
Consider the first term of \eqref{eq:parsedec2new}. By Lemma \ref{lemma:rcseq}, and since by Jensen's inequality $W_1(\rho_1,\rho_2) \le W_2(\rho_1,\rho_2)$ for any two distributions $\rho_1,\rho_2$, we have $W_1(\rho_t^{\delta_n} - \rho_t^{\delta_{n'}}) \to 0$, as $n,n' \to \infty$. Since for the complex exponential functions $\|e^{i\<\blambda,\bx\>}\|_{{\rm Lip}}\le |\blambda|$, by definition of 1-Wasserstein distance, the integrand in the first term converges pointwise to $0$. Furthermore, the integrand is upper bounded by an integrable function, since $|\widehat{\rho_t^{\delta_n}}(\blambda)|\le \|\rho_t^{\delta_n} \|_{\Lp^2(\Omega)}\le C$ 
for all $n$ and every $t\in[0,T_0]$. Hence, by dominated convergence, the first integral in  \eqref{eq:parsedec2new} converges to $0$.}

\comm{As for the second term of \eqref{eq:parsedec2new}, the following chain of inequalities holds:
\begin{equation}\label{eq:parsedec3new}
\begin{split}
\limsup_{n, n'\to\infty}&\int_{|\blambda|\ge \Lambda}| \widehat{\rho_t^{\delta_n}}(\blambda)-\widehat{\rho_t^{\delta_{n'}}}(\blambda)|^2\,\de\blambda\\
&\le 4\sup_{n\ge 1}\int_{|\blambda|\ge \Lambda}| \widehat{\rho_t^{\delta_n}}(\blambda)|^2\,\de\blambda\\
&\le \frac{4}{\Lambda^2}\sup_{n\ge 1}\int_{|\blambda|\ge \Lambda}|\blambda|^2\,| \widehat{\rho_t^{\delta_n}}(\blambda)|^2\,\de\blambda\\
&\le \frac{4}{\Lambda^2}\sup_{n\ge 1}\int_{\mathbb R^d}|\blambda|^2\,| \widehat{\rho_t^{\delta_n}}(\blambda)|^2\,\de\blambda\\
&= \frac{4}{\Lambda^2}\sup_{n\ge 1}\int_{\mathbb R^d}| \widehat{\nabla\rho_t^{\delta_n}}(\blambda)|^2\,\de\blambda\\
&= \frac{4}{\Lambda^2}\sup_{n\ge 1}\int_{\Omega}| \nabla  \rho_t^{\delta_n}(\bx)|^2\,\de\bx,\\
\end{split}
\end{equation}
where in the last equality we have applied again Parseval's theorem. 
In the proof of Lemma \ref{lemma:convL4}, we provide an upper bound, which does not depend on $n$, on the Sobolev norm of $\rho_t^{\delta_n}$ (see \eqref{d-rho-11}). Thus, as $\Lambda\to \infty$, the second term of \eqref{eq:parsedec2new} converges to $0$.}

\comm{By iterating the argument $T/T_0$ times, we obtain that $(\rho_t^{\delta_n})_{n \ge 1}$ is a Cauchy sequence in $\Lp^2(\Omega)$ for $t\in [0, T]$. Let $\tilde{\rho}_t^{\infty} \in \Lp^2(\Omega)$ be its limit. Furthermore, by Lemma \ref{lemma:rcseq}, $(\rho^{\delta_n})_{n \ge 1}$ has limit $\rho^{\infty}$ in $\Cont([0,T],\cuP_2(\Omega))$. Therefore, the measure $\rho^{\infty}_t$ has for almost every $t\in[0, T]$ the density  $\tilde{\rho}^{\infty}_t \in \Lp^2(\Omega)$, and the proof is complete.}
\end{proof}

\comm{Theorem \ref{TH:CONVPME} follows from Lemma \ref{lemma:unique}, Lemma \ref{lemma:convuniquePDE} and Lemma \ref{lemma:newL2}.}

Let us define the free energy associated to the PDE \eqref{eq:PMEbis} as 
\begin{equation}\label{eq:freeen}
\begin{split}
F(\rho) &= \frac{1}{2}\, R(\rho) -\tau\, S(\rho)\\
& = \frac{\nu_0}{2}\|f- \rho\|_{\Lp^2(\Omega)}^2+\tau\int \rho(\bx, t) \log\rho(\bx, t)\, \de\bx\, .
\end{split}
\end{equation}
As explained in Section~\ref{subsec:PDEdelta0}, this limit free energy is displacement convex, and hence its $W_2$ gradient flow converges to the unique minimizer of \eqref{eq:freeen}. These facts are stated and proved formally in the theorem that follows. 

\begin{theorem}\label{th:convminfree}
Assume that the initial condition $\rho^\infty(0)\in \Cont^{\infty}(\Omega)$. Then, the following results hold:
\begin{enumerate}
\item There exists a unique minimizer in $\cuP_2(\Omega)$, call it $\rho^*$, of the free energy $F$ defined in  \eqref{eq:freeen}.

\item For any $t\ge 0$, we have 
\begin{equation}\label{eq:convfe1}
F(\rho^\infty(t)) - F(\rho^*) \le (F(\rho^\infty(0)) - F(\rho^*))e^{-2\alpha\,t},
\end{equation}
where $\alpha$ is defined in \eqref{eq:concavityf}. 

\item For any $n\ge 1$ and for almost any $t\ge 0$, we have 
\begin{equation}\label{eq:convfe2}
F(\rho^{\delta}(t)) - F(\rho^*) \le (F(\rho^\infty(0)) - F(\rho^*))e^{-2\alpha\,t} + \Delta(\delta, T, d),
\end{equation}
where $\alpha$ is defined in \eqref{eq:concavityf} and $\Delta(\delta, T, d)\to 0$ as $\delta\to 0$.
\end{enumerate}  
\end{theorem}

\begin{proof}
The proof follows from the results of \cite{carrillo2001entropy}. The technical assumptions required by \cite{carrillo2001entropy} are satisfied by the PDE \eqref{eq:PMEbis}, since $\Omega$ is convex and bounded, the initial condition $\rho^\infty(0)\in \Lp^{\infty}(\Omega)$, and $f$ satisfies the assumptions {\sf (A2)} and {\sf (A3)}. Note also that the condition $\inf_\Omega V = 0$ coming from assumption (HV3) of \cite{carrillo2001entropy} can be relaxed. In fact, adding a constant to $V$ does not change the entropy functional in \cite[Eq. (3)]{carrillo2001entropy} (which corresponds to the free energy \eqref{eq:freeen}) and the PDE in \cite[Eq. (46)]{carrillo2001entropy} (which corresponds to the PDE \eqref{eq:PMEbis}).

The uniqueness of the minimizer $\rho^*$ follows from \cite[Lemma 6]{carrillo2001entropy}, which proves the first result. Since $\rho^\infty$ is the unique weak solution of the PDE \eqref{eq:PMEbis} with initial and boundary conditions \eqref{eq:PME-BCbis}, then it coincides with the unique, non-negative mass-preserving solution of \cite[Theorem 16]{carrillo2001entropy}. Thus, the inequality \eqref{eq:convfe1} readily follows from \cite[Theorem 16]{carrillo2001entropy}.

\comm{It remains to prove inequality \eqref{eq:convfe2}. By definition of free energy, we obtain
\begin{equation}
F(\rho^{\delta}(t)) - F(\rho^{\infty}(t)) = \frac{1}{2}(R(\rho^{\delta}(t)) - R(\rho^{\infty}(t))) - \tau (S(\rho^{\delta}(t)) - S(\rho^{\infty}(t))).
\end{equation}
Recall that, by Lemma \ref{lemma:newL2}, $\rho^\delta(t)$ converges to $\rho^\infty(t)$ in $\Lp^2(\Omega)$. Consequently, by using the triangle inequality, we have that the term $R(\rho^{\delta}(t)) - R(\rho^{\infty}(t))$ tends to $0$ as $\delta \to 0$.}

\comm{In order to complete the proof, it remains to show that $S(\rho^{\delta}(t)) - S(\rho^{\infty}(t))$ tends to $0$ as $\delta \to 0$. To do so, define
\begin{equation}
\begin{split}
A &= \{\bx \in\Omega\,\,: \,\, |\rho^\delta(\bx, t)-\rho^\infty(\bx, t)|> 1/4\},\\ 
B &= \{\bx \in\Omega\,\,: \,\, \rho^\delta(\bx, t)\in[0, 1/2], \,\rho^\infty(\bx, t)\in[0, 1/2]\},\\ 
C &= \{\bx \in\Omega\,\,: \,\, \rho^\delta(\bx, t)>1/4, \,\rho^\infty(\bx, t)>1/4\}.\\
\end{split}
\end{equation}
Note that $A\cup B\cup C = \Omega$. In fact, suppose that $\bx\not\in B$ and $\bx\not \in C$. Then, one between $\rho^\delta(\bx, t)$ and $\rho^\infty(\bx, t)$ is $\in [0, 1/4]$ and the other is $> 1/2$. Consequently, $|\rho^\delta(\bx, t)-\rho^\infty(\bx, t)|> 1/4$ and $\bx\in A$. This immediately implies that 
\begin{equation}\label{eq:intABC}
\begin{split}
|S(\rho^{\delta}(t)) - S(\rho^{\infty}(t)) |&\le \left|\int_A \left(\rho^\delta(\bx, t)\log\rho^\delta(\bx, t)-\rho^\infty(\bx, t)\log\rho^\infty(\bx, t)\right)\,\de\bx\right|\\
&+ \left|\int_B \left(\rho^\delta(\bx, t)\log\rho^\delta(\bx, t)-\rho^\infty(\bx, t)\log\rho^\infty(\bx, t)\right)\,\de\bx\right|\\
&+ \left|\int_C \left(\rho^\delta(\bx, t)\log\rho^\delta(\bx, t)-\rho^\infty(\bx, t)\log\rho^\infty(\bx, t)\right)\,\de\bx\right|.\\
\end{split}
\end{equation}
We will now upper bound the three integrals in the RHS of \eqref{eq:intABC}. As for the first term, note that
\begin{equation}\label{eq:cardAbis}
\int_\Omega |\rho^\delta(\bx, t)-\rho^\infty(\bx, t)|^2\,\de\bx \ge \int_A |\rho^\delta(\bx, t)-\rho^\infty(\bx, t)|^2\,\de\bx\ge \frac{|A|}{16},
\end{equation}
where $|A|$ denotes the volume of $A$. Furthermore,
\begin{equation}\label{eq:splitintA}
\begin{split}
&\left|\int_A \left(\rho^\delta(\bx, t)\log\rho^\delta(\bx, t)-\rho^\infty(\bx, t)\log\rho^\infty(\bx, t)\right)\,\de\bx\right|\\
&\hspace{3em}\le \left|\int_A \rho^\delta(\bx, t)\log\rho^\delta(\bx, t)\,\de\bx\right|+\left|\int_A \rho^\infty(\bx, t)\log\rho^\infty(\bx, t)\,\de\bx\right|\\
&\hspace{3em}\le |A|^{1/2}\left(\int_A \left(\rho^\delta(\bx, t)\log\rho^\delta(\bx, t)\right)^2\,\de\bx\right)^{1/2}\\
&\hspace{3em}+|A|^{1/2}\left(\int_A \left(\rho^\infty(\bx, t)\log\rho^\infty(\bx, t)\right)^2\,\de\bx\right)^{1/2}. 
\end{split}
\end{equation}
Note that $|t\log t|\le 1$ for $t\in [0, 1]$ and $|\log t|\le t$ for $t\ge 1$. Thus, the RHS of \eqref{eq:splitintA} is upper bounded by 
\begin{equation}
|A|^{1/2}\left(|A|+\|\rho^\delta(t)\|_{\Lp^4(\Omega)}\right)^{1/2}+|A|^{1/2}\left(|A|+\|\rho^\infty(t)\|_{\Lp^4(\Omega)}\right)^{1/2}.
\end{equation}
By Lemma \ref{lemma:newL2}, for almost all $t\in[0, T]$, $\rho^\delta(t)$ converges to $\rho^\infty(t)$ in $\Lp^2(\Omega)$. Thus, by \eqref{eq:cardAbis}, $|A|$ tends to $0$ as $\delta\to 0$. By Lemma \ref{lemma:convuniquePDE}, $\rho^\infty(t)\in \Lp^4(\Omega)$ for almost all $t\in[0, T]$. Furthermore, by Lemma \ref{lemma:convL4}, the quantity $\|\rho^\delta(t)\|_{\Lp^4(\Omega)}$ has a $\delta$-free upper bound for $t\in [0, T_0]$. As a result, for almost all $t\in [0, T_0]$, the first integral in \eqref{eq:intABC} tends to $0$ as $\delta\to 0$. By iterating this argument $T/T_0$ times, we conclude that for almost all $t\in [0, T_0]$, the first integral in \eqref{eq:intABC} tends to $0$ as $\delta\to 0$.}

\comm{In order to bound the second integral in \eqref{eq:intABC}, we write
\begin{equation}\label{eq:splitintB}
\begin{split}
&\left|\int_B \left(\rho^\delta(\bx, t)\log\rho^\delta(\bx, t)-\rho^\infty(\bx, t)\log\rho^\infty(\bx, t)\right)\,\de\bx\right| \\
&\hspace{3em}\le \int_B \left|\rho^\delta(\bx, t)\log\rho^\delta(\bx, t)-\rho^\infty(\bx, t)\log\rho^\infty(\bx, t)\right|\,\de\bx\\
&\hspace{3em}\le \int_B \left|\rho^\delta(\bx, t)-\rho^\infty(\bx, t)\right|\log\frac{1}{\left|\rho^\delta(\bx, t)-\rho^\infty(\bx, t)\right|}\,\de\bx,\\
\end{split}
\end{equation}
where in the last inequality we have applied \cite[Theorem 17.3.3]{cover2006elements}, since $\rho^\delta(\bx, t)$, $\rho^\infty(\bx, t)\in[0, 1/2]$ by definition of $B$. Note that
$$|\log t|\le \max\left(2\sqrt{t}, \frac{1}{t}\right)\le 2\sqrt{t}+\frac{1}{t}.$$
Thus, the RHS of \eqref{eq:splitintB} is upper bounded by
\begin{equation}\label{eq:splitintB2}
\begin{split}
\int_B &\left|\rho^\delta(\bx, t)-\rho^\infty(\bx, t)\right|^2\,\de\bx+2\int_B \left|\rho^\delta(\bx, t)-\rho^\infty(\bx, t)\right|^{1/2}\,\de\bx\\
&\|\rho^\delta(t)-\rho^\infty(t)\|_{\Lp^2(\Omega)}^2+2|\Omega|^{1/2} \|\rho^\delta(t)-\rho^\infty(t)\|_{\Lp^1(\Omega)}^{1/2},
\end{split}
\end{equation}
where in the last step we have used Cauchy-Schwarz inequality. By Lemma \ref{lemma:newL2}, for almost all $t\in[0, T]$, $\rho^\delta(t)$ converges to $\rho^\infty(t)$ in $\Lp^2(\Omega)$. As a result, the second integral in \eqref{eq:intABC} also tends to $0$ as $\delta\to 0$.}

\comm{Finally, let us bound the third integral in \eqref{eq:intABC}. Define $h(x)=x\log x$. Then, for $x>1/4$, 
\begin{equation}
|h'(x)| \le 1+|\log x|\le 1+\log 4 + x.
\end{equation}
Thus,
\begin{equation}
\begin{split}
&\left|\int_C \left(\rho^\delta(\bx, t)\log\rho^\delta(\bx, t)-\rho^\infty(\bx, t)\log\rho^\infty(\bx, t)\right)\,\de\bx\right|\\
&\hspace{3em}\le  \int_C \left|\rho^\delta(\bx, t)\log\rho^\delta(\bx, t)-\rho^\infty(\bx, t)\log\rho^\infty(\bx, t)\right|\,\de\bx\\
&\hspace{3em}\le  \int_C \left|\rho^\delta(\bx, t)-\rho^\infty(\bx, t)\right|\cdot (1+\log 4 + \rho^\delta(\bx, t)+\rho^\infty(\bx, t))\,\de\bx\\
&\hspace{3em}\le  (1+\log 4)\|\rho^\delta(t)-\rho^\infty(t)\|_{\Lp^1(\Omega)} + \|(\rho^\delta(t))^2-(\rho^\infty(t))^2\|_{\Lp^1(\Omega)}\\
&\hspace{3em}\le  (1+\log 4)\|\rho^\delta(t)-\rho^\infty(t)\|_{\Lp^1(\Omega)} \\
&\hspace{3em}+ \|\rho^\delta(t)-\rho^\infty(t)\|_{\Lp^2(\Omega)}\cdot \|\rho^\delta(t)+\rho^\infty(t)\|_{\Lp^2(\Omega)},
\end{split}
\end{equation}
where in the last step we have used Cauchy-Schwarz inequality. By Lemma \ref{lemma:newL2}, for almost all $t\in[0, T]$, $\rho^\delta(t)$ converges to $\rho^\infty(t)$ in $\Lp^2(\Omega)$. By Lemma \ref{lemma:convuniquePDE}, $\rho^\infty(t)\in \Lp^2(\Omega)$ for almost all $t\in[0, T]$. Furthermore, by Lemma \ref{lemma:convL4}, the quantity $\|\rho^\delta(t)\|_{\Lp^2(\Omega)}$ has a $\delta$-free upper bound for $t\in [0, T_0]$. As a result, for almost all $t\in [0, T_0]$, the third integral in \eqref{eq:intABC} tends to $0$ as $\delta\to 0$. By iterating this argument $T/T_0$ times, we conclude that for almost all $t\in [0, T_0]$, the third integral in \eqref{eq:intABC} tends to $0$ as $\delta\to 0$, and the proof is complete.} \end{proof}

At this point, we are ready to provide the proof of Theorem \ref{TH:ENDTOEND}.

\begin{proof}[Proof of Theorem \ref{TH:ENDTOEND}]
By substituting $z$ with $z^{1/2p}$ in Theorem \ref{TH:CONVPDE}, we have that with probability at least $1-1/z$
\begin{equation}\label{eq:end1}
R_N(\bw^k)\le R^\delta(\rho^\delta_{k\eps}) + z^{1/2p}\,\err(N,d,\eps,\delta) \,\,e^{C_*p\delta^{-(d+2)}\, T},
\end{equation}
where $\err(N,d,\eps,\delta)$ is defined in \eqref{eq:errNdz}. The risk $R^\delta(\rho^\delta_{k\eps})$ can be upper bounded as 
\begin{equation}\label{eq:end2}
\begin{split}
R^\delta(\rho^\delta_{k\eps}) &= \nu_0\,\|f-K^\delta \ast \rho^\delta_{k\eps}\|_{\Lp^2(\Omega)}^2 \\
&\le \nu_0\,\left(\|f- \rho^\delta_{k\eps}\|_{\Lp^2(\Omega)} + \|K^\delta \ast \rho^\delta_{k\eps}- \rho^\delta_{k\eps}\|_{\Lp^2(\Omega)}\right)^2 \\
&= R(\rho^\delta_{k\eps}) + \Delta_0(\delta, T, d),
\end{split}
\end{equation}
\comm{where $\Delta_0(\delta, T, d)\to 0$ as  $\delta\to 0$, since both $K^\delta \ast \rho^\delta_{t}$ and $\rho^\delta_{t}$ converge in $\Lp^2(\Omega)$ to $\rho^\infty_t$.} Furthermore, by Theorem \ref{th:convminfree}, 
\begin{equation}\label{eq:end3}
\begin{split}
&R(\rho^\delta_{k\eps}) = 2\,F(\rho^\delta_{k\eps})+2\tau S(\rho^\delta_{k\eps}) \le 2\,F(\rho^\delta_{k\eps}) + 2\tau \log|\Omega|\\
&\le 2\, F(\rho^*) +2\, (F(\rho^\infty(0)) - F(\rho^*))e^{-2\alpha k \eps}+ 2\tau \log|\Omega|+\Delta(\delta, T, d)\\
&= 2\, F(\rho^\infty(0)) e^{-2\alpha k \eps}+2\,(1-e^{-2\alpha k \eps})\, F(\rho^*) + 2\tau \log|\Omega| +\Delta(\delta, T, d),
\end{split}
\end{equation}
where $\Delta(\delta, T, d)\to 0$ as $\delta\to 0$ and we recall that $|\Omega|$ denotes the volume of the set $\Omega$.

Note that 
\begin{equation}\label{eq:ubfrho}
F(\rho^*)\le F(f) =  -\tau S(f),
\end{equation}
since $\rho^*$ is the minimizer of $F$. \comm{By combining \eqref{eq:ubfrho} with \eqref{eq:end3}, we deduce that
\begin{equation}\label{eq:end4}
\begin{split}
&R(\rho^\delta_{k\eps})\le 2\, F(\rho^\infty(0)) e^{-2\alpha k \eps}+ 2\tau \left(\log|\Omega| -(1-e^{-2\alpha k\epsilon})S(f)\right)+\Delta(\delta, T, d)\\
&=R(\rho^\infty(0)) e^{-2\alpha k \eps}+ 2\tau \left(\log|\Omega| -(1-e^{-2\alpha k\epsilon})S(f)-S(\rho^\infty(0))\,e^{-2\alpha k\epsilon}\right) \\
&\hspace{27.5em}+\Delta(\delta, T, d)\\
&\le R_N(\bw^0) e^{-2\alpha k \eps}+ 2\tau \left(\log|\Omega| -(1-e^{-2\alpha k\epsilon})S(f)-S(\rho^\infty(0))\,e^{-2\alpha k\epsilon}\right) \\
&\phantom{AAAAAAAA}+ z\,\err(N,d,\eps,\delta) \,e^{C_*p\delta^{-(d+2)}\, T} \,e^{-2\alpha k\epsilon}+\Delta(\delta, T, d),\\
\end{split}
\end{equation}
where in the last step we use again the result of Theorem \ref{TH:CONVPDE} and the fact that $R(\rho^\infty(0))-R^\delta(\rho^\infty(0))$ tends to $0$ as $\delta\to 0$.}

By optimizing over $p$ in \eqref{eq:end1}, we will set $\Delta_1(N, \eps, T, d, z)$ as in \eqref{eq:delta111}. We also let $\Delta_2(\delta, T, d) = \Delta_0(\delta, T, d) + \Delta(\delta, T, d)$. Then, the result follows by combining \eqref{eq:end1}, \eqref{eq:end2} and \eqref{eq:end4}. 
\end{proof}

\section{Heat kernel in bounded domains with Neumann boundary}
\label{app:Kernel}

Given the domain $D\subseteq\reals^d$ (compact, with $\Cont^2$
boundary $\partial D$), we denote by $G^{D}(\bx,\by;t)$ the
associated heat kernel, with Neumann boundary conditions.
We collect here a few well known facts about this kernel (see, e.g., \cite[Section 6.1]{taylor2013partial}). 

The heat kernel can be defined as a function $G^{D}:D\times D\times \reals_{>0}$ satisfying
\begin{align}
\partial_t G^{D}(\bx,\by;t)& = \Delta_{\by} G^{D}(\bx,\by;t)\, ,\label{eq:defheatbd1}\\
\<\nabla_{\by} G^{D}(\bx,\by;t),\bn(\by)\> & = 0\;\;\;\; \forall \by\in\partial D\, ,\label{eq:defheatbd2}\\
G^{ D}(\bx,\,\cdot\; ;t) &\Rightarrow \delta_{\bx} \, ,\;\;\; \mbox{ as $t\to 0$, $\bx\in  D^{\circ}$.}\label{eq:defheatbd3}
\end{align}
We will also denote by $G(\bx,\by;t)$ the heat kernel on $\reals^{d}$, namely 
\begin{align}
G(\bx,\by;t) \equiv \frac{1}{(4\pi t)^{d/2}}\,\exp\left\{-\frac{\|\bx-\by\|_2^2}{4t}\right\}\, .
\end{align}

The probabilistic interpretation of $G^{ D}$ is as follows (see, e.g., \cite{bakry2013analysis}). Let $\E_{\bx}$ denote expectation with respect to a Brownian motion $\bX_t$,
with initial condition $\bX_0=\bx$, and reflected at $\partial D$ (see Section \ref{app:nonlinear} for definitions of this process, following \cite{tanaka1979stochastic}).
Then, for any bounded continuous function $\varphi: D\to\reals$, 
\begin{align}
\E_{\bx}\big\{\varphi(\bX_t)\big\} &= \int G^{ D}(\bx,\by;t)\, \varphi(\by)\, \de\by\\
& \equiv G^{ D}_{\varphi}(\bx;t)\, .
\end{align}
Finally, $G^{ D}$ can be viewed as the kernel representation of the bounded operator 
$e^{t\Delta/2}$ in $\Lp^2( D, \Unif)$. We  have 
\begin{align}
(e^{t\Delta/2}f)(\by) = \int f(\bx) \, G^{ D}(\bx,\by;t)\, \de\by\, .
\end{align}
Hence $G^{ D}(\bx,\by;t)$ can be represented in terms of the eigenfunctions $\phi_k$, and eigenvalues $\lambda_k$, 
 of $-\Delta$, 
\begin{align}
 G^{ D}(\bx,\by;t) = \sum_{k=0}^{\infty} e^{-\lambda_k t} \phi_k(\bx)\, \phi_k(\by)\, .
\end{align}
Here $0=\lambda_0<\lambda_1\le \lambda_2\le \dots$, with $\lim_{k\to\infty}\lambda_k = \infty$, and
$\phi_0(\bx) = \bfone_{ D}(\bx)/{\rm Vol}( D)^{1/2}$.

\begin{remark}
Since $\Delta$ is self-adjoint in $\Lp^2( D,\Unif)$, it follows that $G^{ D}$ is symmetric, namely 
$G^{ D}(\bx,\by,t) = G^{ D}(\by,\bx;t)$, and therefore it satisfies
\begin{align}
\partial_t G^{ D}(\bx,\by;t)& = \Delta_{\bx} G^{ D}(\bx,\by;t)\, ,\\
\<\nabla_{\bx} G^{ D}(\bx,\by;t),\bn(\bx)\> & = 0\;\;\;\; \forall \bx\in\partial D\, .
\end{align}
\end{remark}

\begin{theorem}\label{thm:Kernel}
The Neumann heat kernel satisfies the following properties:
\begin{enumerate}
\item  We have that
\begin{align}
G^{ D}(\bx,\by;t) = G(\bx,\by;t) + G_R^{ D}(\bx,\by;t)\,,\label{eq:HeatDecomposition}
\end{align}
where $G_R^{ D}\in \Cont^{\infty}( D\times  D\times\reals_{\ge})$.
\item For any $t>0$, $G^{ D}(\;\cdot\; ,\;\cdot\; ; t)\in
  \Cont^{\infty}( D\times D)$.
\item We have that, for a constant $C( D)$,
\begin{align}
\big\|\nabla G^ D(\bx,\,\cdot\,;t)\big\|_{\Lp^1( D)}\le \frac{C( D)}{\sqrt{t}}\, .
\end{align}
\end{enumerate}
\end{theorem}
\begin{proof}
Substituting $G^{ D}(\bx,\by;t) = G(\bx,\by;t) + G_R^{ D}(\bx,\by;t)$ into Eqs.~(\ref{eq:defheatbd1}) to (\ref{eq:defheatbd3})
yields, for $\bx\in D$,
\begin{align}
\partial_t G_R^{ D}(\bx,\by;t)& = \Delta_{\by} G_R^{ D}(\bx,\by;t)\, ,\label{eq:defheatbd1_b}\\
\<\nabla_{\by} G_R^{ D}(\bx,\by;t),\bn(\by)\> & = -\<\nabla_{\by} G(\bx,\by;t),\bn(\by)\> 
\;\;\;\; \forall \by\in\partial D\, ,\label{eq:defheatbd2_b}\\
G^{ D}_R(\bx,\by; ;0)& = 0\, ,\;\;\; \mbox{$\bx,\by \in  D^{\circ}$.}\label{eq:defheatbd3_b}
\end{align}
Thus $G_R$ satisfies the heat equation in  $ D\times [0,T]$ and hence $(\by,t)\mapsto G^{ D}_R(\bx,\by;t)$ is $\Cont^{\infty}$ inside this domain
(see, e.g., \cite[Chapter 2, Theorem 8]{evans2009partial}, which refers to Dirichlet boundary condition, but applies equally well to the Neumann case).
By symmetry, we have the claimed continuity in $(\bx,\by)$, thus proving point 1.

Claim 2 follows by the same decomposition.

Finally, claim 3 follows from Lemma 3.1 in \cite{wang2013gradient}.
\end{proof}

\section{Some useful technical lemmas}
\begin{lemma}[Displacement convexity of quadratic functionals]\label{lem:DisplCvx}
Let $U:\reals^d\to\reals^d$ be twice differentiable with $|U(\bx)|\le C(1+|\bx|^2)$, $U(\bx) = U(-\bx)$, and define $\cuU:\cuP_2(\reals^d)\to\reals$ by
$\cuU(\rho)\equiv \int U(\bx-\bx') \, \rho(\de\bx)\,\rho(\de\bx')$. Then $\cuU$ is displacement convex if and only if $U$ is convex.
\end{lemma}
\begin{proof}
Proposition 7.4 in \cite{santambrogio2015optimal} proves that convexity of $U$ implies displacement convexity of $\cuU$.
To prove the converse implication, let $\bx,\bdelta\in \reals^d$, $\bx\neq \bfzero$ and consider the two probability distributions
$\rho_0= (\delta_{\bfzero}+\delta_{\bx})/2$ and $\rho_1= (\delta_{\bfzero}+\delta_{\bx+\bdelta})/2$.
For $|\bdelta|<|\bx|$, the geodesic path connecting these distribution is $\rho_t= (\delta_{\bfzero}+\delta_{\bx+t\bdelta})/2$, $t\in [0,1]$.
Substituting in the definition of $\cuU$, we get
\begin{align}
\cuU(\rho_t) &= \frac{1}{2}\, U(\bfzero) + \frac{1}{2} U(\bx+t\bdelta)\\
& = \cuU(\rho_0) + \frac{t}{2}\<\nabla U(\bx),\bdelta\> +\frac{t^2}{4}\<\bdelta,\nabla^2 U(\bx)\bdelta\> +o(t^2)\, .
\end{align}
Hence, displacement convexity implies $\<\bdelta,\nabla^2 U(\bx)\bdelta\>\ge 0$. Since  this holds for all  $|\bdelta|<|\bx|$,
we obtain $\nabla^2 U(\bx)\succeq \bfzero$ for all $\bx\neq \bfzero$, which in turns imply that $U$ is convex (by a continuity argument,
it is sufficient to lower bound the Hessian everywhere except at a point).
\end{proof}

\begin{lemma}[A Gronwall type inequality~\cite{bihari1956generalization}]\label{lem:gronwall}
Let $u:[0,T] \to \reals_+$ be a continuous function that satisfies the inequality
\[
u(t) \le A + \int_0^t \Psi(s)\omega(u(s)) \de s\,, \quad t\in [0,T]\,,
\]
where $A\ge 0$, $\Psi:[0,T] \to \reals_+$ is continuous and $\omega:\reals_+\to \reals_+$ is continuous and monotone-increasing. Then, the following holds 
\[
u(t) \le \Phi^{-1}\left(\Phi(A) + \int_0^t \Psi(s) \de s \right)\,, \quad t\in[0,T]\,,
\]
with $\Phi:\reals\mapsto \reals$ given by
\[
\Phi(u) \equiv \int_{u_0}^u \frac{\de s}{\omega(s)}\,, \quad u\in \reals, \quad u_0 \equiv \omega(A)\,.
\]
\end{lemma}

\begin{remark}\label{rem:gronwall}
To derive Equation~\eqref{d-rho-10}, we use Lemma~\ref{lem:gronwall} with $\omega(u) = u^2$, $\Psi(s) = \bar{C}_m C_3$, $A = \bar{C}_mC_4+\bar{C}_m \tau M^2C_ast T_n$. 
\end{remark}

\begin{lemma}[Gagliardo-Nirenberg interpolation inequality, cf. Theorem 1.5.2 of \cite{cherrierlinear}]\label{lem:GN-interpolation} Fix $1\le q, r \le \infty$ and $m$ a positive integer. 
Let $u\in \Lp^q(\Omega)\cap \Lp^r(\Omega)$ and $\nabla^{\otimes m} u\in \Lp^p(\Omega)$.
For integer $j$, $0\le j\le m$, and $\theta\in [j/m,1]$ (with the exception $\theta\neq 1$ if $m-j-d/2$ is a non-negative integer), define $p$ by 
\[
\frac{1}{p} = \frac{j}{d} + \theta \left(\frac{1}{r} - \frac{m}{d}\right) + \frac{1-\theta}{q}\,.
\]
Then $\nabla^{\otimes j} u \in \Lp^p(\Omega)$ and satisfies
\[
\|\nabla^{\otimes j} u\|_p \le C \|\nabla^{\otimes m} u\|_r^\theta\, \|u\|_q^{1-\theta} + C_1 \|u\|_s\,.
\]
with finite arbitrary $1\le s\le \max(r,q)$ and $C>0$ and $C_1\ge 0$ are independent of $u$. The constant $C$ is independent of $\Omega$, while $C_1\to 0$ as $|\Omega| \to \infty$.
In particular, the choice $C_1 = 0$ is admissible if $\Omega = \reals^d$.
\end{lemma}

\bibliographystyle{amsalpha} 
\bibliography{all-bibliography}

\newcommand{\etalchar}[1]{$^{#1}$}
\providecommand{\bysame}{\leavevmode\hbox to3em{\hrulefill}\thinspace}
\providecommand{\MR}{\relax\ifhmode\unskip\space\fi MR }
\providecommand{\MRhref}[2]{%
  \href{http://www.ams.org/mathscinet-getitem?mr=#1}{#2}
}
\providecommand{\href}[2]{#2}
\begin{thebibliography}{MBM{\etalchar{+}}18}

\bibitem[AB09]{anthony2009neural}
Martin Anthony and Peter~L. Bartlett, \emph{Neural network learning:
  Theoretical foundations}, Cambridge University Press, 2009.

\bibitem[AGS08]{ambrosio2008gradient}
Luigi Ambrosio, Nicola Gigli, and Giuseppe Savar{\'e}, \emph{Gradient flows: in
  metric spaces and in the space of probability measures}, Springer Science \&
  Business Media, 2008.

\bibitem[Bac17]{bach2017breaking}
Francis Bach, \emph{Breaking the curse of dimensionality with convex neural
  networks}, The Journal of Machine Learning Research \textbf{18} (2017),
  no.~1, 629--681.

\bibitem[Bar93]{barron1993universal}
Andrew~R. Barron, \emph{Universal approximation bounds for superpositions of a
  sigmoidal function}, IEEE Transactions on Information theory \textbf{39}
  (1993), no.~3, 930--945.

\bibitem[Bar98]{bartlett1998sample}
Peter~L. Bartlett, \emph{The sample complexity of pattern classification with
  neural networks: the size of the weights is more important than the size of
  the network}, IEEE Transactions on Information Theory \textbf{44} (1998),
  no.~2, 525--536.

\bibitem[BGL13]{bakry2013analysis}
Dominique Bakry, Ivan Gentil, and Michel Ledoux, \emph{Analysis and geometry of
  markov diffusion operators}, vol. 348, Springer Science \& Business Media,
  2013.

\bibitem[Bih56]{bihari1956generalization}
Imre Bihari, \emph{A generalization of a lemma of {B}ellman and its application
  to uniqueness problems of differential equations}, Acta Mathematica Hungarica
  \textbf{7} (1956), no.~1, 81--94.

\bibitem[BJW18]{bakshi2018learning}
Ainesh Bakshi, Rajesh Jayaram, and David~P Woodruff, \emph{Learning two layer
  rectified neural networks in polynomial time}, {\sf arXiv:1811.01885} (2018).

\bibitem[BRV{\etalchar{+}}06]{bengio2006convex}
Yoshua Bengio, Nicolas~L. Roux, Pascal Vincent, Olivier Delalleau, and Patrice
  Marcotte, \emph{Convex neural networks}, Advances in Neural Information
  Processing Systems, 2006, pp.~123--130.

\bibitem[BY03]{buhlmann2003boosting}
Peter B{\"u}hlmann and Bin Yu, \emph{{Boosting with the L2 loss: regression and
  classification}}, Journal of the American Statistical Association \textbf{98}
  (2003), no.~462, 324--339.

\bibitem[CB18]{chizat2018global}
Lenaic Chizat and Francis Bach, \emph{On the global convergence of gradient
  descent for over-parameterized models using optimal transport}, Advances in
  Neural Information Processing Systems, 2018.

\bibitem[CFG14]{candes2014towards}
Emmanuel~J. Cand{\`e}s and Carlos Fernandez-Granda, \emph{Towards a
  mathematical theory of super-resolution}, Communications on Pure and Applied
  Mathematics \textbf{67} (2014), no.~6, 906--956.

\bibitem[CJM{\etalchar{+}}01]{carrillo2001entropy}
Jos{\'e}~A. Carrillo, Ansgar J{\"u}ngel, Peter~A. Markowich, Giuseppe Toscani,
  and Andreas Unterreiter, \emph{Entropy dissipation methods for degenerate
  parabolicproblems and generalized sobolev inequalities}, Monatshefte f{\"u}r
  Mathematik \textbf{133} (2001), no.~1, 1--82.

\bibitem[CM12]{cherrierlinear}
Pascal Cherrier and Albert Milani, \emph{Linear and quasi-linear evolution
  equations in {H}ilbert spaces}, Graduate studies in mathematics, American
  Mathematical Soc., 2012.

\bibitem[CMV03]{carrillo2003kinetic}
Jos{\'e}~A. Carrillo, Robert~J. McCann, and C{\'e}dric Villani, \emph{Kinetic
  equilibration rates for granular media and related equations: entropy
  dissipation and mass transportation estimates}, Revista Matematica
  Iberoamericana \textbf{19} (2003), no.~3, 971--1018.

\bibitem[CMV06]{carrillo2006contractions}
\bysame, \emph{Contractions in the 2-{W}asserstein length space and
  thermalization of granular media}, Archive for Rational Mechanics and
  Analysis \textbf{179} (2006), no.~2, 217--263.

\bibitem[CS16]{chen2016generalized}
Yining Chen and Richard~J Samworth, \emph{Generalized additive and index models
  with shape constraints}, Journal of the Royal Statistical Society: Series B
  (Statistical Methodology) \textbf{78} (2016), no.~4, 729--754.

\bibitem[CST00]{cristianini2000introduction}
Nello Cristianini and John Shawe-Taylor, \emph{An introduction to support
  vector machines and other kernel-based learning methods}, Cambridge
  University Press, 2000.

\bibitem[CT06]{cover2006elements}
Thomas~M Cover and Joy~A Thomas, \emph{Elements of information theory}, John
  Wiley \& Sons, 2006.

\bibitem[Cyb89]{cybenko1989approximation}
George Cybenko, \emph{Approximation by superpositions of a sigmoidal function},
  Mathematics of control, signals and systems \textbf{2} (1989), no.~4,
  303--314.

\bibitem[Dob79]{dobrushin1979vlasov}
Roland~L’vovich Dobrushin, \emph{Vlasov equations}, Functional Analysis and
  Its Applications \textbf{13} (1979), no.~2, 115--123.

\bibitem[Don92]{donoho1992superresolution}
David~L. Donoho, \emph{Superresolution via sparsity constraints}, SIAM Journal
  on Mathematical Analysis \textbf{23} (1992), no.~5, 1309--1331.

\bibitem[DZPS18]{du2018gradient}
Simon~S Du, Xiyu Zhai, Barnabas Poczos, and Aarti Singh, \emph{Gradient descent
  provably optimizes over-parameterized neural networks}, {\sf
  arXiv:1810.02054} (2018).

\bibitem[Eva09]{evans2009partial}
Lawrence~C. Evans, \emph{Partial differential equations}, Springer, 2009.

\bibitem[FP08]{figalli2008convergence}
Alessio Figalli and Robert Philipowski, \emph{Convergence to the viscous porous
  medium equation and propagation of chaos}, ALEA Lat. Am. J. Probab. Math.
  Stat \textbf{4} (2008), 185--203.

\bibitem[Fri01]{friedman2001greedy}
Jerome~H. Friedman, \emph{Greedy function approximation: a gradient boosting
  machine}, Annals of Statistics (2001), 1189--1232.

\bibitem[HD13]{hannah2013multivariate}
Lauren~A Hannah and David~B Dunson, \emph{Multivariate convex regression with
  adaptive partitioning}, The Journal of Machine Learning Research \textbf{14}
  (2013), no.~1, 3261--3294.

\bibitem[LS84]{lions1984stochastic}
Pierre-Louis Lions and Alain-Sol Sznitman, \emph{Stochastic differential
  equations with reflecting boundary conditions}, Communications on Pure and
  Applied Mathematics \textbf{37} (1984), no.~4, 511--537.

\bibitem[LSU88]{ladyzhenskaia1988linear}
Olga~Aleksandrovna Ladyzhenskaia, Vsevolod~Alekseevich Solonnikov, and Nina~N.
  Ural'tseva, \emph{Linear and quasi-linear equations of parabolic type},
  vol.~23, American Mathematical Society, 1988.

\bibitem[LY17]{li2017convergence}
Yuanzhi Li and Yang Yuan, \emph{Convergence analysis of two-layer neural
  networks with relu activation}, Advances in Neural Information Processing
  Systems, 2017, pp.~597--607.

\bibitem[MBM{\etalchar{+}}18]{mei2018landscape}
Song Mei, Yu~Bai, Andrea Montanari, et~al., \emph{The landscape of empirical
  risk for nonconvex losses}, The Annals of Statistics \textbf{46} (2018),
  no.~6A, 2747--2774.

\bibitem[McC97]{mccann1997convexity}
Robert~J McCann, \emph{A convexity principle for interacting gases}, Advances
  in mathematics \textbf{128} (1997), no.~1, 153--179.

\bibitem[MMM19]{mei2019mean}
Song Mei, Theodor Misiakiewicz, and Andrea Montanari, \emph{Mean-field theory
  of two-layers neural networks: dimension-free bounds and kernel limit},
  Conference on Learning Theory (COLT), 2019.

\bibitem[MMN18]{mei2018mean}
Song Mei, Andrea Montanari, and Phan-Minh Nguyen, \emph{A mean field view of
  the landscape of two-layer neural networks}, Proceedings of the National
  Academy of Sciences (2018).

\bibitem[NS17]{nitanda2017stochastic}
Atsushi Nitanda and Taiji Suzuki, \emph{Stochastic particle gradient descent
  for infinite ensembles}, {\sf arXiv:1712.05438} (2017).

\bibitem[Oel01]{oelschlager2001sequence}
Karl Oelschl{\"a}ger, \emph{A sequence of integro-differential equations
  approximating a viscous porous medium equation}, Zeitschrift f{\"u}r Analysis
  und ihre Anwendungen \textbf{20} (2001), no.~1, 55--91.

\bibitem[Oel02]{oelschlager2002simulation}
\bysame, \emph{Simulation of the solution of a viscous porous medium equation
  by a particle method}, SIAM Journal on Numerical Analysis \textbf{40} (2002),
  no.~5, 1716--1762.

\bibitem[Phi07]{philipowski2007interacting}
Robert Philipowski, \emph{Interacting diffusions approximating the porous
  medium equation and propagation of chaos}, Stochastic Processes and their
  Applications \textbf{117} (2007), no.~4, 526--538.

\bibitem[PS91]{park1991universal}
Jooyoung Park and Irwin~W. Sandberg, \emph{Universal approximation using
  radial-basis-function networks}, Neural computation \textbf{3} (1991), no.~2,
  246--257.

\bibitem[Ros62]{rosenblatt1962principles}
Frank Rosenblatt, \emph{Principles of neurodynamics}, Spartan Book, 1962.

\bibitem[RR08]{rahimi2008random}
Ali Rahimi and Benjamin Recht, \emph{Random features for large-scale kernel
  machines}, Advances in neural information processing systems, 2008,
  pp.~1177--1184.

\bibitem[RVE18]{rotskoff2018neural}
Grant~M. Rotskoff and Eric Vanden-Eijnden, \emph{Neural networks as interacting
  particle systems: Asymptotic convexity of the loss landscape and universal
  scaling of the approximation error}, Advances in Neural Information
  Processing Systems, 2018.

\bibitem[RW94]{rogers1994diffusions}
L.~Chris~G. Rogers and David Williams, \emph{{Diffusions, Markov processes and
  martingales: Volume 2, It{\^o} calculus}}, vol.~2, Cambridge university
  press, 1994.

\bibitem[San15]{santambrogio2015optimal}
Filippo Santambrogio, \emph{Optimal transport for applied mathematicians:
  Calculus of variations, {PDE}s, and modeling}, vol.~87, Birkh{\"a}user, 2015.

\bibitem[Sch03]{schapire2003boosting}
Robert~E. Schapire, \emph{The boosting approach to machine learning: An
  overview}, Nonlinear estimation and classification, Springer, 2003,
  pp.~149--171.

\bibitem[SJL18]{soltanolkotabi2018theoretical}
Mahdi Soltanolkotabi, Adel Javanmard, and Jason~D. Lee, \emph{Theoretical
  insights into the optimization landscape of over-parameterized shallow neural
  networks}, IEEE Transactions on Information Theory (2018).

\bibitem[{Slo}94]{slominski1994approximation}
{Slomi{\'n}ski, Leszek}, \emph{{On approximation of solutions of
  multidimensional SDE's with reflecting boundary conditions}}, Stochastic
  processes and their Applications \textbf{50} (1994), no.~2, 197--219.

\bibitem[{Slo}01]{slominski2001euler}
\bysame, \emph{{Euler's approximations of solutions of SDEs with reflecting
  boundary}}, Stochastic processes and their applications \textbf{94} (2001),
  no.~2, 317--337.

\bibitem[SS18]{sirignano2018mean}
Justin Sirignano and Konstantinos Spiliopoulos, \emph{Mean field analysis of
  neural networks}, {\sf arXiv:1805.01053} (2018).

\bibitem[Szn91]{sznitman1991topics}
Alain-Sol Sznitman, \emph{Topics in propagation of chaos}, Ecole d'{\'e}t{\'e}
  de probabilit{\'e}s de Saint-Flour XIX---1989, Springer, 1991, pp.~165--251.

\bibitem[Tan79]{tanaka1979stochastic}
Hiroshi Tanaka, \emph{Stochastic differential equations with reflecting
  boundary condition in convex regions}, Hiroshima Mathematical Journal
  \textbf{9} (1979), no.~1, 163--177.

\bibitem[Tay13]{taylor2013partial}
Michael Taylor, \emph{{Partial differential equations I: Basic theory}}, vol.
  115, Springer Science \& Business Media, 2013.

\bibitem[Tho13]{thomas2013numerical}
James~William Thomas, \emph{Numerical partial differential equations: finite
  difference methods}, vol.~22, Springer Science \& Business Media, 2013.

\bibitem[Tia17]{tian2017symmetry}
Yuandong Tian, \emph{Symmetry-breaking convergence analysis of certain
  two-layered neural networks with {ReLU} nonlinearity}, Workshop at
  International Conference on Learning Representation (ICLR), 2017.

\bibitem[V{\'a}z07]{vazquez2007porous}
Juan~Luis V{\'a}zquez, \emph{The porous medium equation: mathematical theory},
  Oxford University Press, 2007.

\bibitem[Vil08]{villani2008optimal}
C{\'e}dric Villani, \emph{Optimal transport: old and new}, vol. 338, Springer
  Science \& Business Media, 2008.

\bibitem[WLLM18]{wei2018margin}
Colin Wei, Jason~D Lee, Qiang Liu, and Tengyu Ma, \emph{On the margin theory of
  feedforward neural networks}, {\sf arXiv:1810.05369} (2018).

\bibitem[WY13]{wang2013gradient}
Feng-Yu Wang and Lixin Yan, \emph{Gradient estimate on convex domains and
  applications}, Proceedings of the American Mathematical Society \textbf{141}
  (2013), no.~3, 1067--1081.

\bibitem[ZSJ{\etalchar{+}}17]{zhong2017recovery}
Kai Zhong, Zhao Song, Prateek Jain, Peter~L. Bartlett, and Inderjit~S. Dhillon,
  \emph{Recovery guarantees for one-hidden-layer neural networks},
  International Conference on Machine Learning, 2017, pp.~4140--4149.

\end{thebibliography}

%

\end{document}